\documentclass[11pt,leqno]{article}
\usepackage{amsmath,amssymb,mathrsfs,amsthm,bm,ddag}

\usepackage[normalem]{ulem}
\usepackage{eucal}

\usepackage{ifpdf}
\usepackage[dvipdfmx]{graphicx}
\usepackage{xy}
\input{xy}
\xyoption{all}

\usepackage[hypertex]{hyperref}

\usepackage{threeparttable}

\usepackage{xspace}

\usepackage[geometry]{ifsym}

\usepackage{comment}

\makeatletter
\def\@seccntformat#1{\csname the#1\endcsname.\hspace{2ex}}

 \renewcommand{\subsection}%
  {\@startsection{subsection}%
  {2}%
  {\z@}%
  {2ex}
  {0ex}
  {\reset@font\normalsize\bfseries}}%

 \@addtoreset{equation}{subsection}

 \newcommand{\nsection}{\@startsection{section}{1}{\z@}%
     {-5ex}
     {1ex}
     {\reset@font\center\large\sc}}

 \renewenvironment{thebibliography}[1]
 {\nsection*{\refname\@mkboth{\refname}{\refname}}%
   \list{\@biblabel{\@arabic\c@enumiv}}%
   {\settowidth
   \labelwidth{\@biblabel{#1}}%
   \leftmargin
	\labelwidth
        \advance
	 \leftmargin
	 \labelsep
         \@openbib@code
         \usecounter{enumiv}%
         \let\p@enumiv\@empty
	 \parskip=0pt
	 \itemsep=1pt
	 \parsep=1pt
	 \itemindent=\z@
         \renewcommand\theenumiv{\@arabic\c@enumiv}}%
   	 \sloppy
   	 \clubpenalty4000
   	 \@clubpenalty\clubpenalty
   	 \widowpenalty4000%
   	 \footnotesize
   	 \sfcode`\.\@m}
  	 {\def\@noitemerr
    	 {\@latex@warning{Empty `thebibliography' environment}}%
   	 \endlist}

\makeatother

\newtheoremstyle{thm}
 {1em}
 {3pt}
 {\itshape}
 {}
 {\bf}
 {. ---}
 {0.5em}
 {}
\newtheoremstyle{dfn}
 {1em}
 {3pt}
 {}
 {}
 {\bf}
 {. {---}}
 {0.5em}
 {}

\swapnumbers
\theoremstyle{thm}

\newtheorem{lem}[subsection]{Lemma}
\newtheorem*{lem*}{Lemma}
\newtheorem{cor}[subsection]{Corollary}
\newtheorem*{cor*}{Corollary}
\newtheorem{prop}[subsection]{Proposition}
\newtheorem*{prop*}{Proposition}

\newtheorem*{conj*}{Conjecture}
\newtheorem*{thm*}{Theorem}

\theoremstyle{dfn}
\newtheorem{dfn}[subsection]{Definition}
\newtheorem*{dfn*}{Definition}

\newtheorem*{ex*}{Example}

\newtheorem*{rem*}{Remark}

\usepackage{color}
\newenvironment{meta}{
\noindent \color{red}
\sffamily[}{\upshape]}

\usepackage{framed}
\definecolor{shadecolor}{gray}{0.80}
\specialcomment{test}%
     {\begin{shaded}
       \noindent\textcolor{red}{{\large\bf Personal comment}}\\
       }%
      {\end{shaded}}%
\excludecomment{test} 

\newsavebox{\circlebox}
\savebox{\circlebox}{\fontencoding{OMS}\selectfont\char13}
\newlength{\circleboxwdht}

\newcommand{\cc}[1]{(\!(#1)\!)}

\newcommand{\dd}[1]{[\![#1]\!]}
\newcommand{\ul}[1]{\underline{#1}}
\newcommand{\Mdf}{\mr{Mdf}}
\newcommand{\sHom}{\mc{H}\mr{om}}
\newcommand{\Comp}{\mr{Fl\acute{E}t}}

\newcommand{\newton}{valuation\xspace}
\newcommand{\NP}{\mr{VP}}
\newcommand{\rat}[1]{\mr{R}(#1)}

\newcommand{\TXsum}{\mathop{\textstyle\sum}\nolimits}

\newcommand{\ally}{ally\xspace}
\newcommand{\good}{good\xspace}

\begin{document}
\title{Ramification theory from homotopical point of view, II}
\author{Tomoyuki Abe}
\date{}
\maketitle

\section*{Introduction}
This paper is a continuation from [Part I], and we will prove a theorem which we used in [Part I].
However, the method is almost independent from [Part I], and, for example, we do not use the language of $\infty$-categories at all.
In the last two sections, we use the terminologies of [Part I, \S4], but this section can be read independently from the other sections of [Part I].
The goal of this paper is to show the following theorem which we promised in [Part I, 5.7].
We keep the notations from [Part I, \S4], some of which we will recall in \S\ref{sect4}.

\begin{thm*}[Corollary \ref{mainresult}]
 Assume that $X$ and $S$ are noetherian affine $\mb{F}_p$-schemes {\normalfont(}not necessarily of finite type{\normalfont)},
 and assume we are given a morphism $X\rightarrow S$ of finite type.
 For any $F\in\ms{S}_d(S)(:=\Comp_d(X/S))$ and a finite ordered set $\mbf{h}$ of $\mc{O}_{X\times\Delta^m}(X\times\Delta^m)$,
 we can find a finite ordered set $\mbf{h}'$ of $\mc{O}_{X}(X)$ such that for a surjective
 $\Mdf$-sequence $\{\mbf{V}_i\}$ adapted to $(F;\mbf{h}\vee\mbf{h}')$,
 $F[\{\mbf{V}_i\}]_{\infty}$ is in $\ms{S}_{d-1}$.
\end{thm*}
By limit argument, we may and do assume that $S$ is of finite type over a field $k$ of characteristic $p$.
We may also reduce to the case where $X\rightarrow S$ is smooth.
If we are given a morphism $h\colon Y\rightarrow\mb{A}^1$, we denote $h^*\mc{L}$,
where $\mc{L}$ is the Artin-Schreier sheaf on $\mb{A}^1$ associated with a non-trivial additive character, by $\mc{L}(h)$.
The function $h$ is said to be {\em $S$-separable} if the smooth locus in $X$ of the induced morphism $X\rightarrow\mb{A}^1\times S$ is fiberwise dense over $S$.
A key to show the above theorem is a weak structure theorem of the Artin-Schreier sheaf $\mc{L}(h)$.
This states that after a suitable alteration of $S$, $\mc{L}(h)$ can be written, generically with respect to each fiber of $X\rightarrow S$,
as $\mc{L}(f/\sigma)\otimes\mc{G}$ where $f$ is a separable function on $X$, $\sigma$ is a function on $S$,
and $\mc{G}$ is a locally constant constructible module on $X$.
A benefit of writing in this way is that we can compute the nearby cycles of $\mc{L}(h)$ very explicitly,
at least generically with respect to each fiber.
In this introduction, we will focus on how to prove the weak structure theorem.
For a proof of this key result, we use the argument invented by Kedlaya used to prove semistable reduction theorem \cite{Ked4}
and the existence of good formal model \cite{Ked5}.

From now on, for simplicity, we fix a smooth morphism $X\rightarrow S$ between integral schemes of finite type over a field $k$ of characteristic $p$.
Since, for a function $f$, we have $\mc{L}(f^p-f)\cong\Lambda$, the weak structure theorem follows readily from the following (purely geometric) result:

\begin{thm*}[Theorem \ref{mainthmadm}]
 For any $h\in k(X)$, we can find an alteration $S'\rightarrow S$ over which $h$ is admissible.
 Here, $h$ is said to be {\em admissible} if, Zariski locally on $S$, there exists a fiberwise dense open subscheme $U\subset X$,
 an $S$-separable function $f$ on $U$, a regular function $\sigma\in\mc{O}_S\cap k(S)^{\times}$, and $g\in k(X)$
 such that $h|_U+(g^p-g)\in f/\sigma+\mc{O}_X$.
\end{thm*}
Let us outline the strategy of the proof.
Since $X\rightarrow S$ is assumed to be smooth, we may assume that there exists a local coordinate $\{x_i\}$ of $X$ over $S$.
It is not hard to reduce to the case where $h=h'/\tau$ where $h'\in\mc{O}_X$, $\tau\in\mc{O}_S$.
For simplicity, we assume $S=\mr{Spec}(k[s,t])$ in the following.
Formally, we may write $h=\sum_{\ul{k}\geq0}a_{\ul{k}}\ul{x}^{\ul{k}}$ where $a_{\ul{k}}\in k(S)$.

Let us see how we take $g$ in some examples.
First consider $h=\frac{s}{t}x$.
In this case, $h$ is not of the form $f/\sigma$ at $(s,t)=(0,0)$.
What we should do is to take the blowup $S'\rightarrow S$ at $(0,0)$.
Then we may write $h=ux$ or $h=x/u$ using the coordinate around the exceptional fiber, and these are the forms that we need.

Now, let us consider $h=\frac{1}{s}x^p+\frac{t}{s}x$.
In this case, $h$ is not separable over $t=0$.
What we will do in this case is to ``mollify'' $h$ by adding functions of the form $g-g^p$.
In this case, we take $g=\frac{1}{s^{1/p}}x$ (by extending $S$).
Then we have $h+(g-g^p)=\frac{t+s^{p-1/p}}{s}x$.
Now, this function can be handled as in the first example, by taking a suitable blowup to transform it into the desired form.
As this example shows, a basic strategy is to ``eliminate'' terms $a_{\ul{k}}\ul{x}^{\ul{k}}$ such that $p\mid\ul{k}$ ({\em i.e.}\ $p\mid k_i$ for any $i$) by adding $g-g^p$.
Such a $g$ is called a {\em mollifier}.

Finally, let us consider $h=\frac{1}{s}+\frac{1}{s}x+\frac{1}{s}x^2+\dots$.
Since $h$ is equal to $\frac{1}{s}\frac{1}{1-x}$, this is already admissible.
In other words, we have nothing to do.
However, if we take the philosophy of the 2nd example seriously, we could remove the factors of the form $\frac{1}{s}x^{pl}$ by adding some $g$.
Thus, we need to distinguish the case where we need further mollification and the case where we do not.

Keeping these basic examples in mind, our strategy, after Kedlaya, is to attain the admissibility ``locally'' in the Zariski-Riemann space of $S$,
and globalize by using the compactness of Zariski-Riemann space.
By writing valuation of height greater than $1$ by composition of valuation of height $1$,
it is not too hard to reduce the argument to the height $1$ case, so we only consider height $1$ valuation in this introduction.
To explain the strategy, let us fix a height $1$ valuation $v$ of $k(S)$, trivial on $k$.
Let $f=\sum a_{\ul{k}}\ul{x}^{\ul{k}}$.
Since $v$ is of height $1$, $v(k(S)^{\times})\subset\mb{R}$.
We may consider a number
\begin{equation*}
 \mr{VP}(f)(v):=\min_{\ul{k}\in\mb{N}^d\setminus\{\ul{0}\}}\bigl\{0,v(a_{\ul{k}})\bigr\}\leq0.
\end{equation*}
We view this as a function on $v$ later.
Now, we will perform mollifications, namely adding function of the form $g^p-g$, formally.
For this, put $f^{(m)}_{\ul{k}}:=\sum_{k=0}^{m}(a_{p^k\ul{k}})^{1/p^k}$.
Using this, we may define $\mr{VP}^{(\infty)}(f)(v)$.
This is, roughly saying, the limit $\lim_{m\rightarrow\infty}\mr{VP}(\sum_{\ul{k}}f^{(m)}_{\ul{k}}\ul{x}^{\ul{k}})$.
It is not hard to show the existence of the limit.
With this limit value, we may distinguish the 2nd and the 3rd example:
if $\mr{VP}^{(\infty)}(f)(v)=\mr{VP}(f)(v)$, then we do not need further mollification.
In the 2nd example, consider a valuation $v$ such that $v(s)>v(t)>0$, $v(s)-v(t)>v(s)/p$.
Then $\mr{VP}^{(\infty)}(f)(v)\neq\mr{VP}(f)(v)$, which implies that we {\em do} need mollifications around these valuations.
Thus, the question will be to find $g$ such that $\mr{VP}(f+g-g^p)(v)=\mr{VP}^{(\infty)}(f)(v)$.
This is not possible if we take $f$ to be an arbitrary formal function.
However, if $f$ comes from a function on $X$, we will show that we may find such a function $g$ if we take an alteration ``locally around $v$''.
Once this is shown, we find a global alteration dominating all the local alterations.
This is done by using the fact that the Riemann-Zariski space is compact.
From now on, we focus on the local arguments.

Generally, valuations are hard to describe explicitly.
However, this difficulty has a hierarchy, called the {\em transcendence defect} after Kedlaya.
The transcendence defect is a non-negative integer for each valuation, and when the transcendence defect is $0$, the valuation is relatively easy to describe.
For example, consider the case $S=\mr{Spec}(k[s,t])$ and $v$ has transcendence defect $0$ and the dimension of the center is of dimension $0$.
This kind of valuation is usually called an Abhyankar valuation.
Then $v(s)$, $v(t)$ are $\mb{Q}$-linearly independent,
and we have $v\bigl(\sum \alpha_{m,n}s^mt^n\bigr)=\min\bigl\{mv(s)+nv(t)\mid \alpha_{m,n}\neq0\bigr\}$.
In this case, we may concretely construct a mollifier $g$ that attains local admissibility around $v$.
See Lemma \ref{monomialvalok}.

In order to treat valuations with positive transcendence defect, we use the induction.
Let $v_0$ be a valuation with transcendence defect $n>0$ that we wish to show the local admissibility.
We assume that the local admissibility is known for valuations with transcendence defect less than $n$.
In the Riemann-Zariski space of $S$, valuations with transcendence defect less than $n$ clusters around $v_0$.
The main step is to find a valuation $v_1$ close enough to $v_0$ so that the transcendence defect is $n-1$
and $\mr{VP}^{(\infty)}(h)(v_0)=\mr{VP}^{(\infty)}(h)(v_1)$.
Once we find such a $v_1$, we may transfer the local admissibility of $v_1$ to $v_0$ using certain the monotonicity of valuation polygon.
To show the existence of $v_1$, we choose a suitable ``Berkovich disc containing $v_0$''.
If we consider $\mr{VP}^{(\infty)}(f)(v)$ as a function of $v$ on this Berkovich disc, a polygon appears.
This is the reason for the notation, and $\mr{VP}$ stands for ``valuation polygon''.
In this disc, $v_0$ becomes a terminal point, more precisely, a point of type either 1 or 4.
In the type 1 case, it is not hard to find a valuation $v_1$ in the disc, and the only issue is to treat the type 4 points.
What we need to show now is that as the valuation $v$ approaches to $v_0$ the valuation polygon $\mr{VP}^{(\infty)}(f)(v)$ stabilizes.
This is not true unless we use the algebraicity of $f$, and we believe that this is the core of the proof.
For this stability, we use the fact that Hadamard product preserves algebraic functions shown in \cite{SW}.

Let us overview the structure of this paper.
In \S\ref{sect1}, we fix terminologies and show basic facts in algebraic geometry.
In \S\ref{sect2}, we perform an analysis around type 4 point of a Berkovich disc.
Especially, we prove the stability of the valuation polygon around type 4 points using the algebraicity of Hadamard product.
In \S\ref{sect3}, we show the admissibility of the function using the stability result of \S\ref{sect2}.
This section is strongly inspired by Kedlaya's argument.
In \S\ref{sect4}, we recall some notations from Part I, and prove some complimentary results.
In \S\ref{sect5}, we show the main result.

\subsection*{Acknowledgments}\mbox{}\\
The author wishes to thank Kiran Kedlaya for answering various questions on his paper \cite{Ked4}.
He also thanks Takeshi Saito for encouragements and numerous feedbacks.
In particular, the author learned how to use reduced fiber theorem from him.
He is grateful to Yoichi Mieda, Kazuhiro Fujiwara, and Deepam Patel for comments and answering questions.

This work is supported by JSPS KAKENHI Grant Numbers 16H05993, 18H03667, 20H01790.

\subsection*{Notations and conventions}\mbox{}\\
Let $\ul{n}=(n_1,\dots,n_d)\in\mb{N}^d$. For an integer $m$, we denote
by $m\mid\ul{n}$ if $m\mid n_i$ for any $i$, and by $m\nmid\ul{n}$
otherwise. We put $\ul{0}:=(0,\dots,0)$, and $\ul{n}\geq\ul{0}$ if and
only if $n_i\geq0$ for any $i$.

Let $A$ be an integral $\mb{F}_p$-algebra. We denote by
$A^p:=\{a^p\mid a\in A\}\subset A$, and $A\subset A^{1/p}$ the ring $A$
considered as an $A$-algebra by the Frobenius endomorphism
$A\rightarrow A=:A^{1/p}$.

Let $X\rightarrow S$ and $T\rightarrow S$ are morphisms of schemes.
We often denote $X\times_S T$ by $X_T$. In particular, for a morphism
$X\rightarrow S$ and a point $s\in S$, we put
$X_s:=X\otimes_S\mr{Spec}(k(s))$. Similarly, for a sheaf $\mc{F}$ on
$X$, we denote by $\mc{F}_T$ the pullback on $X_T$.

For a scheme $X$, we put $\rat{X}:=\indlim_U\Gamma(U,\mc{O}_X)$ where
$U\subset X$ runs over open dense subscheme, which is usually called the
ring of rational functions.
If $X$ is an integral $k$-scheme, we sometimes denote $\rat{X}$ by
$k(X)$, which is a field.
We often denote by $f\in\mc{O}_X$ for $f\in\Gamma(X,\mc{O}_X)$.

We fix two different prime numbers $p$, $\ell$.
We fix a finite commutative ring $\Lambda$ whose residue characteristic is $\ell$.
Throughout this paper, we fix a field $k$ of characteristic $p\,(>0)$.

\section{Preliminary --- Algebraic geometry}
\label{sect1}

\subsection{}
Recall that a morphism $\pi\colon X\rightarrow Y$ between schemes is said to be {\em maximally dominant} if any generic point of $X$ is sent to a
generic point of $Y$. Flat morphisms between schemes are examples of maximally dominant morphisms (cf.\ \cite[Exp.\ II, 1.1.3]{G}).
We say that $\pi$ is an {\em alteration} if it is proper, surjective, generically finite, maximally dominant.
We say that $f$ is a {\em modification} if it is an alteration and there exists a open dense subscheme $V\subset Y$
such that $\pi^{-1}(V)\rightarrow V$ is an isomorphism.

\begin{dfn}
 \label{dfnfuncprop}
 Let $\pi\colon X\rightarrow S$ be a morphism of finite type, and
 $g\in\mc{O}_X$.
 \begin{enumerate}
  \item\label{dfnfuncprop-noncon}
       If $S=\mr{Spec}(k)$, the function $g$ is said to be {\em
       non-constant} if the induced morphism $X\rightarrow\mb{A}^1_k$
       is maximally dominant.
       If the image of $X$ in $\mb{A}^1_k$ is not dense, we say that $g$
       is {\em constant}.
       For a general $S$, it is {\em $S$-non-constant} if for any $s\in
       S$, the induced function $g_s$ on $X_s$ is non-constant.
	
  \item Assume $\pi$ is smooth. The function $g$ is said to be
	{\em $S$-separable}\footnote{We caution that this terminology is not consistent	with [SGA I, Exp.\ X].}
	if there exists an open subscheme $U\subset X$ such that the induced morphism $U\rightarrow\mb{A}^1\times S$
	is smooth and $U_s\subset X_s$ is dense for any $s\in S$.
       
  \item Assume $\pi$ is smooth.
	If $X$ is connected, the function $g$ is said to be {\em $S$-Frobenius separable}\footnote{
	This terminology is suggested to the author by T. Saito.}
	if there exists an $S$-separable function $g'$ such that $g=g'^q$ with some power $q$ of $p$.
	If $X$ is not connected, $g$ is said to be {\em $S$-Frobenius separable} if the restriction to each
	connected component of $X$ is so.
 \end{enumerate}
 If no confusion may arise, we often omit the prefix ``$S$-''.
\end{dfn}

\begin{lem}
 \label{lemonprofun}
 Let $\pi\colon X\rightarrow S$ be a morphism of finite type, and
 $g\in\mc{O}_X$.
 \begin{enumerate}
  \item\label{lemonprofun-nonconbc}
       Non-constancy is stable under base change. Moreover, for a
       surjective morphism $S'\rightarrow S$, if the pullback of $g$
       to $X\times_{S}S'$ is $S'$-non-constant, then $g$ is
       $S$-non-constant.

  \item\label{lemonprofun-flnonconchar}
       If $\pi$ is flat, then $g$ is $S$-non-constant if and only if the induced morphism $X\xrightarrow{\pi\times g} S\times\mb{A}^1$ is flat.

  \item\label{lemonprofun-pwsmoook}
       Assume $\pi$ is smooth, then $g$ is $S$-separable if and only if for any point $s\in S$ the base change $g_s$ is separable.
       This is equivalent to saying that $Z_s\subset X_s$ is nowhere dense for any $s\in S$, where $Z$ is the zero-locus of $\mr{d}g\in\Omega^1_{X/S}$.
       
  \item\label{lemonprofun-smgenpinsep}
       Assume $\pi$ is smooth, $S$ is excellent, and $g$ is $S$-non-constant.
       Then there exists a {\em quasi-finite} dominant morphism $F\colon S'\rightarrow S$ such that $F^*(g)$ in $\mc{O}_{X_{S'}}$ is $S'$-Frobenius separable.
       In fact, $F$ can be taken as an iterated Frobenius
       endomorphism followed by an immersion $U\subset S_{\mr{red}}\hookrightarrow S$, where $U$ is open inside $S_{\mr{red}}$.
 \end{enumerate}
\end{lem}
\begin{proof}
 To show \ref{lemonprofun-nonconbc}, we may easily reduce to the case where $S$ is the spectrum of a field.
 In this case, all the irreducible components of $X$ that appear after changing the base lie over the generic points of $X$ by [EGA IV, 4.5.11], and the claim follows.
 To show \ref{lemonprofun-flnonconchar}, if
 $S$ is $\mr{Spec}(k)$, the claim follows since any maximally dominant morphism
 $X\rightarrow\mb{A}^1=\mr{Spec}(k[t])$ from a reduced scheme $X$ is
 flat, because $\mc{O}_X$ is $t$-torsion free.
 When $S$ is general, use [EGA IV, 11.3.10].
 To show \ref{lemonprofun-pwsmoook}, since $f$ is smooth,
 the induced morphism $X_s\rightarrow\mb{A}^1_{k(s)}$ is flat for any
 $s\in S$ by \ref{lemonprofun-flnonconchar}.
 Now, use [EGA IV, 17.8.2].

 Let us check \ref{lemonprofun-smgenpinsep}.
 Since the claim is local around the generic points of $S$, we may
 shrink $S$ and assume that $S$ is irreducible, so replacing by its
 reduced scheme, we may even assume that $S$ is integral with the
 generic point $\eta$.
 Shrinking $S$ further, we may assume that $S$ is regular since $S$ is
 assumed excellent. Thus, $X$ is a disjoint union of
 irreducible components, and we may assume that $X$ is integral
 as well.

 Now, let $K/k$ be an extension of fields of finite type.
 We claim that $\bigcap_n k(K^{p^n})\subset l$, where $l$ is the
 algebraic closure of $k$ in $K$.
 Indeed, we may assume that $k$ is perfect.
 If $x\in\bigcap_n k(K^{p^n})$, then
 $k(x^{1/p^{\infty}})\subset\bigcap_n k(K^{p^n})\subset K$.
 Since $K$ is finitely generated over $k$, so is
 $k(x^{1/p^{\infty}})$. This is possible only when $x$ is algebraic over
 $k$, and the claim follows.
 Now, put $K:=\rat{X}$, $k:=\rat{S}$ which are fields by assumption, and apply this observation.
 Since any function in $l$ is constant, the non-constancy
 assumption implies that there exists an integer $n$ such that $g\in
 k(K^{p^n})$ but $g\not\in k(K^{p^{n+1}})$.
 Thus, there exists $g'\in k^{1/p^n}(K)$ such that $g'^{p^n}=g$.
 Then since $g'\not\in k^{1/p^n}(K^p)$, by
 [EGA $0_{\mr{IV}}$, 21.4.6], $g'$ is
 a separable function defined generically.
 Now, since $X$ is regular, the morphism $\mc{O}_X^q\rightarrow\mc{O}_X$ is faithfully flat for any power $q$ of $p$.
 Moreover, we have $\rat{X}\cong\rat{X}^{q}\otimes_{\mc{O}_X^q}\mc{O}_X$ because the Frobenius endomorphism is homeomorphic.
 Thus, by [EGA $0_{\mr{I}}$, 6.6.3], we have
 $\rat{X}^q\cap\mc{O}_X=\mc{O}_X^q$ in $\rat{X}$. Since
 $g\in\rat{X}^{p^n}\cap\mc{O}_X$, we have $g'\in\mc{O}_X$, and the claim
 follows.
\end{proof}

\begin{lem}
 \label{intFrob}
 Let $Y$ be an integral smooth scheme over $k$, and $\{y_i\}$ be a coordinate system on $Y$ at a $k$-rational point,
 which determines an embedding $\mc{O}_Y\hookrightarrow k\dd{\ul{y}}$ by {\normalfont [EGA IV, 17.6.3]}.
 Then for a power $q$ of $p$, we have
 \begin{equation*}
  \mc{O}_Y\cap k\dd{\ul{y}^q}=\mc{O}^q_Y\otimes_{k^q}k.
 \end{equation*}
 In particular, if $k$ is perfect,
 $f=\sum a_{\ul{k}}\ul{y}^{\ul{k}}\in\mc{O}_Y$ is
 separable if and only if there exists $p\nmid\ul{n}$ such that
 $a_{\ul{n}}\neq0$.
\end{lem}
\begin{proof}
 Since $Y$ is smooth over a field $k$, the homomorphisms $\mc{O}_{Y}^q\rightarrow\mc{O}_{Y}^q\otimes_{k^q}k\rightarrow\mc{O}_{Y}$ are faithfully flat,
 which can be checked easily by taking local coordinates.
 We have the canonical homomorphism $k\dd{\ul{y}^q}\otimes_{(\mc{O}_{Y}^q\otimes_{k^q}k)}\mc{O}_{Y}\rightarrow k\dd{\ul{y}}$.
 This is a surjective homomorphism between finite flat $k\dd{\ul{y}^q}$-modules of the same rank, and the homomorphism is an isomorphism.
 Thus, [EGA $0_{\mr{I}}$, 6.6.3] implies the formula.
 Let us show the second claim.
 Let $K_Y:=k(Y)$.
 The smoothness implies $K_Y^q\cap\mc{O}_Y=\mc{O}_Y^q$ in $K_Y$, and by taking $\otimes_{k^q}k$,
 we get $k(K_Y^q)\cap\mc{O}_Y=\mc{O}_Y^q\otimes_{k^q}k$ in $K_Y\otimes_{k^q}k\cong K_Y$.
 By [EGA $0_{\mr{IV}}$, 21.4.6], the claim follows.
\end{proof}

\begin{lem}
 \label{nowhconslem}
 Let $\pi\colon X\rightarrow S$ be a flat morphism of finite type between noetherian schemes,
 and $x\in\mc{O}_X$ be an $S$-non-constant function.
 Let $g$ be a function on an open subscheme $U\subset X$ such that $g|_{U_s}$ is not nilpotent around each generic point of $U_s$ for any $s\in S$.
 Then there exists an integer $N$ such that $(x^n g)|_U$ is $S$-non-constant for any $n>N$.
\end{lem}
\begin{proof}
 We show the lemma using the noetherian induction.
 It suffices to check that there exists an open dense subscheme
 $W\subset S$ such that the lemma holds for $S$ replaced by $W$.
 We may assume $S$ irreducible.
 It suffices to show the existence of $W$ after changing the base by a dominant morphism $\alpha\colon S'\rightarrow S$ of finite type.
 Indeed, if the lemma holds over an open subscheme $W'\subset S'$, then the lemma holds over the set $\alpha(W')$.
 Since $\alpha(W')$ is a constructible set by [EGA IV, 1.8.4] which contains the generic point of $S$,
 we may take an open subscheme $W\subset\alpha(W')$, and the claim follows.
 Moreover, it suffices to show the claim for each irreducible component of $X$.
 Thus, we may assume that for the generic point $\eta\in S$, the fiber $X_\eta$ is geometrically irreducible.
 Now, take a compactification $X\hookrightarrow\overline{X}$ over
 $S\times\mb{A}^1$.
 Replacing $\overline{X}$ by the closure of $X_\eta$ in $\overline{X}$,
 we may assume that the the open immersion
 $X_\eta\subset\overline{X}_\eta$ is dense, and thus,
 $\overline{X}_\eta$ is geometrically irreducible.
 Since $x$ is assumed flat and $S$ is irreducible,
 the morphism $\overline{x}\colon\overline{X}\rightarrow
 S\times\mb{A}^1$ is surjective.
 Thus $V(\overline{x})\cong\overline{X}\times_{\overline{x},\mb{A}^1}\{0\}\subset\overline{X}\rightarrow S$ is surjective.
 Shrinking $S$, we may assume that any fiber of $\overline{\pi}\colon\overline{X}\rightarrow S$ is geometrically irreducible by [EGA IV, 9.7.7].
 Shrinking $S$ further, we may take an open  affine subscheme $V\subset\overline{X}$ such that $V(\overline{x})\cap V\rightarrow S$ is surjective.
 By the fiberwise irreducibility, $V_s\cap U_s\subset U_s$ is dense.
 It suffices to show the lemma after replacing $X$ by $V$ and $x$ by $\overline{x}$.
 Thus, it suffices to show the lemma assuming that $X$, $S$ are irreducible affine, $V(x)\rightarrow S$ is surjective,
 $V(x)\subsetneq X$, and $X\rightarrow S$ is fiberwise irreducible, but $x\colon X\rightarrow S\times\mb{A}^1$ is not anymore flat.

 Let $Y$ be a noetherian scheme and $\mc{I}\subset\mc{O}_Y$ be an ideal.
 Consider the $\mc{I}$-adic filtration on $\mc{O}_Y$.
 For $f\in\mc{O}_Y$, $\mr{ord}_{\mc{I}}(f)=:n$ denotes the largest integer $k$ or $\infty$ such that $f\in\mc{I}^k$.
 For such $f$, the image in $\mr{gr}_{\mc{I}}^{n}\mc{O}_Y$ is denoted by $\sigma(f)$ if $n$ is an integer and $0$ if $n=\infty$.
 This defines a map $\sigma_Y\colon\mc{O}_Y\rightarrow\mr{gr}_{\mc{I}}(\mc{O}_Y)$, which is merely a map of sets.
 By [EGA $0_{\mr{I}}$, 7.3.7], the $\mc{I}$-adic filtration is separated if $Y$ is integral and $\mc{I}\neq\mc{O}_Y$ since $1+\mc{I}$ does not contain $0$.
 In this case, this in particular implies that $\sigma(f)\neq0$ unless $f=0$. The following claims are easy to check:
 \begin{itemize}
  \item Let $Y\rightarrow\mr{Spec}(k)$ be a morphism of finite type between irreducible schemes, and $y$ a non-constant function on $Y$.
	If $h\in\mc{O}_Y$ is a {\em constant} function which is non-nilpotent,
	then $\mr{ord}_{y\mc{O}_Y}(hg)=\mr{ord}_{y\mc{O}_Y}(g)$ for any function $g$.
	This is because $h$ is invertible.

  \item Let $\mc{I}\subset\mc{O}_X$ be an ideal, and $S'\rightarrow S$
	be a morphism. Put $\rho\colon X':=X\times_SS'\rightarrow X$.
	If $\mr{gr}_\mc{I}\mc{O}_X$ is $\mc{O}_S$-flat, then we have a
	canonical isomorphism
	$\mr{gr}_{\mc{I}\mc{O}_{X'}}\mc{O}_{X'}\cong
	\mr{gr}_\mc{I}\mc{O}_X\otimes_{\mc{O}_S}\mc{O}_{S'}$.
	Indeed, the direct summand $\mc{I}^{n-1}/\mc{I}^{n}$ is
	$\mc{O}_S$-flat, thus $\mc{O}_{X}/\mc{I}^n$ is flat as well, and
	$\mc{I}^n\otimes_{\mc{O}_S}\mc{O}_{S'}\rightarrow\mc{I}^n\mc{O}_{X'}$
	is an isomorphism.
  \item We keep the notation.
	Let $g$ be a function on $X$. If $\rho^*\sigma_X(g)\in
	\mr{gr}_\mc{I}\mc{O}_X\otimes_{\mc{O}_S}\mc{O}_{S'}$ does not
	vanish, then $\rho^*\sigma_X(g)=\sigma_{X'}\rho^*(g)$ via the
	previous isomorphism.
 \end{itemize}
 Put $\mc{I}:=x\mc{O}_X$.
 We may replace $X$ and $S$ by its underlying reduced schemes, and assume $X$ and $S$ to be integral.
 Invoking [EGA IV, 9.7.7], by shrink $S$ further, we may assume that each fiber of $\pi$ is geometrically integral.
 Since $X$ is assumed integral and $V(x)$ is non-empty, the $\mc{I}$-adic filtration on $\mc{O}_X$ is separated.
 This implies that the symbol map $\sigma$ is injective.
 Moreover, since $V(x)\neq X$, $x$ is non-zero. Thus, $\mr{Ann}(x)=0$,
 and the multiplication by $\sigma(x)$ on $\mr{gr}_{\mc{I}}\mc{O}_X$ is
 injective.
 Shrinking $S$, we may assume that
 $\mr{Spec}_X(\mr{gr}_{\mc{I}}\mc{O}_X)\rightarrow S$ is flat by
 [EGA IV, 6.9.1].
 Note that, since $V(x)\rightarrow S$ is assumed surjective, the
 $\mc{I}$-adic filtration on $\mc{O}_X$ remains separated even after
 shrinking $S$.

 Now, write $g=g_1/g_2$ with $g_i\in\mc{O}_X$.
 Since $g$ is not $0$, we have $\sigma(g_i)\neq0$ by the injectivity of $\sigma$.
 By shrinking $S$ further, we may assume that $\bigl(\sigma(g_i)\bigr)_s$ does not vanish for any $s\in S$.
 Take $N:=\max\bigl\{0,\mr{ord}_{\mc{I}}(g_2)-\mr{ord}_{\mc{I}}(g_1)\bigr\}$.
 Since $(x^ng_1/g_2)_s$ is not $0$ by the assumption that $\pi$ is fiberwise geometrically integral,
 if $(x^ng_1/g_2)_s$ is constant, we have
 \begin{equation}
  \label{equifconst}\tag{$\star$}
  \mr{ord}_{\mc{I}_s}\sigma\bigl((x^ng_1)_s\bigr)
   =
   \mr{ord}_{\mc{I}_s}\sigma(g_{2,s}).
 \end{equation}
 On the other hand, $\mr{ord}_{\mc{I}}(x^n g_1)=n+\mr{ord}_{\mc{I}}(g_1)$.
 Since $\bigl(\sigma(g_i)\bigr)_s$ does not vanish for any $s\in S$, we have $\mr{ord}_{\mc{I}}(g_i)=\mr{ord}_{\mc{I}_s}(g_{i,s})$.
 Thus, (\ref{equifconst}) cannot happen for $n>N$, and we have the lemma.
\end{proof}

\subsection{}
\label{nonconsfam}
Let $\pi\colon X\rightarrow S$ be a morphism. A set of functions
$\mf{h}:=\{h_i\}_{i\in I}$ in $\mc{O}_X$ is said to be a 
{\em non-constant family} on $X$ over $S$ if for any point $s\in S$,
there exists $i\in I$ such that $h_i|_{Y}$ is non-constant for
any irreducible component $Y$ of $X_s$ {\em of dimension $>0$}, and
$h\in\mf{h}$ implies $h^n\in\mf{h}$ for any integer $n>0$.

\begin{rem*}
 In the definition, let $X'_s\subset X_s$ be the union of irreducible components of dimension $>0$.
 Caution that $h_i|_{X'_s}$ is a non-constant function (cf.\ Definition \ref{dfnfuncprop}.\ref{dfnfuncprop-noncon}),
 but $h_i|_{X_s}$ is not if $X'_s\neq X_s$.
 For example, if $\pi$ is a finite morphism, then any family of functions is non-constant family,
 but no function is non-constant.
\end{rem*}

\begin{lem}
 \label{nonconst-prop}
 \begin{enumerate}
  \item\label{nonconst-prop-fin}
       Let $S$ be a scheme of finite type over $k$,
       and let $\pi\colon X\rightarrow Y$ be a finite surjective morphism between flat $S$-schemes of finite type
       which are equidimensional over $k$.
       If $\mf{h}$ is a non-constant family on $Y$, so is $\pi^*\mf{h}$.
  \item\label{nonconst-prop-afn}
       Let $X$ be an affine scheme of finite type over $k$.
       Let $\Omega/k$ be a field extension, and assume we
       are given a closed subscheme $Z\subset X_\Omega$ such that each
       irreducible component is of dimension $>0$.
       Then there exists a function $f\in\mc{O}_X$ such that its
       pullback to $Z$ is non-constant.

  \item\label{nonconst-prop-bc}
       Let $X\rightarrow S$ be a morphism of finite type between affine schemes.
       For any morphism $S'\rightarrow S$, and any closed subscheme $Z\subset X\times_S S'$,
       the pullback of $\Gamma(X,\mc{O}_X)$ is a non-constant family on $Z$.
 \end{enumerate}
\end{lem}
\begin{proof}
 Let us show \ref{nonconst-prop-fin}.
 Take $s\in S$. Since $X$ and $Y$ are flat over $S$,
 $X_s\rightarrow Y_s$ is a finite morphism between equidimensional $k(s)$-schemes with the same dimension by [EGA IV, 14.2.4].
 Thus, this is maximally dominant. Since the pullback of a non-constant function by a maximally dominant morphism remains to be
 non-constant by Definition \ref{dfnfuncprop}.\ref{dfnfuncprop-noncon}, the claim follows.

 Let us show \ref{nonconst-prop-afn}.
 By limit argument, we may assume that $\mr{Spec}(\Omega)$ is the generic point of an integral affine scheme $W$ of finite type over $k$ and
 there exists a closed subscheme $\widetilde{Z}\subset X_W$ such that $\widetilde{Z}_\Omega=Z$ and $\widetilde{Z}$ is maximally dominant over $W$.
 Let $\{\widetilde{Z}_i\}$ be the set of irreducible components of $\widetilde{Z}$.
 By shrinking $W$, we may assume that $\widetilde{Z}_i\cap\widetilde{Z}_j$ is either faithfully flat over $W$ or empty for any $i,j$ by [EGA IV, 6.9.1].
 For any $w\in W$, any irreducible component of $\widetilde{Z}_{i,w}$ is an irreducible component of $\widetilde{Z}_w$.
 Indeed, let $Y\subset \widetilde{Z}_{i,w}$ be an irreducible component.
 If $Y$ were not an irreducible component of $\widetilde{Z}_w$, we could take an irreducible component $Y'$ of $\widetilde{Z}_w$ which contains $Y$.
 There exists $j$ such that $Y'$ is contained in $\widetilde{Z}_j$.
 Since $Y'\neq Y$, $i$ and $j$ are not the same.
 By using [EGA IV, 14.2.4], we have $\dim(\widetilde{Z}_{i,\eta})=\dim(Y)\leq\dim(\widetilde{Z}_i\cap\widetilde{Z}_j)_\eta<\dim(\widetilde{Z}_{i,\eta})$,
 which is a contradiction, and $Y$ is an irreducible component of $\widetilde{Z}_w$.

 Now, fix a closed point $w\in W$, and assume that there exists a function $f\in\mc{O}_{X_W}$ such that $f|_{\widetilde{Z}_w}$ is non-constant.
 For each $i$, we take a point $z_i\in\widetilde{Z}_{i,w}$ which is a generic point of $\widetilde{Z}_w$.
 Since $f|_{\widetilde{Z}_w}\colon \widetilde{Z}_w\rightarrow\mb{A}^1_{k(w)}$ is maximally dominant, this is flat around $z_i$.
 Then, since $\widetilde{Z}\rightarrow W$ is assumed flat, $\widetilde{Z}\rightarrow\mb{A}^1_W$ is flat around $z_i$ by [EGA IV, 11.3.10].
 Thus $f$ is flat around the generic point of $\widetilde{Z}_i$ by [EGA IV, 11.1.1], and $f$ is non-constant over $\Omega$.
 Since $X_W$ is affine, any function on $X_w$ extends to a function on $X_W$, and thus,
 it suffices to find a function $g$ on $X_w$ such that $g|_{\widetilde{Z}_w}$ is non-constant.
 In particular, we may assume that $\Omega$ is an algebraic finite extension of $k$.
 By taking the image of $Z$ to $X$, we may assume that $\Omega=k$.
 Since $X$ is affine, any function on $Z$ extends to $X$, and thus, we may even assume that $Z=X$.
 Now, let $X=\bigcup_{1\leq i\leq N} X_i$ be the decomposition into irreducible components, and take a closed immersion $X\hookrightarrow\mb{A}^n_k$ such that $n>N$.
 For each $i$, let $l_i\subset\mb{A}^n_k$ be a line
 ({\it i.e.}\ $1$-dimensional irreducible closed subscheme defined by polynomials of degree $1$ with coefficients in $k$)
 contained in $X_i$ if it exists, and $\emptyset$ it does not.
 Let $l_{i,0}$ be the parallel transform of $l_i$ which contains the origin.
 Since $n>N$, we may take a hyperplane $H\subset\mb{A}^n$ passing through the origin which does not contain any of $l_{i,0}$.
 The function defining $H$ is a function we are looking for.

 Let us show the last claim.
 We may assume that $S'=\mr{Spec}(\Omega)$ is a point for some field $\Omega$.
 Since $S$ is assumed affine, for any closed subscheme $W\subset S$, the map $\Gamma(S,\mc{O}_S)\rightarrow\Gamma(W,\mc{O}_W)$ is surjective.
 Thus, we may assume that $S$ is integral and the image of $\mr{Spec}(\Omega)$ is the generic point $\eta$ of $S$.
 By \ref{nonconst-prop-afn}, we may find a function $f\in\Gamma(X,\mc{O}_X)\otimes\mc{O}_{S,\eta}$
 such that $f$ is non-constant on each irreducible component of positive dimension of $Z$.
 There exists $g\in\Gamma(S,\mc{O}_S)^{\times}$ such that $gf$ is defined on $\Gamma(X,\mc{O}_X)$.
 By construction, $gf\in\Gamma(X,\mc{O}_X)$ is non-constant on each irreducible component of positive dimension of $Z$,
 and we conclude the proof.
\end{proof}

\begin{lem}
 \label{geomirrlem}
 Let $X\rightarrow S$ be a smooth morphism between {\em excellent} {\normalfont(}noetherian{\normalfont)} schemes of finite dimension.
 Then there exists a finite surjective morphism $T\rightarrow S$ such
 that
 \begin{quote}
  {\normalfont (*)} for any $x\in X_T:=X\times_S T$, there exists an open
  neighborhood $U$ such that each fiber of $U\rightarrow T$ is
  geometrically irreducible.
 \end{quote}
\end{lem}
\begin{proof}
 First, we note that the following holds:
 \begin{itemize}
  \item If (*) holds for $T$, then it also holds for any
	$T'\rightarrow T$.
  \item Let $T'=\bigcup T'_i$ where $T'_i$ is an irreducible
	component. If (*) holds for $T=T'_i$, then (*) holds for $T'$.
 \end{itemize}
 For the second claim, we note the following: Let $Y$ be a scheme and $Y_i$ be an irreducible component.
 Take $y\in\bigcap_{i\in I} Y_i$ for some finite set $I$.
 Assume we are given an open neighborhood $V_i\subset Y_i$ of $y$ for each $i\in I$.
 Put $V:=Y\setminus\bigcup_{i\in I}(Y_i\setminus V_i)$.
 Then $V$ is a neighborhood of $y$ in $Y$ and $V\cap Y_i\subset V_i$ for each $i\in I$.

 Since the dimension of $S$ is finite, we show the claim by the induction on the dimension of $S$.
 Assume the claim holds for $\dim(S)<n$, and assume now that $\dim(S)=n$.
 Consider the finite morphism $\coprod S_{i,\mr{red}}\rightarrow S$, where $\{S_i\}$ is the set of irreducible components,
 and replacing $S$ by $\coprod S_i$, we may assume that $S$ is integral with generic point $\eta$.
 There exists a finite extension $L$ of $k(\eta)$ such that the number of irreducible components of $X_L$ is the same as that of geometric irreducible components.
 We replace $S$ by the normalization of $S$ in $L$, we may assume that $S$ is normal and each irreducible component of the generic fiber is geometrically irreducible.
 By considering the connected component, we may assume that $X$ is connected.

 In this situation, let us show that the generic fiber is geometrically irreducible.
 Let $\{X_{\eta,i}\}$ be the set of irreducible components of $X_\eta$. Let $X_i$ be the closure of $X_{\eta,i}$,
 and put $X_{ij}:=X_i\cap X_j$.
 Assume $X_{ij}\cap X_{t}\neq\emptyset$ for some $t\in S$.
 Take a point $x$ in the intersection.
 Then $\mc{O}_{X,x}$ has more than one minimal primes, so it cannot be geometrically unibranch.
 However, since $\mc{O}_{S,t}$ is assumed to be normal, it is geometrically unibranch (cf.\ [EGA $0_{\mr{IV}}$, 23.2.1]) this contradicts with [EGA IV, 17.5.7].
 Thus, $X_{ij}$ does not intersect with $X_{t}$ for any $t\in S$.
 This implies that $X_i$ and $X_j$ are disjoint. Since $X$ is assumed connected, this can happen only when $X_i=X_j$, thus $X_\eta$ is irreducible.

 Since $X_\eta$ is geometrically irreducible,
 there exists an open dense subscheme $U\subset S$ such that for any $s\in U$, $X_s$ is geometrically irreducible by [EGA IV, 9.7.8].
 Since the dimension of $Z:=S\setminus U$ is $<n$,
 we may use the induction hypothesis, and we have a finite dominant morphism $W\rightarrow Z$ over which (*) is satisfied.
 Use Gabber's lemma \cite[1.6.2]{O2} to produce a finite dominant morphism $S'\rightarrow S$ such that each irreducible component of
 $Z':=S'\times_S Z$ maps to $W$. Thus (*) holds for $Z'$ by using the remarks at the beginning of this proof.
 Take any point $x\in X':=X\times_S S'$.
 If $x$ is not over $Z'$, we can just take the complement of $X'\times_{S'}Z'\subset X'$ since each fiber of $X_U\rightarrow U$ is geometrically irreducible.
 Assume $x$ is mapped to $Z'$. Then there exists an open neighborhood $V$ of $x$ in $X'\times_{S'}Z'$ which is fiberwise geometrically irreducible over $Z'$.
 Take any extension of $V$ to an open subset in $X'$.
 Since $X'$ is already fiberwise geometrically irreducible outside of $Z'$, this neighborhood is what we need, and (*) is satisfied for $X'\rightarrow S'$.
\end{proof}

\subsection{}
For terminologies concerning valuations, we follow \cite{V}.
Let $v$ be a valuation on an integral ring $A$ over a finite field $\mb{F}_p$.
Then $v$ extends uniquely to the perfection $A^{\mr{perf}}:=A^{1/p^{\infty}}:=\indlim A^{1/p^n}$.
Abusing notations, we denote the extension also by $v$.

Let $A$ be an integral noetherian ring, and $v$ be a valuation on it.
Then it is known that the value semigroup is well-ordered by \cite[Prop 9.1]{V}.
We moreover have the following result which can be found in Zariski-Samuel's book.
For the reader's convenience, we recall also the proof.

\begin{lem*}
 \label{discvalsemgr}
 Let $A$ be an integral noetherian ring with its field of fractions $K$,
 and let $v$ be a height $1$ valuation on $K$ centered on $A$.
 Then the value semigroup $\Gamma:=v(A\setminus\{0\})\subset\mb{R}_{\geq0}$ is
 discrete.
\end{lem*}
\begin{proof}
 Let $\mf{p}\subset A$ be the center of $A$. Since the value semigroup of $A$ is contained in that of $A_{\mf{p}}$,
 we may assume that $\mf{p}$ is a maximal ideal of $A$. For $r\in\mb{R}_{\geq0}$,
 let $I_r:=\bigl\{a\in A\mid v(a)\geq r\bigr\}$, which is an ideal of $A$.
 The semigroup $\Gamma$ is well-ordered as we remarked above.
 Take the minimal value $\gamma\in\Gamma\setminus\{0\}$.
 Then $\mf{p}=I_\gamma$.
 Take any $\rho\in\mb{R}_{\geq0}$ greater than $\gamma$.
 We may find an integer $N$ such that $\rho<N\gamma$.
 Then we have $\mf{p}^N\subset I_{\rho}$. Since $\mf{p}$ is assumed to be maximal, the quotient $A/\mf{p}^N$ is an Artinian ring.
 This implies that there can only be finitely many distinct ideals in the set $\{I_r\}_{\gamma\leq r\leq\rho}$,
 and thus $\Gamma\cap[\gamma,\rho]$ is a finite set.
\end{proof}

\begin{cor*}
 Let $K$ be a field, $v$ be a valuation on $K$ of height $1$, and $R$ be a {\em noetherian} subring of $K$ such that $v(R)\geq0$.
 Let $M\subset K$ be a finitely generated $R$-submodule.
 Then $v(M\setminus\{0\})\subset\mb{R}$ is a discrete subset.
\end{cor*}
\begin{proof}
 Write $M=\sum_{i\in I} R\cdot\alpha_i$ where $\alpha_i\in K$.
 Since $M$ is finitely generated, we may take $I$ to be a finite set.
 Thus, there exists $\beta\in K^{\times}$ such that $v(\beta\cdot\alpha_i)\geq0$ for any $i\in I$.
 Consider $R':=R\left[\{\beta\cdot\alpha_i\}_{i\in I}\right]\subset K$.
 This is a noetherian ring such that $v(R'\setminus\{0\})\geq0$ and contains $\beta\cdot M$.
 Invoking the lemma above, $v(R'\setminus\{0\})$ is a discrete subset of $\mb{R}_{\geq0}$.
 Thus, $v(M\setminus\{0\})=v(\beta\cdot M\setminus\{0\})-v(\beta)\subset v(R'\setminus\{0\})-v(\beta)$ is discrete.
\end{proof}

\section{Analysis around type 4 points}
\label{sect2}

\subsection{}
\label{fixnotatbersp}
In this section, we fix an algebraically closed complete non-archimedean
valuation field $\ell$ of characteristic $p>0$ with non-trivial
valuation $v$.
We denote by $\mf{o}_\ell$ the ring of integers of $\ell$.
We, moreover, fix a subring $R\subset\mf{o}_\ell$, which is assumed {\em
noetherian}. This implies that the value semigroup $v(R)$ is a discrete
subset of $\mb{R}_{\geq0}$ by Lemma \ref{discvalsemgr}.
The Tate algebra over $\ell$ is denoted by
$\ell\!\left<t\right>$, and
$R\!\left<t\right>:=R\dd{t}\cap\mf{o}_\ell\!\left<t\right>$.

\subsection{}
\label{NPdefandfirstprop}
A Berkovich point $\alpha$ of the Tate algebra $\ell\!\left<t\right>$ is
a multiplicative semi-norm $|\cdot|_\alpha$ such that $\leq\|\cdot\|$, where $\|\cdot\|$ is the $0$-Gauss norm.
However, in order to keep accordance with the notations for Riemann-Zariski space,
we use valuation theoretic notation, and put $v_\alpha(f):=\log(|f|_\alpha)$.
This is possible since there is no archimedean semi-norm on $\ell\!\left<t\right>$.
Beware that, since points of Berkovich space are merely semi-norms,
the valuation $v_\alpha\colon \ell\!\left<t\right>\rightarrow\mb{R}_\infty$ corresponding to a point $\alpha$,
where $\mb{R}_\infty:=\mb{R}\sqcup\{\infty\}$, may have non-trivial preimage of $\infty$.

For $\omega\in\ell$ and $r\in\left[0,\infty\right]$, we define a
semi-norm $v_{\omega,r}$ by
\begin{equation*}
 v_{\omega,r}\bigl(\sum a_i(t-\omega)^i\bigr):=
  \min\bigl\{v(a_i)+ir\bigr\}.
\end{equation*}
This semi-norm is multiplicative, and defines a point of $\mb{D}_\ell$.
This point is denoted by $\alpha_{\omega,e^{-r}}$ in \cite{Ked4}.
Let $\alpha$ be a Berkovich point on $\ell\!\left<t\right>$.
We put $r_\alpha:=-\log(r(\alpha))$ in $\mb{R}_{\infty}$, where $r(\alpha)$ is the radius of $\alpha$ (cf.\ \cite[2.2.11]{Ked4}).
For $s\in[0,r_\alpha]$, we have the valuation $v_{\alpha,s}$
corresponding to the point $\alpha_{e^{-s}}$.
Note that $v_{\alpha,0}$ corresponds to the Gauss point for any
$\alpha$. These points can be depicted in the following figure.

\begin{figure}[h]
\centering
  \setlength{\unitlength}{1bp}%
  \begin{picture}(201.14, 53.74)(0,0)
  \put(0,0){\includegraphics{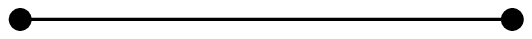}}
  \put(166.74,8.73){\fontsize{12}{15}\selectfont $v_{\alpha,r_{\alpha}}$}
  \put(5.67,8.73){\fontsize{12}{15}\selectfont $v_{\alpha,0}=v_{0,0}$}
  \put(13.56,36.96){\fontsize{12}{15}\selectfont Gauss}
  \put(174.64,36.96){\fontsize{12}{15}\selectfont $\alpha$}
  \end{picture}%
\end{figure}

For $f\in\ell\!\left<t\right>$, we consider the function
\begin{equation*}
 \NP_\alpha(f)\colon\left[0,r_\alpha\right]
  \rightarrow\mb{R};\quad
  s\mapsto v_{\alpha,s}(f).
\end{equation*}
This function $\NP_\alpha(f)$ is continuous piecewise affine,
convex, and there are only fintely many breaks. Indeed, we may assume
that $\ell$ is spherically complete. In this case, we may assume that
$\alpha=v_{\omega,r}$ for some $\omega\in\ell$ and $r$.
Then $\NP_\alpha(f)$ is nothing but (a truncation of) the \newton
polygon of $f(z+\omega)$ around $v_{0,\infty}$.
Now, let $g=\sum a_it^i\in\ell\!\left<t\right>$.
Then there exists $n$ such that $v(a_i)>v(a_n)$ for any $i>n$.
Then we have 
$\NP_{v_{0,\infty}}(g)=\NP_{v_{0,\infty}}(\sum_{i\leq n}a_it^i)$.
The latter function obviously satisfies the claim, and so is $\NP_\alpha(f)$.
This function is called the {\em \newton polygon} at $\alpha$.

\begin{lem}
 \label{finforfampoly}
 Let $R'\subset\mf{o}_\ell$ be a noetherian ring, and $\{f_i\}_{i\in I}$ be
 a set of functions in $R'\!\left<t\right>$.
 Then the function $V(s):=\inf_{i\in
 I}\bigl\{\NP_{\alpha}(f_i)(s)\bigr\}$
 is continuous piecewise affine, convex, and there are only fintely many
 breaks.
\end{lem}
\begin{proof}
 Since the slopes of $\NP(f_i)(s)$ are non-negative integers and the
 slope decreases as $s$ goes to the infinity,
 the lemma holds over $\left[\varepsilon,\infty\right[$ for any
 $\varepsilon>0$.
 It suffices to check that the infimum has ``initial slope'' around
 $0$, namely there exists $\epsilon>0$ and $d$, $a$ such that
 $V(s)=ds+a$ for any $s\in[0,\epsilon]$.
 Because $R'$ is assumed notherian,
 $D:=v(R')\subset\left[0,\infty\right[$ is discrete by Lemma \ref{discvalsemgr}.
 Thus, the set $\{v_0(f_i)\}$ has the minimum.
 Let $a$ be the minimum. Then $V(0)=a$.
 Let $J$ be the set of $i\in I$ such that $\NP(f_i)(0)=a$.
 Since the slopes are integer, there exists the minimum in the set of
 initial slopes of $\NP(f_i)$ for $i\in J$.
 Let $d$ be the slope.
 Let $a'$ be the element in $D$ next to $a$.
 Then we may check that $V(s)=ds+a$ for $s\in\left[0,(a'-a)/s\right]$.
 Thus the lemma follows.
\end{proof}

\subsection{}
Now, we fix a Berkovich point $\alpha$, and we omit the index $\alpha$
from $\NP_\alpha(f)$.
Let us fix an element $\pi\in R$, and let
$\mbf{a}:=\{a_k\}_{k\geq0}$ be a sequence of elements in
$R\!\left<t\right>$. Valuations on $\ell\!\left<t\right>$ extends
uniquely to that of $\ell\!\left<t\right>^{\mr{perf}}$, and we identify
the Berkovich space of $\ell\!\left<t\right>$ and
$\ell\!\left<t\right>^{\mr{perf}}$.
Now, for $n\geq0$, we put
\begin{equation*}
 g^{(n)}(\mbf{a}):=\sum_{k=0}^{n}(a_k/\pi)^{1/p^k},\quad
  \NP^{(n)}(\mbf{a}):=\min_n\bigl\{0,\NP(g^{(n)}(\mbf{a}))\bigr\}.
\end{equation*}
For any $\varepsilon>0$, there exists $N$ such that $-v(\pi)/p^N>-\varepsilon$.
This implies that
\begin{equation}
 \label{NPtruncsta}
 \min\bigl\{-\varepsilon,\NP^{(n)}(\mbf{a})\bigr\}=
  \min\bigl\{-\varepsilon,\NP^{(N)}(\mbf{a})\bigr\}
  \leq0
\end{equation}
for any $n\geq N$. Thus,
\begin{equation*}
 \NP^{(\infty)}(\mbf{a}):=\lim_{n\rightarrow\infty}\NP^{(n)}(\mbf{a})
\end{equation*}
converges.
Now, let $I$ be a set, and for $i\in I$, assume we are given a sequence $\mbf{a}_i:=\{a_{i,k}\}_{k\geq0}$ in $R\!\left<t\right>$.
We put
\begin{equation*}
 \NP^{(\infty)}(\{\mbf{a}_i\}_{i\in I}):=
  \inf_{i\in I}\big\{\NP^{(\infty)}(\mbf{a}_i)\}.
\end{equation*}
The infimum is in fact the minimum.
Indeed, for any $N>0$, we have $v\bigl(\sum_{N\geq n\geq0}(R/\pi)^{1/p^n}\bigr)\subset v(R^{1/p^N}/\pi)$.
Since $R^{1/p^N}$ is noetherian, this is a discrete set by Lemma \ref{discvalsemgr}.
Thus, by (\ref{NPtruncsta}), $v\bigl(\sum_{n\geq0}(R/\pi)^{1/p^n}\bigr)\cap\left]-\infty,0\right[$ is a discrete subset, and the minimum exists.

\begin{lem}
 \label{defandfinofNP}
 The function $\NP^{(\infty)}(\{\mbf{a}_i\}_{i\in I})(s)$ is continuous piecewise affine and convex.
 Furthermore, there are only fintely many breaks.
 If $\alpha$ is a type 1 point, then the terminal slope is $0$, in other words the function is constant for $s\gg0$.
\end{lem}
\begin{proof}
 We may assume that $\ell$ is spherically complete.
 In this case, we may write $\alpha=v_{\omega,r}$.
 Since $\NP(f)$ is the same as the \newton polygon of $f(t+\omega)$ around $v_{0,\infty}$ (and truncate to $[0,r]$)
 for any $f\in\!\ell\left<t\right>^{1/q}$ where $q$ is some power of $p$,
 by replacing $R$ by $R[\omega]\subset\mf{o}_{\ell}$, which is indeed noetherian, we may assume that $\alpha=v_{0,\infty}$.
 Note that if $\alpha$ is a point of type 1 equal to $v_{\omega,\infty}$,
 then $\NP_\alpha(f(t))=\NP_{v_{0,\infty}}(f(t+\omega))$ and we do not need truncation.
 This implies that the claim for type 1 point is also reduced to the case $\alpha=v_{0,\infty}$.
 By (\ref{NPtruncsta}) together with Lemma \ref{finforfampoly} applying
 for $R'=R^{1/p^N}$,
 $\bigl(\NP^{(\infty)}(\mbf{a})\bigr)^{-1}
 (\left]-\infty,-\varepsilon\right[)$
 is continuous piecewise affine and convex for any $\varepsilon>0$.
 Thus, we only need to check that the breaks are finite and the terminal
 slope is $0$.
 
 Let $C_i:=\lim_{n\rightarrow\infty}\min \bigl\{0,v(\sum_{k=0}^n(a_{i,k}(0)/\pi)^{1/p^k})\bigr\}$.
 This limit converges since we have $\lim_{k\rightarrow\infty}\min\bigl\{0,v(a_{i,k}(0)/\pi)/p^k\bigr\}=0$.
 Write $a_{i,k}=a'_{i,k}+a_{i,k}(0)$, and put $\mbf{a}'_i:=\{a'_{i,k}\}_{k}$. 
 Since $\NP(a'_{i,k})$ does not contain segments of slope $0$ for any $i,k$, we have
 \begin{equation*}
  \NP^{(\infty)}(\mbf{a}_i)=\min
   \bigl\{\NP^{(\infty)}(\mbf{a}'_i),C_i\bigr\}.
 \end{equation*}
 Thus, we may assume that $a_{i,k}(0)=0$ for any $k$, namely $a_{i,k}$ does not have the contant term.

 Fix a point $\gamma\in\left]0,\infty\right[$ such that $\NP^{(\infty)}(\mbf{a})(\gamma)<0$.
 If we do not have such a point, $\NP^{(\infty)}(\mbf{a})=0$, and there is nothing to prove.
 On the interval $[0,\gamma]$, the function
 $\NP^{(\infty)}(\mbf{a})$ has only finitely many breaks because of
 (\ref{NPtruncsta}) and Lemma \ref{finforfampoly} (applied for $R'=R^{1/p^N}$).
 Thus, we only need to check the finiteness for
 $\left[\gamma,\infty\right[$. There exists an integer $N>0$ such that
 $(N+1)\gamma>v(\pi)$. For $g=\sum_{i\geq0}g_it^i\in R\!\left<t\right>$,
 denoted by $g^{\leq N}:=\sum_{N\geq i\geq0}g_it^i\in R[t]$, the
 $N$-truncated polynomial.
 Since $g_i\in R$ and thus $v(g_i)\geq0$ and $\alpha=v_{0,\infty}$,
 we have $\min\{0,v_s(g/\pi)\}=\min\{0,v_s(g^{\leq N}/\pi)\}$ for any $s\in\left[\gamma,\infty\right[$.
 Thus, by putting $\mbf{a}^{\leq N}_i:=\{a_{i,k}^{\leq N}\}_{k}$, we have
 $\NP^{(\infty)}(\{\mbf{a}_i\})=\NP^{(\infty)}(\{\mbf{a}_i^{\leq N}\})$ on $\left[\gamma,\infty\right[$.
 Thus, from now on, we may and do assume that $a_{i,k}$ is a polynomial without constant term of degree $\leq N$ for any $i\in I$ and $k\geq0$.

 Now, put $R':=R^{1/p^{\lceil\log_p(N)\rceil}}$, and
 let $\mb{N}'$ be the set of natural numbers $0<n\leq N$ such that
 $p\nmid n$. Denote by $[1,N]_{\mb{N}}:=[1,N]\cap\mb{N}$.
 Let $P$ be the (non-empty) set of monomials $\alpha t^i$ ($i\in\mb{Q}$) in $\ell[t]^{\mr{perf}}$ such that
 \begin{enumerate}
  \item $i$ is either in $p^{-k}\mb{N}'$ for some positive integer $k$,
	or in $[1,N]_{\mb{N}}$. In the latter case, put $k=0$.
  \item $\alpha\in(\pi^{-1}R')^{1/p^k}$ and $v(\alpha)+i\gamma<0$.
 \end{enumerate}
 Now, let $\{a_m\}_m$ be a sequence of $a_m\in R[t]$ of degree $\leq N$ and $a_m(0)=0$.
 Let $n\geq0$, and write $g^{(n)}(\{a_m\})=\sum_{i\in p^{-n}\mb{N}}g_it^i$ with $g\in\ell^{\mr{perf}}$.
 We claim that either $g_it^i\in P$ or $v(g_i)+i\gamma\geq0$.
 Indeed, we may assume $v(g_i)+i\gamma<0$, and we must show that $g_it^i$ belongs to $P$.
 By construction of $g^{(n)}$ and the assumption that there is no constant term in $a_m$, the coefficient $g_i$ is non-zero only if
 \begin{equation*}
  i\in p^{-n}\mb{N}'\sqcup p^{-n+1}\mb{N}'\dots p^{-1}\mb{N}'\sqcup
   [1,N]_{\mb{N}}.
 \end{equation*}
 We need to check that for $i\in p^{-k}\mb{N}'$ (resp.\ $i\in[1,N]_{\mb{N}}$),
 $g_i\in (\pi^{-1}R')^{1/p^k}$ (resp.\ $g_i\in \pi^{-1}R'$).
 Indeed, for $i\in[1,N]_{\mb{N}}$, since the degree of $a_m$ is bounded by $N$,
 there is no contribution to $g_it^i$ from $(a_m/\pi)^{1/p^m}$ if
 $m>\lceil\log_p(N)\rceil$. This implies that
 \begin{equation*}
  g_i\in \pi^{-1}R+(\pi^{-1}R)+\dots+
   (\pi^{-1}R)^{1/p^{\lceil\log_p(N)\rceil}}
   \subset \pi^{-1}R'.
 \end{equation*}
 We can check similarly for $i\in p^{-k}\mb{N}'$.

 Now, we claim that there exists a (non-empty) {\em finite} set $L$ of (half) lines with positive integral slope in
 $\left[\gamma,\infty\right[\times\mb{R}$ such that $\NP(\alpha t^i)\in\bigcup_{i\geq0}p^{-i}L$ for any $\alpha t^i\in P$.
 Indeed, let $\alpha t^i\in(\pi^{-1}R')^{1/p^k}t^i$ in $P$.
 Then we have
 \begin{equation*}
  \NP(\alpha t^i)(s)=
   p^{-k}\times\NP\bigl((\alpha t^i)^{p^k}\bigr)(s),
 \end{equation*}
 so we only need to construct $L$ such that $\NP(\alpha t^i)\in L$
 for any $\alpha t^i\in P$ and $i\in\mb{N}$.
 We show that the set
 \begin{equation*}
  L:=\Bigl\{\NP(\alpha t^i)\subset\left[\gamma,\infty\right[\times\mb{R}
   \,\big|\,
   \mbox{(*)\,$\alpha t^i\in P$ and $i\in\mb{N}$}\Bigr\}
 \end{equation*}
 is finite. Since $R'$ is noetherian, the set
 $v(R')-v(\pi)+\mb{N}_{\leq N}\cdot\gamma$ is discrete.
 We have
 \begin{equation*}
  \NP(\alpha t^i)(\gamma)\in
   \bigl\{v(R')-v(\pi)+\mb{N}_{\leq N}\gamma\bigr\}\cap
   \left]-\infty,0\right]=:D
 \end{equation*}
 for any $\alpha t^i$ satisfying (*), and the discreteness implies that $D$ is actually a finite set.
 Any line in $L$ has the endpoint in the set $\bigl\{(\gamma,d)\mid d\in D\bigr\}$ and the slope is in $\mb{N}_{\leq N}$.
 There are only $N\times(\#D)$ such lines, and the finiteness of $L$ follows.

 What we have shown implies that there is a set $S\subset\bigcup_{i\geq0}p^{-i}L$ such that $\NP^{(\infty)}(\mbf{a})=\inf_{l\in S}\{l,\mbf{0}\}$,
 where $\mbf{0}:=\left[\gamma,\infty\right[\times\{0\}\subset\left[\gamma,\infty\right[\times\mb{R}$.
 Since $L$ is a finite set, the set
 \begin{equation*}
  Z:=\bigl\{l\cap\mbf{0}
   \mid
   l\in\bigcup p^{-i}L
   \bigr\}
   =
   \bigl\{l\cap\mbf{0}
   \mid l\in L\bigr\}
   \subset\mb{R}^2
 \end{equation*}
 is finite (and non-empty).
 Let $\gamma_{\max}:=\max\{\mr{pr}_1(Z)\}$, where $\mr{pr}_1\colon\mb{R}^2\rightarrow\mb{R}$ is the first projection.
 Then for $s\geq\gamma_{\max}$, we have $\NP^{(\infty)}(\mbf{a})(s)=0$, thus the terminal slope is $0$.
 Since $\NP^{(\infty)}(\mbf{a})(\gamma)<0$, there exists $\gamma_0\in Z$ such that $\NP^{(\infty)}(\mbf{a})(\gamma_0)=0$.
 Take $\gamma_1<\gamma_0$ such that $\left(\left[\gamma_1,\gamma_0\right[\times\{0\}\right)\cap Z=(\gamma_1,0)$ if it exists and $\gamma_1=\gamma$ if not.
 Then $\NP^{(\infty)}(\mbf{a})(\left[\gamma_1,\gamma_0\right])$ is a segment.
 Indeed, if there were a break on $\left]\gamma_1,\gamma_0\right[$, then there would be a line in $\bigcup p^{-i}L$ which has intersection with
 $\left]\gamma_1,\gamma_0\right[\times\{0\}$, which is impossible by the choice.
 Thus there are finitely many breaks on $[\gamma,\gamma_0]$, and the lemma follows.
\end{proof}

\subsection{}
\label{mainpropstabil}
Let $f=\sum_{\ul{k}}a_{\ul{k}}\ul{x}^{\ul{k}}\in\ell\!\left<t\right>\dd{x_1,\dots,x_n}$, and let $\alpha\in\mb{D}_\ell$.
Let $\Lambda$ be the set $\bigl\{\ul{n}\in\mb{N}^n\setminus\{\ul{0}\}\mid p\nmid\ul{n}\bigr\}$.
For $\ul{n}\in\Lambda$, we put $A_{\ul{n}}:=\{a_{p^k\ul{n}}\}_{k\geq0}$, and
\begin{equation*}
 \NP^{(\infty)}_\alpha(f/\pi):=
  \NP^{(\infty)}_\alpha(\{A_{\ul{n}}\}
  _{\ul{n}\in\Lambda}).
\end{equation*}
We wish to study the behavior of this polygon around terminal points ({\it i.e.}\ $\alpha$ is either of type 1 or 4) when $f$ is algebraic.

\begin{prop*}
 Let $f=\sum a_{\ul{k}}x^{\ul{k}}$ be a formal power series in $R\!\left<t\right>\dd{x_1,\dots,x_n}$ which is algebraic
 over $\mr{Frac}\left(\ell\!\left<t\right>[x_1,\dots,x_n]\right)$.
 Assume that $\alpha$ is a terminal point.
 Then there exists $\varepsilon>0$ such that the function $\NP^{(\infty)}_{\alpha}(f/\pi)(s)$ is constant on $s\in[r_\alpha-\varepsilon,r_\alpha]$.
\end{prop*}
The rest of this section is devoted to proving this proposition.

\subsection{}
\label{lemlrrstr}
We start by recalling some basic facts.
Till \ref{proofmainprop}, the noetherian ring $R$ plays no role.
Let $K$ be a field of characteristic $p>0$.
A sequence $\{x_k\}_{k\geq0}$ in $K$ is said to satisfy a {\em linear recurrence relation} (resp.\ {\em reversed linear recurrence relation}),
abbreviated as {\em LRR} (resp.\ {\em rLRR}), if there exist $c_0,\dots,c_n\in\overline{K}$ such that
\begin{equation*}
 c_0x_k+c_1x_{k+1}^{p}+\dots+c_nx_{k+n}^{p^n}=0\qquad
  (\mbox{resp.\ }
  c_0x_k^{p^n}+c_1x_{k+1}^{p^{n-1}}+\dots+c_nx_{k+n}=0)
\end{equation*}
for any $k\geq0$.
We believe that the following lemma is well-known to experts.
A proof can be found, for example, in \cite[Lemma 4]{Kedmis}.

\begin{lem*}
 Let $c_0,\dots,c_n$ be elements in $\overline{K}$ such that
 $c_0,c_n\neq0$.

 (i) The roots of the ``characteristic polynomial''
 $P(X):=c_0X+c_1X^p+\dots+c_nX^{p^n}$
 {\normalfont(}resp.\ $P(X):=c_0X^{p^n}+c_1X^{p^{n-1}}+\dots+c_nX${\normalfont)}
 form a $\mb{F}_p$-vector subspace of $\overline{K}$ of dimension $n$.

 (ii) Let $z_1,\dots, z_n$ be a basis of the vector space of (i). Then a
 sequence $\{x_k\}_{k\geq0}$ in $K$ satisfies LRR {\normalfont(}resp.\ rLRR{\normalfont)}
 with respect to $c_0,\dots,c_n$ if and only if
 there exist $\lambda_1,\dots,\lambda_n$ in $\overline{K}$ such that
 $x_k$ is of the form
 \begin{equation*}
  x_k=z_1\lambda_1^{1/p^k}+\dots+z_n\lambda_n^{1/p^k},\qquad
   \mbox
   {{\normalfont(}resp.\ $x_k=z_1\lambda_1^{p^k}+\dots+z_n\lambda_n^{p^k}${\normalfont)}}.
 \end{equation*}
 for any $k\geq0$.
\end{lem*}

\begin{lem}
 \label{algHadm}
 Let $K$ be a field of characteristic $p>0$.
 Let $f=\sum a_{\ul{k}}\ul{x}^{\ul{k}}$ be a formal power series in
 $K\dd{x_1,\dots,x_n}$ which is algebraic over $K(x_1,\dots,x_n)$.
 Then for any $\ul{n}\neq\ul{0}$, the series
 \begin{equation*}
  h(X):=\sum_{k\geq0}a_{p^k\cdot \ul{n}}X^{p^k}
 \end{equation*}
 in $K\dd{X}$ is algebraic over $K(X)$.
\end{lem}
\begin{proof}
 The argument here is a variant of \cite[Theorem 7.1]{SW}.
 Without loss of generality, we may assume $n_1\neq0$.
 Let $g=\sum_{k\geq0}\ul{x}^{p^k\cdot\ul{n}}$.
 This is algebraic over $K(x_1,\dots,x_n)$ since
 $g-g^p=\ul{x}^{\ul{n}}$. Then
 \begin{equation*}
  f*g=
   \sum_{k\geq0}a_{p^k\cdot \ul{n}}\ul{x}^{p^k\cdot\ul{n}}
   =h(\ul{x}^{\ul{n}}),
 \end{equation*}
 where $*$ denotes the Hadamard product (cf.\ \cite[Definition 2.1]{SW}).
 Since algebraicity is preserved by taking Hadamard product by
 \cite{SW},
 $f*g$ is algebraic over $K(x_1,\dots,x_n)$. This implies that
 $h(\ul{x}^{\ul{n}})=(f*g)(x_1,\dots,x_n)$ is algebraic over
 $K(X,x_2,\dots,x_n)$ where $X=\ul{x}^{\ul{n}}$, since
 $K(x_1,\dots,x_n)\supset K(X,x_2,\dots,x_n)$ is a finite algebraic
 extension by the assumption $n_1\neq0$.
 Note that $X,x_2,\dots,x_n$ do not have algebraic relation over $K$ in
 $K(x_1,\dots,x_n)$, thus $K(X,x_2,\dots,x_n)$ is the polynomial field
 with $n$-indeterminates.
 Then there exist an integer $N>0$ and polynomials $b_i\in K[X,x_2,\dots,x_n]$ for $0\leq i\leq N$
 such that $b_N\neq0$ and $h(X)$ is a solution of the equation
 \begin{equation}
  \label{algeqhapro}
   \tag{$\star$}
  b_Nt^N+b_{N-1}t^{N-1}+\dots+b_0=0.
 \end{equation}
 We write the left side of the equation as
 $\sum_{\ul{k}=(k_2,\dots,k_n)}
 c_{\ul{k}}\cdot(x_2^{k_2}\dots x_n^{k_n})$ where $c_{\ul{k}}\in
 K[X,t]$, a polynomial in $x_2,\dots,x_n$ with coefficients in $K[X,t]$.
 We have $c_{\ul{k}}(X,h(X))=0$ for any $\ul{k}$ in $K(X)$.
 Since the equation (\ref{algeqhapro}) is non-trivial, there exists
 $\ul{k}_0$ such that $c_{\ul{k}_0}\neq0$, and then,
 $h(X)$ is a solution of $c_{\ul{k}_0}(X,t)=0$, as required.
\end{proof}

\begin{lem}
 \label{algrLRR}
 Let $f:=\sum a_i\,X^{p^i}\in K\dd{X}$ be algebraic over $K(X)$.
 Then there exists $N$ such that $\{a_{i+N}\}_{i\geq0}$ is rLRR.
\end{lem}
\begin{proof}
 Since $f$ is algebraic, the set $\{1,f,f^{p},f^{p^2},\dots\}$ is linearly dependent over $K(X)$.
 Thus, there exists $b,b_0,\dots,b_n\in K[X]$ such that
 \begin{equation}
  \label{equdifrec}
   \tag{$\star$}
  b=b_0f+b_1f^p+\dots+b_nf^{p^n}
 \end{equation}
 in $K[X]$.
 Let $M:=\max_{0\leq i\leq n}\{\deg(b),\deg(b_i)\}$.
 Take $N$ such that $p^{n+N-1}>M$.
 Let $c_i$ be the constant term of $b_i$.
 For $i\geq N$, comparing the $X^{p^{i+n}}$-terms of (\ref{equdifrec}), we have
 \begin{equation*}
  0=c_0a_{i+n}+c_1a_{i+n-1}^p+\dots+c_na_{i}^{p^n}.
 \end{equation*}
 This is what we wanted.
\end{proof}

\subsection{}
Let $K$ be a field equipped with a valuation $v$.
Let $P(X)=\sum a_i X^i\in K[X]$, where $a_i\in K$.
For $a\in\mb{R}$, we define $\overline{\NP}(P)(a):=\min_i\bigl\{v(a_i)+ia\bigr\}$.
This function $\overline{\NP}(P)$ is called the {\em complete valuation polygon} of $P$.
Evidently, this function is piecewise affine and convex.
 
\begin{lem*}
 \label{valuextlem}
 Assume that the valuation polygon $\overline{\NP}(P)$ has a break at $r\in\mb{R}$ and the slope changes by $n$.
 \begin{enumerate}
  \item\label{valuextlem-1}
       Let $\overline{K}$ be an algebraic closure of $K$ equipped with an extension $\widetilde{v}$ of $v$.
       There exist $n$ roots $\alpha_i\in\overline{K}$ of the equation $P(t)=0$, counted with multiplicities, such that $\widetilde{v}(\alpha_i)=r$.

  \item\label{valuextlem-2}
       Assume $K$ is complete.
       Let $\alpha\in K$ such that $v(\alpha)\geq r$.
       Then there is an extension $\widetilde{v}$ of $v$ on $K\!\left<\alpha^{-1}X\right>/(P(X))$ so that $\widetilde{v}(X)=r$.
 \end{enumerate}
\end{lem*}
\begin{proof}
 Let us check the first claim.
 We may assume $P$ to be monic, and let $P(X)=\prod(X-\alpha)$ in $\overline{K}$.
 Since the \newton polygon of $X-\alpha$ has the break at
 $r=\widetilde{v}(\alpha)$ and the slope changes by $1$, the lemma
 follows since $\overline{\NP}(P)=\sum\overline{\NP}(X-\alpha)$.

 Let us check the second one.
 Let $L$ be the completion of an algebraic closure of $K$.
 Note that $v$ extends (uniquely) on the algebraic closure by invoking \cite[1.4.9]{Kedbook}.
 It suffices to construct an extension $w$ of $v$ on $L\!\left<\alpha^{-1} X\right>/(P(X))$ so that $w(X)=r$.
 We write $P(X)=\prod_{i\in I}(X-\alpha_i)^{r_i}$ where $\alpha_i\in L$ and $r_i\in\mb{N}$ such that $\alpha_i\neq\alpha_j$ for $i\neq j$.
 Then we have $L\!\left<\alpha^{-1} X\right>/(P(X))\cong\prod_{i\in I'}L\!\left<\alpha^{-1} X\right>/(X-\alpha_i)^{r_i}$,
 where $I'$ is the subset of $i\in I$ such that $v(\alpha_i)\geq v(\alpha)$.
 There exists $i_0\in I'$ such that $v(\alpha_{i_0})=r$ by assumption.
 We may check easily that there exists a unique valuation on
 $L\!\left<\alpha^{-1} X\right>/(X-\alpha_{i_0})^{r_{i_0}}$ extending $v$ by using \cite[2.3.1]{Kedbook}.
\end{proof}

\subsection{}
\label{algbmainlrr}
Let $\alpha$ be a point of the Berkovich disk $\mb{D}_{\ell}:=\mr{Sp}(\ell\!\left<t\right>)$,
and put $r:=-\log(r(\alpha))$ as \ref{NPdefandfirstprop}.
Denote by $v_0$ the Gauss point of $\mb{D}_\ell$.

\begin{lem*}
 Let $\{a_i\}_{i\geq0}$ be an rLRR sequence in $\ell\!\left<t\right>$ and $a$ be an element of $\ell\!\left<t\right>^{\mr{perf}}$.
 Assume that there exists $\lambda>0$ such that $v_0(a_i)>-\lambda$ for any $i$.
 For $n\geq0$, consider the partial sum
 \begin{equation*}
  S_n:=a+\sum_{n\geq i\geq 0}a_i^{1/p^i}.
 \end{equation*}
 Fix a type 4 point $\alpha\in\mb{D}_\ell$, and put $r:=r_\alpha$.
 Then there exists a morphism of $\ell$-affinoid spaces $\psi\colon\mb{X}=\mr{Sp}(A)\rightarrow\mb{D}_\ell$,
 a continuous map $\varphi\colon\left[r-\gamma,r\right]\hookrightarrow\mb{X}$ such that $\gamma>0$ and $\psi\circ\varphi(r)=\alpha$,
 and an element $S$ in $A$ such that the following holds:
 for any $\varepsilon>0$, there exists an integer $N$ such that for any $n>N$ and $\rho\in\left[r-\gamma,r\right]$, we have
 \begin{equation*}
  v_{\varphi(\rho)}(S-S_n)>-\varepsilon.
 \end{equation*}
\end{lem*}
\begin{proof}
 Let $K:=\mr{Frac}(\ell\!\left<t\right>)$, and fix an algebraic closure $\overline{K}$ of $K$.
 We may assume $a=0$. Indeed, let $S'$ be the element in $\overline{K}$ which satisfies the condition for $a=0$.
 Then $S:=S'+a$ satisfies the condition. From now on, we assume that $a=0$.

 Since $\{a_i\}$ is rLRR, there exist $\alpha_j$, $\beta_j$ for $1\leq j\leq m'$ in $\overline{K}$ such that
 \begin{equation*}
  a_i=\alpha_1\beta_1^{p^i}+\dots+\alpha_{m'}\beta_{m'}^{p^i}
 \end{equation*}
 by Lemma \ref{lemlrrstr}.
 Thus, we have $a_i^{1/p^i}=\sum_{j=1}^{m'}\alpha_j^{1/p^i}\beta_j$, which implies that $\{b_i:=a_i^{1/p^i}\}_{i\geq0}$ satisfies LRR.
 Now, let $V\subset\overline{K}$ be the $\mb{F}_p$-subspace spanned by $\{\beta_j\}$.
 We note that $m:=\dim(V)\leq m'$.
 We put
 \begin{equation*}
  P(X):=\prod_{\beta\in V}(X-\beta)=
   d_0X+d_1X^p+\dots+d_{m-1}X^{p^{m-1}}+d_mX^{p^m}.
 \end{equation*}
 Of course, $d_m=1$.
 By Lemma \ref{lemlrrstr} again, we have
 \begin{equation*}
  d_0b_i+d_1b_{i+1}^p+\dots+d_{m-1}b_{i+m-1}^{p^{m-1}}+
   d_mb_{i+m}^{p^m}=0
 \end{equation*}
 for any $i\geq0$.
 We put
 \begin{equation*}
  f(x_1,\dots,x_m):=
   \sum_{0\leq i\leq m}d_i\,
   \bigl(\sum_{1\leq j\leq i} x_j\bigr)^{p^i},
   \qquad
  g(x_1,\dots,x_m):=
  \sum_{0\leq i\leq m}d_i\,
  \bigl(\sum_{i+1\leq j\leq m} x_j\bigr)^{p^i}.
 \end{equation*}
 Since $d_i(x_j+y_j)^{p^i}=d_ix_j^{p^i}+d_iy_j^{p^i}$, we have
 \begin{equation*}
  f(x_1+y_1,x_2+y_2,\dots,x_m+y_m)=
   f(x_1,\dots,x_m)+f(y_1,\dots,y_m).
 \end{equation*}
 Furthermore, we have
 \begin{align*}
  f(b_{n+1},&b_{n+2},\dots,b_{n+m})+
  g(b_{n-m+1},b_{n-m+2},\dots,b_{n})\\
  &=
  \sum_{0\leq i\leq m}d_i\,
  \bigl(\sum_{1\leq j\leq i}b_{n+j}\bigr)^{p^i}
  +
  \sum_{0\leq i\leq m}d_i\,
  \bigl(\sum_{i+1\leq j\leq m} b_{n-m+j}\bigr)^{p^i}\\
  &=
  \sum_{0\leq i\leq m}d_i\,
  \bigl(\sum_{m+1\leq j\leq m+i}b_{n-m+j}\bigr)^{p^i}
  +
  \sum_{0\leq i\leq m}d_i\,
  \bigl(\sum_{i+1\leq j\leq m} b_{n-m+j}\bigr)^{p^i}\\
  &=
  \sum_{0\leq i\leq m}d_i\,
  \bigl(\sum_{i+1\leq j\leq m+i}b_{n-m+j}\bigr)^{p^i}
  =
  \sum_{0\leq i\leq m}d_i\,
  \bigl(\sum_{1\leq j\leq m}b_{n-m+i+j}\bigr)^{p^i}\\
  &=
  \sum_{1\leq j\leq m}\sum_{0\leq i\leq m}
  d_ib_{n-m+i+j}^{p^i}=0.
  \end{align*}
 For $n'>n>0$, put $S_{n,n'}:=S_{n'}-S_{n-1}=\sum_{n\leq i\leq n'}b_i$.
 We put $f(\mathbf{b}_n):=f(b_n,\dots,b_{n+m-1})$.
 Now, the summands of $P(S_{n,n'})$ can be seen as follows.
 \begin{equation*}
  \begin{array}{cccccccccc}
   &&&d_0b_n&d_0b_{n+1}&\cdots&d_0b_{n'-m}&d_0b_{n'-m+1}&\cdots&d_0b_{n'}\\
   &&d_1b_n^p&d_1b^p_{n+1}&d_1b^p_{n+2}&\cdots&d_1b^p_{n'-m+1}&d_1b^p_{n'-m+2}&&\\
   &&\vdots&\vdots&\vdots&&\vdots&\vdots\\
   &&\cdots&\cdots&\cdots&&\cdots&d_{m-1}b_{n'}^{p^{m-1}}\\
   d_mb_{n}^{p^m}&\cdots&d_mb_{n+m-1}^{p^m}&d_mb_{n+m}^{p^m}&
    d_mb_{n+m+1}^{p^m}&\cdots&d_mb_{n'}^{p^m}&&&\\
   \multicolumn{3}{c}{\underbrace{\hspace{4cm}}_{f(b_n,\dots,b_{n+m-1})}}&
    \multicolumn{4}{c}{\underbrace{\hspace{6.5cm}}_{0}}&
    \multicolumn{3}{c}{\underbrace{\hspace{3.5cm}}_{g(b_{n'-m+1},\dots,b_{n'})}}
  \end{array}
 \end{equation*} 
 Thus, we have
 \begin{align}
  \label{impequP}
  \tag{$\star$}
  P(S_{n,n'})=f(b_n,\dots,b_{n+m-1})+g(b_{n'-m+1},\dots,b_{n'})=f(\mathbf{b}_n)-f(\mathbf{b}_{n'+1}),
 \end{align}
 where the second equality holds by the relation above.
 For the convenience of the reader, let us outline the rest of the argument.
 We take a ``sufficiently large'' $n_0$, and define $S$ to be a (suitable) solution of the equation $P(X)=f(\mathbf{b}_{n_0})$.
 For $n\geq n_0$, (\ref{impequP}) implies that $P(S-S_{n_0,n})=f(\mathbf{b}_{n+1})$.
 To estimate $S-S_{n_0,n}$ at some valuation $v$, we may use Lemma \ref{valuextlem}
 by considering the \newton polygon of $P(X)=f(\mathbf{b}_{n+1})$.

 Now, let $\mb{X}':=\mr{Sp}(A')$ be the normalization of $\mr{Sp}\bigl(\ell\!\left<t\right>[d_i]\bigr)$.
 Since $d_i$ are algebraic over $\ell\!\left<t\right>$, the morphism $\psi'\colon\mb{X}'\rightarrow\mb{D}_\ell$ is finite,
 and we may find an extension $\widetilde{\alpha}$ of $\alpha$.
 Since $\alpha$ is of type 4, so is $\widetilde{\alpha}$.
 Take a continuous embedding $\varphi'\colon\left[r-\gamma,r\right]\hookrightarrow\mb{X}'$ with $\varphi'(r)=\widetilde{\alpha}$.
 By \cite[5.12]{BPR}, by decreasing $\gamma$, we may assume that $\psi'\circ\varphi'$ is injective and the image in $\mb{D}_\ell$ is contained in the generic path of
 $\alpha$ (cf.\ \cite[2.2.13]{Ked4}).
 By decreasing $\gamma$ further, we may assume that the function $\rho\mapsto v_{\varphi'(\rho)}(d_i)$ is constant on $\left[r-\gamma,r\right]$ for any $i$,
 since $\widetilde{\alpha}$ is of type 4 and we may take a neighborhood of $\widetilde{\alpha}$ which is isomorphic to an open ball by \cite[4.27]{BPR}.
 Let us consider the \newton polygon of $P(X)$.
 Namely, let $w_a$ ($a\in\mb{R}$) be the valuation on $\mr{Frac}(A')[X]$ such that $w_a(\sum c_iX^i)=\min_i\{v_{\varphi'(\rho)}(c_i)+ia\}$.
 Recall that the complete \newton polygon of $P(X)$ is the function sending $a\in\mb{R}$ to $\overline{\NP}(a):=w_a(P)$.
 By the choice of $\gamma$, $w_a$ does not depend on the choice of $\rho\in[r-\gamma,r]$.
 Note that all the slopes of $\overline{\NP}$ are positive integers since $P(0)=0$.

 \begin{figure}[h]
  \centering
  \setlength{\unitlength}{1bp}%
  \begin{picture}(201.87, 136.06)(0,0)
   \put(0,0){\includegraphics{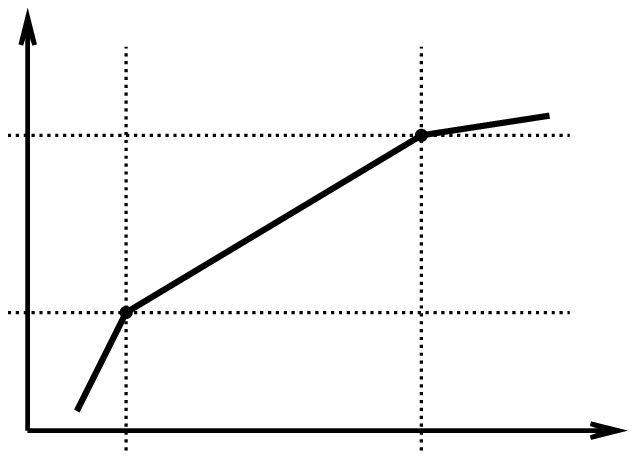}}
   \put(51.38,14.03){\fontsize{8.54}{10.24}\selectfont $R$}
   \put(136.96,13.83){\fontsize{8.54}{10.24}\selectfont $0$}
   \put(7.68,98.11){\fontsize{8.54}{10.24}\selectfont $M$}
   \put(5.67,47.57){\fontsize{8.54}{10.24}\selectfont $M'$}
   \put(175.22,14.84){\fontsize{14.23}{17.07}\selectfont $a$}
   \put(170.39,112.23){\fontsize{11.38}{13.66}\selectfont $\overline{\NP}(a)$}
  \end{picture}%
 \end{figure}

 Let $M:=\overline{\NP}(0)=\min_i\{v_{\varphi'(\rho)}(d_i)\}$.
 Let $R<0$ such that $(R,M':=\overline{\NP}(R))$ is a break point of $\overline{\NP}(a)$ and there is no break point in $\left]R,0\right[$.
 If there is no such a break point, we take any $R<0$.
 Note that $M'<M$.
 Recall that, since $a_n$ is in $\ell\!\left<t\right>$ and $\varphi'$ maps to a sub-segment of the generic path of $\alpha$,
 we have $v_{\varphi'(r-\gamma)}(b_n)\leq v_{\varphi'(\rho)}(b_n)$.

 Let $\varepsilon'>0$, and recall that we are given $\lambda>0$ in the statement of the lemma.
 Take $N$ such that $\varepsilon'>p^{m-N}\lambda$.
 Then for any $n>N$ and $\rho\in\left[r-\gamma,r\right]$, we have
 \begin{equation*}
  p^m v_{\varphi'(\rho)}(b_n)=p^m v_{\psi'\varphi'(\rho)}(b_n)\geq
   p^mv_0(b_n)
   >-p^{m-n}\lambda>-\varepsilon',
 \end{equation*}
 where the first inequality holds since $a_n\in\ell\!\left<t\right>$.
 Thus,
 \begin{equation*}
  v_{\varphi'(\rho)}(f(\mathbf{b}_n))=
   v\Bigl(\sum_{0\leq i\leq m}d_i\,
   \bigl(\sum_{1\leq j\leq i} b_{n+j}\bigr)^{p^i}\Bigr)
   >M-\varepsilon'.
 \end{equation*}
 This implies that we may take large enough $N$ such that $v_{\varphi'(\rho)}(f(\mathbf{b}_n))>M'$ and $v_{\varphi'(\rho)}(b_n)>R$ for any $n\geq N$ and $\rho$.
 Since the characteristic polynomial $P(X)$ remains the same even if we consider the sequence $\{a_{i+N}\}_{i\geq0}$ instead of $\{a_i\}_{i\geq0}$,
 we may and do assume that
 \begin{equation}
  \label{twoassumpmain}\tag{$\star\star$}
  v_{\varphi'(\rho)}(f(\mathbf{b}_n))>M',\qquad
   v_{\varphi'(\rho)}(S_{n})>R
 \end{equation} 
 for any $n\geq0$.

 Now, take $\alpha\in\ell$ such that $v(\alpha)\geq -R>0$, and put $A:=A'\!\left<\alpha X\right>/(P(X)-f(\mbf{b}_0))$ and $\mb{X}:=\mr{Sp}(A)$.
 Since $A\widehat{\otimes}_{A'}\ms{H}(\varphi'(r))\cong\ms{H}(\varphi'(r))\!\left<\alpha X\right>/(P(X)-f(\mbf{b}_0))$
 and since $v_{\varphi'(r)}(f(\mathbf{b}_0))>M'$ by (\ref{twoassumpmain}),
 using Lemma \ref{valuextlem}.\ref{valuextlem-2},
 there exists a valuation $w$ on $A\widehat{\otimes}_{A'}\ms{H}(\varphi'(r))$ extending that of $\ms{H}(\varphi'(r))$ such that $w(X)>R$.
 This defines a point $\beta$ of $\mb{X}$.
 Take any continuous map $\varphi\colon[r-\gamma,r]\rightarrow\mb{X}$ such that $\varphi(r)=\beta$.
 By \cite[5.12]{BPR}, decreasing $\gamma$ if necessary, $\varphi$ becomes a lifting of $\varphi'$ and further $v_{\varphi(\rho)}(X)>R$ for any $\rho\in[r-\gamma,r]$.
 We take $S=X\,(=\alpha^{-1}\cdot(\alpha X)\in A)$, which is a solution of $P(X)=f(\mathbf{b}_0)$ such that $v_{\varphi(\rho)}(S)>R$.
 Let us check that $S$ satisfies the condition of the lemma.
 Let $\varepsilon>0$ as in the claim of the lemma.
 By (\ref{impequP}), we have
 \begin{align*}
  P(S-S_n)&=P(S-S_{0,n})=P(S)-P(S_{0,n})\\
  &=f(\mathbf{b}_0)-
  \bigl(f(\mathbf{b}_0)-f(\mathbf{b}_{n+1})\bigr)\\
  &=f(\mathbf{b}_{n+1}).
 \end{align*}
 We have checked that, for $n\gg0$, $v_{\varphi(\rho)}(f(\mathbf{b}_{n+1}))=v_{\varphi'(\rho)}(f(\mathbf{b}_{n+1}))>\overline{\NP}(-\varepsilon)$ for any $\rho$.
 This implies that any solution $s$ of the equation $P(X)=f(\mathbf{b}_{n+1})$ such that $v_{\varphi(\rho)}(s)>R$
 have the property that $v_{\varphi(\rho)}(s)>-\varepsilon$ by Lemma \ref{valuextlem}.\ref{valuextlem-1}.
 On the other hand, we have $v_{\varphi(\rho)}(S-S_n)\geq \min\{v_{\varphi(\rho)}(S),v_{\varphi(\rho)}(S_n)\}>R$ for any $\rho$ by (\ref{twoassumpmain}).
 Thus $v_{\varphi(\rho)}(S-S_n)>-\varepsilon$ as required.
\end{proof}

\begin{cor}
 \label{conscornp}
 Let $\alpha$ be a type 4 point on $\mb{D}_{\ell}$, and $\{a_i\}$ and $S_n$ be as in the lemma.
 Then there exists $\gamma>0$ such that one of the following conditions holds:
 \begin{itemize}
  \item for any $\varepsilon>0$, there exists $N$ such that $v_{\alpha,\rho}(S_n)>-\varepsilon$ for any $n>N$ and $\rho\in\left[r-\gamma,r\right]$.

  \item there exists $\mu<0$ and $N>0$ such that $v_{\alpha,\rho}(S_n)=\mu$ for any $n>N$ and $\rho\in[r-\gamma,r]$.
 \end{itemize}
\end{cor}
\begin{proof}
 Let $\varphi\colon[r-\gamma,r]\rightarrow\mb{X}$ be as in the lemma.
 We have
 \begin{equation*}
  v_{\alpha,\rho}(S_n)=v_{\varphi(\rho)}(S_n)
   =v_{\varphi(\rho)}(S_n-S+S)\geq
   \min\{v_{\varphi(\rho)}(S_n-S),v_{\varphi(\rho)}(S)\}.
 \end{equation*}
 Since $\alpha$ is a type 4 point, by increasing $\gamma$, we may
 assume that $v_{\varphi(\rho)}(S)$ is constant $\mu$ for any
 $\rho\in\left[r+\gamma,r\right]$.
 The conclusion of the lemma implies that one of the following
 conditions holds:
 \begin{itemize}
 \item if $\mu\geq0$ then for any $\varepsilon>0$,
       there exists
       $N$ such that $v_{\alpha,\rho}(S_n)>-\varepsilon$ for any
       $n>N$.
 \item if $\mu<0$, then there exists $N$ such that
       $v_{\alpha,\rho}(S_n)=v_{\varphi(\rho)}(S)$ for any $n>N$.
 \end{itemize}
 Thus the corollary follows.
\end{proof}

\subsection{Proof of Proposition \ref{mainpropstabil}}\label{proofmainprop}\mbox{}\\
If $\alpha$ is of type 1, the claim follows by Lemma
\ref{defandfinofNP}. Thus, we may assume $\alpha$ to be of type 4.
We use the notation of \ref{mainpropstabil}.
By Lemma \ref{defandfinofNP}, there exists the terminal slope of
$\NP^{(\infty)}(f)$. Thus, it remains to show that for any
$\ul{n}\in\Lambda$, the terminal slope of the \newton polygon
$\NP^{(\infty)}(A_{\ul{n}})$ is $0$.
By Lemma \ref{algHadm}, the function $\sum_{k\geq0}a_{p^k\ul{n}}X^{p^k}$
is algebraic, so $\sum_{k\geq0}(a_{p^k\ul{n}}/\pi)X^{p^k}$ is algebraic
as well.
Thus, by putting $b_i:=a_{p^i\ul{n}}/\pi$, the sequence
$\{b_i\}_{i\geq N}$ satisfies rLRR for $N\gg0$ by Lemma \ref{algrLRR}.
Apply Corollary \ref{conscornp} above for $a:=\sum_{k<N}b^{1/p^k}_k$ and $a_i:=b^{1/p^N}_{i+N}$ to get the desired result.
Indeed, with these choices of $a$ and $a_i$, we have $\NP(S_n)=\NP^{(n+N)}(A_{\ul{n}})$.
In the first case of the corollary, {\em for any} $\rho\in[r-\gamma,r]$, we have $\NP^{(\infty)}(A_{\ul{n}})(\rho)=\lim_{n\rightarrow\infty}\NP(S_n)(\rho)=0$.
Thus, the terminal slope is $0$.
In the second case where $\mu<0$, the corollary implies that $\lim_{n\rightarrow\infty}\NP(S_n)(\rho)=\mu$ for any $\rho\in[r-\gamma,r]$,
and the terminal slope is also $0$, as required.
\qed

\section{Existence of admissible alteration}
\label{sect3}
This section is devoted to proving Theorem \ref{mainthmadm}, which says that given a smooth morphism $X\rightarrow S$ and a rational function $f$ on $X$,
we may take an alteration $S'$ of $S$ so that $f$ is ``admissible''.
We follow \cite{Ked4} for a proof: We define ``local admissibility'' around each point of the Riemann-Zariski space of $S$.
It takes most of the section to prove that this local admissibility can be attained after taking an alteration around each valuation of $S$.
At the very end, we deduce the theorem from the local admissibility by using the quasi-compactness of Riemann-Zariski space.

\begin{dfn}
 \label{dfnadmfun}
 Let $X\rightarrow S$ be a smooth morphism, and $h\in\rat{X}$.
 We say that $h$ is {\em admissible} if the following condition holds:
 there exists an open finite covering $\{V_i\}_{i\in I}$ of $S$ and open subschemes $\{U_i\}_{i\in I}$ of $X$ where $U_i\subset X\times_SV_i$ such that;
 1.\ $\bigcup_{i\in I}U_{i,s}\subset X_s$ is dense for any $s\in S$;
 2.\ there exists a function $f_i\colon U_i\rightarrow\mb{A}^1_S$ which is {\em $S$-separable},
 a regular function $\sigma_i\in\mc{O}_S(V_i)\cap\rat{S}^{\times}$, and $g_i\in\rat{X}$ such that
 \begin{equation*}
  h|_{U_i}+(g_i^p-g_i)\in f_i/\sigma_i+\mc{O}_X(U_i).
 \end{equation*}
\end{dfn}

\begin{rem*}
 The 2nd condition that $f_i$ being $S$-separable can be replaced by $f_i$ being either $S$-separable {\em or $0$}.
 Indeed, assume that $h|_{U_i}+(g_i^p-g_i)\in\mc{O}_{X}(U_i)$.
 By refining $U_i$ and $V_i$, we may assume that there exists an $S$-separable function $\alpha_i\in\mc{O}_{U_i}$.
 Then we have $h|_{U_i}+(g_i^p-g_i)\in \alpha_i/1+\mc{O}_{X}(U_i)$.
\end{rem*}

For a motivation of this definition, one can see Lemma \ref{AsneacyAS}.

\subsection{}
Let $X\rightarrow S$ be a flat morphism and $S'\rightarrow S$ be a
maximally dominant morphism.
Then $X_{S'}\rightarrow X$ is maximally dominant by \cite[Exp.\ II,
1.1.5]{G}. Thus, the pullback homomorphism
$\rat{X}\rightarrow\rat{X_{S'}}$ is defined.

\begin{thm*}
 \label{mainthmadm}
 Let $X\rightarrow S$ be a smooth morphism between $k$-schemes of finite
 type, and take a rational function $h\in\rat{X}$.
 Then there exists an alteration $S'\rightarrow S$ such that the
 pullback of $h$ to $X_{S'}$ is admissible.
\end{thm*}

\subsection{}
From the next paragraph, we use Zariski-Riemann space.
We recall briefly the basic notions just to fix them.
See \cite{V,Ked2,Ked4} for more details.
Let $S$ be an integral scheme of finite type over $k$.
We denote by $\left<S\right>$, or $\left<S\right>_{/k}$ if we want to be more precise, inspired by the notation of \cite[II, \S E]{FK},
the set of valuations $v$ on $k(S)$ such that $v|_k$ is trivial and $v$ is centered on $S$.
This space is called the {\em Zariski-Riemann space} associated to $S$.
When $S$ is proper, $\left<S\right>$ is often denoted by $S_{k(S)/k}$ in the literature.
As a set, we may write $\left<S\right>=\invlim S'$ where $S'$ runs over all the modifications of $S$.
The topology induced by the Zariski topology of $S'$ is called the Zariski topology of $\left<S\right>$.
Let $U\subset S$ be an open subscheme.
Define $\left<S\right>_U:=\invlim S'$ where the projective limit runs over modifications $\rho\colon S'\rightarrow S$ such that $\rho^{-1}(U)\rightarrow U$ is an isomorphism.
For $v\in\left<S\right>$, denote by $k_v$ the residue field, $R_v$ the valuation ring, and $\Gamma_v$ the value semigroup.
The height (or rank) of $v$ is denoted by $\mr{ht}(v)$.
The {\em rational rank} of $v$ denoted by $\mr{rat.rk}(v)$ is $\mr{rk}_{\mb{Q}}(\Gamma_v\otimes\mb{Q})$.
We say that $v$ is a {\em minimal valuation} if $\mr{trdeg}_k(k_v)=0$ and $\mr{ht}(v)=1$,
an {\em Abhyankar valuation} if $\dim(S)=\mr{rat.rk}(v)+\mr{trdeg}_k(k_v)$,
a {\em monomial valuation} if it is Abhyankar and minimal,
a {\em divisorial valuation} if it is Abhyankar, $\mr{rat.rk}(v)=1$ and $\mr{ht}(v)=1$.

\subsection{}
\label{loclhypintr}
Until the beginning of \ref{conclpfmain}, we further assume that the following condition holds, and fix the data:
\begin{quote}
 (*) $X\rightarrow S$ is a smooth morphism between integral affine $k$-schemes of finite type with global coordinates
 $\{x_1,\dots,x_d\}$ over $S$ such that the fiber $X_s$ is geometrically irreducible over $k(s)$ for any $s\in S$,
 and it is endowed with a section $S\rightarrow X$ whose image is the zero locus $\mr{Zero}:=V(x_1,\dots,x_d)$.
 We denote by $\eta$ the generic point of $S$.
\end{quote}
Recall that  we write $f\in\mc{O}_T$ in the sense of $f\in\Gamma(T,\mc{O}_T)$.
For $v\in\left<S\right>$, a morphism of finite type $S'\rightarrow S$ is said to be a {\em local alteration around $v$} if it is generically finite,
$S'$ is affine integral, and $v$ extends to a valuation centered on $S'$.
Note that this condition (*) is stable under changing the base $S$ by a local alteration around $v$.
With these fixed coordinates, by [EGA $0_{\mr{IV}}$, 19.5.4], the
completion $\widehat{X}:=\widehat{X}_{\mr{Zero}}$ is isomorphic {\em canonically} to
$\widehat{\mb{A}}^d_S$, where $S$ is considered as a formal scheme with
the discrete topology.
Let $f\in\mc{O}_{X_\eta}$.
We can write
\begin{equation*}
 f=\sum_{\ul{k}\geq\ul{0}}a_{\ul{k}}\ul{x}^{\ul{k}}
\end{equation*}
where $a_{\ul{k}}\in k(S)$ in $\mc{O}_S\dd{x_1,\dots,x_d}\otimes_{\mc{O}_S}k(S)\subset k(S)\dd{x_1,\dots,x_d}$.
Put
\begin{equation*}
 \Lambda:=\bigl\{\ul{n}\in\mb{N}^d
  \mid
  \mbox{$\ul{n}\neq\ul{0}$ and $p\nmid\ul{n}$}\bigr\}.
\end{equation*}
For a point $v\in\left<S\right>$ of the Riemann-Zariski space, we define
\begin{equation*}
 \NP(f)(v):=
  \min_{\ul{k}\in\mb{N}^d\setminus\{\ul{0}\}}
  \bigl\{
  0,v(a_{\ul{k}})
  \bigr\}
  =
  \min_{\ul{n}\in\Lambda}\,\min_{k}\bigl\{
  0,v(a_{p^k\ul{n}})
  \bigr\}
\end{equation*}
in $\Gamma_v$.
Note that the minimum exists.
Indeed, since $f\in\mc{O}_{X_\eta}$, we may write $f=g/\pi$ with $g\in\mc{O}_X$ and $\pi\in\mc{O}_S$.
Then $a_{\ul{n}}\in\pi^{-1}\mc{O}_S$.
Since $\mc{O}_S$ is noetherian, the semigroup $v(\mc{O}_S)\subset\Gamma_v$ is well-ordered (cf.\ \ref{discvalsemgr}),
so is $v(\pi^{-1}\mc{O}_S)=v(\mc{O}_S)-v(\pi)$, and the minimum exists.

Now, {\em assume that $\mr{ht}(v)=1$}. For $\ul{n}\in\Lambda$, let
\begin{equation*}
 f^{(m)}_{\ul{n}}:=\sum^m_{k=0}(a_{p^k\ul{n}})^{1/p^k}.
\end{equation*}
This sum is defined in
$\pi^{-1}A:=\sum_{n\geq0}(\pi^{-1}\mc{O}_S)^{1/p^n}$.
Note that $v$ extends uniquely to the perfection $k(S)^{1/p^{\infty}}$.
Since $v$ is of height $1$, for any $\epsilon<0$ in $\Gamma_v\subset\mb{R}$, there exists an integer $N>0$ such that $\epsilon<-v(\pi)/p^N$.
This implies that
\begin{equation}
 \label{cutsmalpa}
 \bigl\{\lambda\in v(\pi^{-1}A)\mid \lambda\leq\epsilon\bigr\}
  =
  \bigl\{\lambda\in v\bigl(\TXsum_{n\leq N}(\pi^{-1}\mc{O}_S)^{1/p^n}
  \bigr)\mid \lambda\leq\epsilon\bigr\}.
\end{equation}
The set on the right is well-ordered because $\mc{O}_S^{1/p^N}$ is noetherian.
This implies that the set $J:=\bigl\{\lambda\in v(\pi^{-1}A)\mid \lambda\leq0\bigr\}$ is well-ordered. 
For $\ul{n}\in\Lambda$, assume $\lim_{m\rightarrow\infty}\min\bigl\{0,v(f^{(m)}_{\ul{n}})\bigr\}\neq0$.
Then $\liminf_{m\rightarrow\infty}v(f^{(m)}_{\ul{n}})<0$.
For any $\liminf v(\dots)<\epsilon<0$, there exists $M>\log_p(-v(\pi)/\epsilon)$ such that $v(f^{(M)}_{\ul{n}})<\epsilon$.
Then $v(f^{(M)}_{\ul{n}})=v(f^{(n)}_{\ul{n}})$ for any $n\geq M$, and
$\lim_{m\rightarrow\infty}\min\bigl\{0,v(f^{(m)}_{\ul{n}})\bigr\}=v(f^{(M)}_{\ul{n}})\in J$.
Since $J$ is well-ordered, this implies that the following definition makes sense:
\begin{equation*}
 \NP^{(\infty)}(f)(v):=
  \min_{\ul{n}\in\Lambda}\Bigl\{
  \lim_{m\rightarrow\infty}\min\bigl\{0,v(f^{(m)}_{\ul{n}})\bigr\}
  \Bigr\}.
\end{equation*}
Note that, by definition, $\NP(f)(v),\NP^{(\infty)}(f)(v)\leq0$.
For a dominant morphism $\alpha\colon S'\rightarrow S$ of integral schemes, $\NP$ and $\NP^{(\infty)}$ are invariant under base change by $\alpha$.
Namely, for any extension $v'$ to $S'$ of $v\in\left<S\right>$, we have the equalities
\begin{equation}
 \label{invarpulnp}
 \NP(\alpha^*f)(v')=\NP(f)(v),\qquad
  \NP^{(\infty)}(\alpha^*f)(v')=\NP^{(\infty)}(f)(v).
\end{equation}
Let $L/k(S)$ be a finite extension, $w$ be an extension of $v$, and $f'\in\mc{O}_{X_L}$, where $X_L:=X\times_S\mr{Spec}(L)$.
Then we can take a local alteration $S'\rightarrow S$ such that $k(S')\cong L$ and $w$ is centered on $S'$.
Then $f'\in\mc{O}_{X_{\eta'}}\cong\mc{O}_{X_L}$, where $\eta'$ is the generic point of $S'$,
and the values $\NP(f')(w)$, $\NP^{(\infty)}(f')(w)$ are defined.
By (\ref{invarpulnp}), these values only depend on $L$ and $w$, and not on specific ``model'' $S'$.
We use these values without referring to $S'$.

\begin{dfn}
 Take a point $v\in\left<S\right>$, and a function
 $h\in\mc{O}_{X_\eta}$.
 \begin{itemize}
  \item An {\em \ally} of $h$ is a rational function of the form $h+(g^p-g)$ with $g\in\mc{O}_{X_\eta}$.
	An \ally $h'$ of $h$ is said to be {\em \good on $S$ at $v$} if it belongs to $f/\sigma+\mc{O}_{X,z}$,
	where $z$ is the center of $v$ on $S$, $\mc{O}_{X,z}:=\mc{O}_X\otimes_{\mc{O}_S}\mc{O}_{S,z}$,
	$\sigma\in\mc{O}_{S}\setminus\{0\}$,
	and $f\in\mc{O}_{X}$ which is either an $S$-separable function or in $\mc{O}_S$.
	Such a $g$ is called a {\em mollifier}.
	For an open subscheme $U\subset S$, a mollifier $g$ is said to {\em provide a \good \ally of $h$ on $U$}
	if $h+(g^p-g)$ is a \good \ally of $h$ at any valuation $v$ centered on $U$.

  \item We say that {\em $h$ is locally admissible at $v$} if there exists
	an integral scheme $S'$ such that $k(S')=k(S)$, $v$ is centered on $S'$, and there exists a \good \ally of $h$ on $S'$ at $v$.

  \item Assume $\mr{ht}(v)=1$.
        We say that {\em $h$ is weakly admissible at $v$ on $k(S)$} if there exists a function
	$g\in\mc{O}_{X_\eta}=\mc{O}_X\otimes_{\mc{O}_S}k(S)$ such that $\NP^{(\infty)}(h)(v)=\NP(h+(g^p-g))(v)$.
	Such a $g$ is called a {\em weak mollifier}.
 \end{itemize}
\end{dfn}
Note that if $h$ is good at $v$ which is centered at $z\in S$, then there exists an open neighborhood subscheme $U$ of $z$ over which $h$ is good.
Indeed, it suffices to show that if $f$ is separable at $z$, then $f$ is separable on an open neighborhood of $z$.
Since we are assuming $X\rightarrow S$ to be geometrically irreducible, and thus equidimensional,
this follows by Lemma \ref{lemonprofun}.\ref{lemonprofun-pwsmoook} and [EGA IV, 13.1.3].
Even though we do not use, we remark that if $h$ admits a \good \ally, then it is locally weakly admissible.
This can be seen from Lemma \ref{propnpval} below.
The converse is ``almost'' true by Lemma \ref{goodmodelze}.

\begin{lem}
 \label{basiinequnp}
 For any $f,g\in\mc{O}_{X_\eta}$, we have
 \begin{equation*}
  \NP^{(\infty)}\bigl(f + (g^p-g)\bigr)(v)=
   \NP^{(\infty)}(f)(v),\qquad
  \NP(f)(v)\leq\NP^{(\infty)}(f)(v).
 \end{equation*}
\end{lem}
\begin{proof}
 We may take $\pi\in\mc{O}_S\setminus\{0\}$ such that $f,g\in\pi^{-1}\mc{O}_S$.
 Let $\mathbf{a}:=\{a_k\}_{k\geq0}$ be a sequence in $\pi^{-1}\mc{O}_S$, and put
 \begin{equation*}
  v(\mathbf{a}):=\min_k\bigl\{0,v(a_k)\bigr\},
   \quad
   v^{(\infty)}(\mathbf{a}):=
   \min\Bigl\{0,\lim_{m\rightarrow\infty}v\Bigl(\sum_{k=0}^ma_k^{1/p^k}\Bigr)\Bigr\}.
 \end{equation*}
 Now, let $\mathbf{b}=\{b_k\}_{k\geq0}$ be a sequence in $\pi^{-1}\mc{O}_S$. Put $\mathbf{b}^p:=\{b_{k-1}^p\}_{k\geq0}$ ($b_{-1}:=0$).
 Then we claim that
 \begin{equation*}
  v^{(\infty)}(\mathbf{a}+(\mathbf{b}^p-\mathbf{b}))
   =
   v^{(\infty)}(\mathbf{a}),\qquad
   v(\mathbf{a})
   \leq
   v^{(\infty)}(\mathbf{a}).
 \end{equation*}
 If these claims hold, the lemma follows.
 Indeed, in order to show the lemma, we take $\{a_k\}$ (resp.\ $\{b_k\}$) to be the coefficients of
 $\ul{x}^{p^k\ul{n}}$ for each $\ul{n}\in\Lambda$ in the expansion of $f$ (resp.\ $g$).
 Now, to check the first claim, we may compute
 \begin{align*}
  (a_0-b_0)+\sum_{k=1}^m\bigl(
  a_k+(b_{k-1}^p-b_k)
  \bigr)^{1/p^k}
  =
  \bigl(\sum_{k=0}^ma_k^{1/p^k}\bigr)-b_m^{1/p^m}.
 \end{align*}
 Since $v(b_m)\geq -v(\pi)$, we have $\liminf_{m\rightarrow\infty}v(b_m^{1/p^m})\geq0$, and the first equality follows.
 The second inequality is easy to check.
\end{proof}

\subsection{}
\label{prepringa}
Let $f\in\mc{O}_X$ and $\pi\in\mc{O}_S$, and let $v\in\left<S\right>$ whose center is $z\in S$.
We assume that $k$ is perfect and $k_v=k(z)$.
We put
\begin{equation*}
 A:=\pi\sum_{n\geq0}(\pi^{-1}\mc{O}_S)^{1/p^n}
  \supset\mc{O}_S,\qquad
  J:=\bigl\{\lambda\in v(\pi^{-1}A)\mid \lambda\leq 0\bigr\}\subset\Gamma_v.
\end{equation*}
Then $A$ is an $\mc{O}_S$-module contained in $\mc{O}_S^{1/p^\infty}$.
Recall from \ref{loclhypintr} that $J$ is well-ordered.
The filtration defined by $v$ on $\mc{O}_S^{1/p^\infty}$ induces a filtration on $A$.
For a submodule $B\subset k(S)^{1/p^\infty}$ and $\lambda\in\Gamma_v$, we put $\mr{Fil}^\lambda_v(B):=\bigl\{f\in B\mid v(f)\geq\lambda\bigr\}$,
$\mr{Fil}^{\lambda+}_v(B):=\bigl\{f\in B\mid v(f)>\lambda\bigr\}$, and $\mr{gr}^{\lambda}_v(B):=B^{\lambda}/B^{\lambda+}$.
We often denote $\mr{Fil}^{\lambda(+)}_v B$, $\mr{gr}^\lambda_v B$ by $B^{\lambda(+)}$, $\mr{gr}^\lambda B$ if no confusion may arise.
We have inclusions
\begin{equation*}
 \mr{gr}^\lambda(\mc{O}_S)\subset\mr{gr}^\lambda(A)\subset
  \mr{gr}^\lambda(\mc{O}_S^{1/p^\infty}).
\end{equation*}
By assumption on $v$, we have
$k_v:=\mr{gr}^0(k(S))=\mr{gr}^0(\mc{O}_S)=:k(z)$.
Since $\mc{O}_X$ is flat over $\mc{O}_S$ we may define a filtration on
$\mc{O}_X\otimes \pi^{-1}A$ by
$(\mc{O}_X\otimes \pi^{-1}A)^\lambda:=
\mc{O}_X\otimes(\pi^{-1}A)^\lambda$.
The flatness moreover implies that
$\mr{gr}(\mc{O}_X\otimes A)\cong\mc{O}_X\otimes\mr{gr}(A)$.

\begin{lem}
 Let $M$ be an $\mc{O}_S$-module, and $N\subset M$ be a submodule.
 Put $\widehat{\mc{O}}_X:=\mc{O}_{\widehat{X}_{\mr{Zero}}}$.
 Then we have the following equality  in
 $M\otimes_{\mc{O}_S}\widehat{\mc{O}}_X$:
 \begin{equation*}
  (N\otimes_{\mc{O}_S}\widehat{\mc{O}}_X)\cap
   (M\otimes_{\mc{O}_S}\mc{O}_X)=
   N\otimes_{\mc{O}_S}\mc{O}_{X}.
 \end{equation*}
\end{lem}
\begin{proof}
 Let $J\subset\mc{O}_X$ be the ideal defining the closed subscheme $Z:=\mr{Zero}$ in \ref{loclhypintr} (*).
 We put $\mc{O}_{X,Z}:=(1+J)^{-1}\mc{O}_X$.
 The pair $(\mc{O}_{X,Z},J\mc{O}_{X,Z})$ is a Zariskian pair \cite[I, B.1 (a)]{FK}, and thus, the completion homomorphism
 $\mc{O}_{X,Z}\rightarrow\widehat{\mc{O}}_{X,Z}\cong\widehat{\mc{O}}_X$ is faithfully flat.
 The fully faithfulness implies that the following diagram on the right
 is Cartesian by [EGA $0_{\mr{I}}$, 6.6.3]:
 \begin{equation*}
  \xymatrix{
   N\otimes_{\mc{O}_S}\mc{O}_{X}
   \ar@{^{(}->}[r]\ar@{^{(}->}[d]&
   N\otimes_{\mc{O}_S}\mc{O}_{X,Z}
   \ar@{^{(}->}[d]\\
  M\otimes\mc{O}_X\ar@{^{(}->}[r]&
   M\otimes\mc{O}_{X,Z},
   }
   \qquad
   \xymatrix{
   N\otimes_{\mc{O}_S}\mc{O}_{X,Z}
   \ar@{^{(}->}[r]\ar@{^{(}->}[d]\ar@{}[rd]|\square&
   N\otimes_{\mc{O}_S}\widehat{\mc{O}}_{X,Z}
   \ar@{^{(}->}[d]\\
  M\otimes\mc{O}_{X,Z}\ar@{^{(}->}[r]&
   M\otimes\widehat{\mc{O}}_{X,Z}.
   }
 \end{equation*}
 Thus, it remains to show that the left diagram above is Cartesian.
 By the flatness of $\mc{O}_X$ over $\mc{O}_S$, this is equivalent to showing that the induced homomorphism
 $\alpha\colon L:=(M/N)\otimes\mc{O}_X\rightarrow(M/N)\otimes\mc{O}_{X,Z}$ is injective.
 Since $\alpha$ is a localization homomorphism by the multiplicative system $1+J$, it suffices to show that,
 for any $f\in 1+J$, we have $\mr{Ass}(L)\subset D(f)$ in $X=\mr{Spec}(\mc{O}_X)$ by [EGA IV, 3.1.9].
 If there exist $f$, $\mf{p}$ such that $\mf{p}\in\mr{Ass}(L)\cap V(f)$, then $V(J)\cap V(\mf{p})\subset V(f)\cap V(J)=\emptyset$.
 Thus, it suffices to check that, for any $\mf{p}\in\mr{Ass}(L)$,
 the intersection $V(J)\cap V(\mf{p})\neq\emptyset$.
 Since $\mc{O}_X$ is flat over $\mc{O}_S$, the associated prime of $L$ is equal to
 $\bigcup_{\mf{p}\in\mr{Ass}(M/N)}\mr{Ass}\bigl(k(\mf{p})\otimes\mc{O}_X\bigr)$ by [EGA IV, 3.3.1].
 Since each fiber of $\mc{O}_X$ is geometrically irreducible and is smooth over $\mc{O}_S$, $k(\mf{p})\otimes\mc{O}_X$ is integral, and
 thus $\mr{Ass}\bigl(k(\mf{p})\otimes\mc{O}_X\bigr)$ is equal to the generic point of the fiber $X\otimes_{\mc{O}_S}k(\mf{p})$.
 Since the morphism $Z\rightarrow S$ is an isomorphism, the intersection is not empty, and the claim follows.
\end{proof}

\begin{cor}
 \label{intergoff}
 Let $M$ be a finitely generated $\mc{O}_S$-module, and $N\subset M$ be a submodule.
 Let $h\in M\otimes_{\mc{O}_S}\mc{O}_X$, and write $h=\sum a_{\ul{k}}\ul{x}^{\ul{k}}$ in $M\otimes\widehat{\mc{O}}_X\cong M\dd{x_1,\dots,x_d}$.
 If $a_{\ul{k}}\in N$ for any $\ul{k}$, then $h\in N\otimes_{\mc{O}_S}\mc{O}_X$.
 In particular,  if $k$ is perfect and we are in the situation of {\normalfont\ref{prepringa}},
 $h\in\mc{O}_X\otimes\pi^{-1}A$ belongs to $(\mc{O}_X\otimes\pi^{-1}A)^{\lambda}$ if and only if $a_{\ul{k}}\in(\pi^{-1}A)^\lambda$ for any $\ul{k}$.
\end{cor}
\begin{proof}
 Since $N$ is finitely generated, we have
 \begin{equation*}
  N\otimes_{\mc{O}_S}\widehat{\mc{O}}_X
   \cong
   \Bigl\{\sum a_{\ul{k}}\ul{x}^{\ul{k}}\mid
   a_{\ul{k}}\in N\Bigr\}
   \subset M\dd{x_1,\dots,x_d},
 \end{equation*}
 and the first claim follows by the lemma.
 To check the second claim, ``if'' part is the only problem. The $\mc{O}_S$-module $M:=\pi^{-1}\sum_{n\leq a}(\pi^{-1}\mc{O}_S)^{1/p^n}$
 is finitely generated for any $a$ since $k$ is assumed to be perfect.
 Apply the first claim for $N:=\bigl(\pi^{-1}\sum_{n\leq a}(\pi^{-1}\mc{O}_S)^{1/p^n}\bigr)^{\lambda}$, and take the inductive limit over $a$ to conclude.
\end{proof}

\subsection{}
\label{propnpval}
Let $k$ be perfect and $v\in\left<S\right>$.
This corollary implies that for $h\in \mc{O}_X\otimes \pi^{-1}A$, $\NP(h)(v)$ is the largest $\lambda\in J$ such that $h\in(\mc{O}_X\otimes \pi^{-1}A)^\lambda$.
The image of $h$ in $\mr{gr}^{\NP(h)(v)}(\mc{O}_X\otimes\pi^{-1}A)$ is called the {\em symbol} denoted by $\sigma_v(h)$ or simply by $\sigma(h)$.
By construction, $\sigma(h)\neq0$ if $\lambda=\NP(h)(v)<0$.
An element in $\mc{O}_X\otimes\mr{gr}^\lambda(\pi^{-1}A)$ is said to be {\em separable} if the image of the element by the homomorphism
\begin{equation*}
 \mc{O}_X\otimes\mr{gr}^\lambda(\pi^{-1}A)\hookrightarrow
  \mc{O}_X\otimes\mr{gr}^\lambda\bigl(k(S)^{1/p^\infty}\bigr)
  \cong
  \mc{O}_X\otimes\mr{gr}^0\bigl(k(S)^{1/p^\infty}\bigr)
  \cong
  \mc{O}_X\otimes_k k_v^{1/p^\infty}
\end{equation*}
is ($\mr{Spec}(k_v^{1/p^{\infty}})$-)separable.
Note that the first isomorphism between $k_v^{1/p^\infty}$-vector spaces $\mr{gr}^\lambda$ and $\mr{gr}^0$
is given once we fix $f\in k(S)^{1/p^{\infty}}$ such that $v(f)=\lambda$.
In particular, it is not canonical, but the separability does not depend on the choice.

\begin{lem*}
 We keep assuming $k$ to be perfect.
 Let $h\in\mc{O}_X\otimes\pi^{-1}A$ and assume $\NP(h)(v)<0$.
 \begin{enumerate}
  \item If $\sigma_v(h)$ is separable, then there exists a power $q$ of $p$ and a local modification $S'\rightarrow S^{1/q}$
	such that $h$ can be written as $h'/\pi'$ where $h'$ is a separable function on $\mc{O}_{X_{S'}}$ and $\pi'\in\mc{O}_{S'}$.
	If $h\in\pi^{-1}\mc{O}_X$, then we may take $q=1$.
       
  \item\label{propnpval-2}
       Assume $\mr{ht}(v)=1$ and $\NP^{(\infty)}(h)(v)<0$.
       Then the symbol $\sigma_v(h)$ is separable if and only if $h$ is weakly admissible at $v$,
       namely $\NP^{(\infty)}(h)(v)=\NP(h)(v)$.
 \end{enumerate}
\end{lem*}
\begin{proof}
 Let us show the first one. Let $\lambda:=\NP(h)(v)<0$.
 We may take $\pi'\in k(S)^{1/p^\infty}$ such that $v(\pi')=-\lambda>0$.
 Take $S'$ so that $\pi'$ and $h':=\pi'\cdot h$ belong to $\mc{O}_{S'}$.
 Then $h':=\pi'\cdot h$ is separable if we shrink $S'$ around the center of $v$, thus the claim follows.

 Let us show the second claim. Write $h=\sum a_{\ul{n}}\ul{x}^{\ul{n}}$.
 Let $\Sigma$ be the set of $\ul{n}$ such that $v(a_{\ul{n}})=\NP(h)(v)$.
 The symbol $\sigma(h)$ is separable if and only if there exists $\ul{n}\in\Sigma$ such that $p\nmid\ul{n}$ by Lemma \ref{intFrob}.
 Assume there exists $\ul{n}_0\in\Sigma$ such that $p\nmid\ul{n}_0$.
 Then for any $i>0$,
 \begin{equation*}
  p^{-i}\cdot v(a_{p^i\ul{n}_0})>v(a_{\ul{n}_0}).
 \end{equation*}
 Indeed, if $v(a_{p^i\ul{n}_0})\geq0$, the claim follows since $v(a_{\ul{n}_0})<0$, and otherwise,
 $p^{-i}\cdot v(a_{p^i\ul{n}_0})>v(a_{p^i\ul{n}_0})\geq v(a_{\ul{n}_0})$.
 This implies that $v\bigl(\sum_{i=0}^N (a_{p^i\ul{n}_0})^{1/p^i}\bigr)=v(a_{\ul{n}_0})$ for any $N$.
 Thus, $\NP^{(\infty)}(h)(v)=\NP(h)(v)$.

 Assume, on the other hand, that any element of $\Sigma$ is divisible by $p$.
 Take $\ul{n}\in\Lambda$.
 If $\bigl\{p^i\ul{n}\bigr\}_{i\geq0}\cap\Sigma=\emptyset$, then $v(a_{p^i\ul{n}})>\NP(h)(v)$ for any $i$.
 If $p^a\ul{n}\in\bigl\{p^i\ul{n}\bigr\}_{i\geq0}\cap\Sigma$, then $a>0$ by assumption.
 In this case, for any $k$, we have
 \begin{align*}
  v\bigl(\sum_{i=0}^k(a_{p^i\ul{n}})^{1/p^i}\bigr)
  >
  v(a_{p^a\ul{n}}).
 \end{align*}
 Since $\NP^{(\infty)}(h)(v)\in J$, which is well-ordered, we have $\NP^{(\infty)}(h)(v)>\NP(h)(v)$.
\end{proof}

\begin{lem}
 \label{redulemav}
 Assume that $k$ is algebraically closed and $v$ is minimal.
 For any $h\in\mc{O}_X\otimes \pi^{-1}A$ such that $\NP(h)<\NP^{(\infty)}(h)$,
 there exists $g\in\mc{O}_X\otimes (\pi^{-1}A)^{1/p}\subset \mc{O}_X\otimes \pi^{-1}A$ such that $\NP(h)<\NP\bigl(h + (g^p-g) \bigr)$.
\end{lem}
\begin{proof}
 Since $\mr{trdeg}_k(k_v)=0$ and $k$ is assumed to be algebraically closed, $\mr{gr}^0(\mc{O}_S)$ is equal to $k$.
 Moreover,
 \begin{equation*}
  \mr{gr}\bigl(k(S)^{1/p^\infty}\bigr)\cong k[\Gamma_v^{(p)}],
 \end{equation*}
 where $\Gamma_v^{(p)}:=\bigcup_{i>0}p^{-i}\Gamma_v$.
 Caution that this is a non-canonical isomorphism of $k$-modules.
 Since $\mr{gr}^\lambda(A)$ is a $\mr{gr}^0(\mc{O}_S)\cong k$-vector subspace of $\mr{gr}^{\lambda}(k(S)^{1/p^{\infty}})$,
 we get that $\mr{gr}^\lambda(A)$ is either $k$-module of rank $0$ or $1$.
 Since the map $\mr{gr}^\lambda(\pi^{-1}A)\rightarrow\mr{gr}^{\lambda+v(\pi)}(A)$ sending $\alpha$ to $\pi\alpha$ is an isomorphism,
 the same is true for $\mr{gr}^\lambda(\pi^{-1}A)$.
 In the same way, $\mr{gr}^\lambda\bigl((\pi^{-1}A)^{1/q}\bigr)$ is a $k$-vector space of dimension either $0$ or $1$.
 Now, take $0>\lambda\geq-v(\pi)$ such that $\mr{gr}^\lambda(\pi^{-1}A)$ is of rank $1$.
 Then there exists $\alpha\in \pi^{-1}A$ such that $v(\alpha)=\lambda$.
 Let $q$ be a power of $p$.
 Because valuations are multiplicative, the element $\sigma(\alpha^{1/q})$ does not vanish in $\mr{gr}^{\lambda/q}\bigl((\pi^{-1}A)^{1/q}\bigr)$.
 Thus the latter $k$-vector space is also of dimension $1$.
 This implies that we have the following commutative diagram whose vertical homomorphisms are isomorphic:
 \begin{equation}
  \tag{$\star$}
  \label{diagqpowcom}
  \vcenter{
  \xymatrix@C=60pt{
   \mc{O}_X\otimes\mr{gr}^{\lambda/q}\bigl((\pi^{-1}A)^{1/q}\bigr)
   \ar[r]^-{(-)^q}\ar[d]_{\sim}&
   \mc{O}_X\otimes\mr{gr}^{\lambda}(\pi^{-1}A)\ar[d]^{\sim}\\
  \mc{O}_X\otimes\mr{gr}^{\lambda/q}\bigl(k(S)^{1/p^\infty}\bigr)
   \ar[r]_-{(-)^q}&
   \mc{O}_X\otimes\mr{gr}^{\lambda}\bigl(k(S)^{1/p^\infty}\bigr).
   }
   }
 \end{equation}
 Here, the horizontal homomorphisms are $q$-th power map.
 
 Now, let us go back to the proof.
 By Lemma \ref{propnpval}.\ref{propnpval-2}, $\sigma(h)$ is not separable.
 This implies that $\sigma(h)$ is in the image of the $q$-th power map in $\mc{O}_X\otimes\mr{gr}(k(S)^{1/p^{\infty}})$
 for some power $q>1$ of $p$ by Lemma \ref{lemonprofun}.\ref{lemonprofun-smgenpinsep}.
 Since the vertical homomorphisms of (\ref{diagqpowcom}) are isomorphisms,
 there exists a separable function $g'\in\mr{gr}\bigl(\mc{O}_X\otimes (\pi^{-1}A)^{1/q}\bigr)$ such that $g'^q=-\sigma(h)$ where $q>1$ is a power of $p$.
 Now, we can take an extension $g\in\mc{O}_X\otimes(\pi^{-1}A)^{1/q}$ such that $\sigma(g)=g'$.
 By the choice of $g'$, we have $v(h+(g^q-g))>v(h)=0$.
 This implies that
 \begin{equation*}
  \NP\bigl(h + (g^q-g)\bigr)(v)>\NP(h)(v).
 \end{equation*}
 Finally by replacing $g$ by $g+g^p+\dots+g^{q/p}$, we may take $q=p$, and the lemma follows. 
\end{proof}

\begin{lem}
 \label{npnegok}
 Let $f\in\mc{O}_X$, $\pi\in\mc{O}_S\setminus\{0\}$, and let $v$ be a minimal valuation such that $\NP^{(\infty)}(f/\pi)(v)<0$.
 Then there exists a finite extension $L$ of $k(S)$ such that $f/\pi$ is weakly admissible at any extension of $v$ to $L$.
\end{lem}
\begin{proof}
 We may assume that $k$ is algebraically closed.
 Since $v$ is of height $1$, $\Gamma_v\subset\mb{R}$.
 For any $-v(\pi)<\varepsilon<0$, the set
 $I_\varepsilon:=v(\pi^{-1}A)\cap\left]-\infty,\varepsilon\right]\subset J$
 is finite by Lemma \ref{discvalsemgr} and (\ref{cutsmalpa}).
 Set $I:=I_{\NP^{(\infty)}(v)}$.
 By Lemma \ref{basiinequnp}, we have $\NP\bigl(f/\pi + (g^p-g)\bigr)(v)\leq\NP^{(\infty)}(v)$ for any $g\in\mc{O}_X\otimes_{\mc{O}_S}\pi^{-1}A$,
 we have $\NP\bigl(f/\pi + (g^p-g)\bigr)(v)\in I$.
 Now, apply Lemma \ref{redulemav} with $h=f/\pi$, and we can find a function $g\in\mc{O}_X\otimes(\pi^{-1}A)^{1/p}$
 such that $\NP\bigl(f/\pi+(g^p-g)\bigr)>\NP(f/\pi)$.
 Note that $h + (g^p-g)$ is still in $\mc{O}_X\otimes \pi^{-1}A$,
 and thus $\NP\bigl(f/\pi+(g^p-g)\bigr)\in I$. We again apply Lemma \ref{redulemav}, this time, with $h=f/\pi+(g^p-g)$.
 We repeat this procedure. Since $I$ is a finite set, this procedure terminates after finitely many times,
 and we find $g'\in\mc{O}_X\otimes\pi^{-1}A$ such that $\NP^{(\infty)}(f/\pi)=\NP\bigl(f/\pi + (g'^p-g') \bigr)$.
 We take $S':=S^{1/q}$ so that $\pi g'\in\mc{O}_{X_{S'}}$ where $q$ is a power of $p$.
\end{proof}

\begin{lem}
 \label{goodmodelze}
 If $h\in\mc{O}_{X_\eta}$ is weakly admissible at a valuation $v$ of height $1$,
 then there exists a finite extension $L$ of $k(S)$ so that $h$ is admissible at any extension of $v$.
\end{lem}
\begin{proof}
 Assume $\NP(h)(v)=0$.
 We may write $h=f/\pi$ where $f\in\mc{O}_X$, $\pi\in\mc{O}_S$.
 Write $f=\sum a_{\ul{k}}\ul{x}^{\ul{k}}$.
 By assumption, $v(a_{\ul{k}})\geq v(\pi)$.
 Let $I\subset\mc{O}_{S}$ be the ideal generated by $\{a_{\ul{k}}\}$.
 For any $b\in I$, we have $v(b)\geq v(\pi)$.
 Since $\mc{O}_S$ is noetherian, there exist finitely many generators $c_i$ of $I$.
 We take a modification $S'$ of $S$ so that $c_i/\pi$ are regular functions around the center of $v$, which is possible since $v(c_i/\pi)\geq0$.
 Then the claim follows since $h\in\mc{O}_{X_{S'}}$ by Corollary \ref{intergoff}. Note that in this case, we can take the local alteration to be a ``local modification''.
 Assume, now, that $\NP(h)(v)<0$.
 Let $l/k$ be an extension of fields.
 For a $k$-scheme $Z$ and a morphism $a\colon Z\rightarrow\mb{A}^1$, if $a\otimes_kl$ is smooth, then $a$ is smooth by [EGA IV, 17.7.3].
 Thus, $h$ is a \good \ally if and only if $h\otimes_kl$ is a \good \ally, and we may assume that $k$ is perfect.
 Then the claim follows by Lemma \ref{propnpval}.
\end{proof}

\subsection{}
By combining what we have proven so far, when $v$ is minimal and $\NP^{(\infty)}(h)(v)<0$,
there exists a local alteration over which $h$ admits a \good \ally around an extension of $v$.
The next step is to treat the case where $v$ is minimal and $\NP^{(\infty)}(h)(v)=0$.
This case requires more serious analysis, and we use the induction on the transcendence defect of $v$ following \cite{Ked4}.
The base of the induction is the case where $v$ is monomial.
In fact, we can show the weakly admissibility for Abhyankar valuation of height $1$ as follows:

\begin{lem*}
 \label{monomialvalok}
 Let $v$ be an Abhyankar valuation of $\mr{ht}(v)=1$,
 and assume we are given an \'{e}tale morphism $S\rightarrow\mb{A}^e_{(s_1,\dots,s_e),k}$ such that $v(s_1),\dots,v(s_r)$
 is $\mb{Q}$-linearly independent with $r=\mr{rat.rk}(v)\leq e$, and $Z:=V(s_1,\dots,s_r)$ is irreducible.
 Let $f\in\mc{O}_X$ and $\pi\in\mc{O}_S$.
 Assume that $V(\pi)\subset V(s_1\cdots s_r)$ and $\NP^{(\infty)}(f/\pi)(v)=0$.
 Then $f/\pi$ is weakly admissible at $v$, and admits a weak mollifier so that $g\in \pi^{-1}\mc{O}_X$.
 If $v$ is divisorial {\normalfont(}{\it i.e.}\ $r=1${\normalfont)},
 then we can take $g$ so that it provides a \good \ally of $f/\pi$ on $S$ at $v$.
\end{lem*}
\begin{proof}
 The closure of the center of $v$ on $S$ is equal to $Z$ since $Z$ is assumed irreducible.
 Note that $k_v\supset\mc{O}_Z$.
 Invoking [EGA $0_{\mr{IV}}$, 19.5.4], the coordinates $s_i$ together with the fixed coordinates $x_i$
 yield the embeddings
 \begin{equation*}
  \mc{O}_X\hookrightarrow\mc{O}_{X_Z}\dd{s_1,\dots,s_r}
   \subset\mc{O}_{X_v}\dd{s_1,\dots,s_r},
 \end{equation*}
 where $X_v:=X\times_S\mr{Spec}(k_v)$.
 For $g=\sum a_{\ul{n}}\ul{s}^{\ul{n}}\in\mc{O}_{X_v}\dd{\ul{s}}$, we
 put $\widetilde{v}(g):=\min\{v(\ul{s}^{\ul{n}})\mid a_{\ul{n}}\neq0\}$.
 Note that minimum is uniquely attained by the assumption that $v(s_1),\dots,v(s_r)$ are linearly independent,
 and $\widetilde{v}$ is a valuation that extends $v$.
 If no confusion can arise, we denote $\widetilde{v}$ abusively by $v$.
 Note that for any $h\in\mc{O}_{X_Z}\dd{\ul{s}}$ and $\lambda>0$,
 there exists $\widetilde{h}\in\mc{O}_X$ such that
 $v(\widetilde{h}-h)>\lambda$ in $\mc{O}_{X_v}\dd{\ul{s}}$
 because $v(s_i)>0$ and $\mc{O}_{X_Z}\dd{\ul{s}}$ is the completion of
 $\mc{O}_X$ with respect to $(s_1,\dots,s_r)$-adic topology.
 Because $S$ is regular and $V(\pi)\subset V(s_1\dots s_r)$, using [EGA IV, (21.1.3.3) and 21.6.9],
 we may write $\pi=\ul{s}^{\ul{\alpha}}\pi^{\times}$ where $\pi^{\times}\in\mc{O}^{\times}_S$ (thus $v(\pi^{\times})=0$).
 The set $\Sigma$ of $\ul{n}\in\mb{N}^r$ such that $v(\ul{s}^{\ul{n}})< v(\pi)=v(\ul{s}^{\ul{\alpha}})$ is finite.
 Then we may write uniquely
 $(\pi^{\times})^{-1}f=\sum_{\ul{n}\in\Sigma}b_{\ul{n}-\ul{\alpha}}
 \ul{s}^{\ul{n}}+f_{\mr{res}}$ in $\mc{O}_{X_Z}\dd{\ul{s}}$ where
 $v(f_{\mr{res}})\geq v(\pi)$ and $b_{\ul{n}}\in\mc{O}_{X_Z}$.
 Note that $f_{\mr{res}}$ may not be in the image of $\mc{O}_X$.
 Then
 \begin{equation*}
  f/\pi=
   \bigl(\ul{s}^{-\ul{\alpha}}\sum_{\ul{n}\in\Sigma}
   b_{\ul{n}-\ul{\alpha}}\ul{s}^{\ul{n}}
   \bigr)
   +\ul{s}^{-\ul{\alpha}}\cdot f_{\mr{res}}
   =
   \bigl(
   \sum_{\ul{n}\in\Sigma-\ul{\alpha}}
   b_{\ul{n}}\ul{s}^{\ul{n}}
   \bigr)
   +\ul{s}^{-\ul{\alpha}}\cdot f_{\mr{res}}.
 \end{equation*}
 Put $\mc{O}_{X_Z}\cc{\ul{s}}:=\indlim_{|\ul{n}|\rightarrow\infty}\ul{s}^{-\ul{n}}\cdot\mc{O}_{X_Z}\dd{\ul{s}}$.
 Note that $\ul{s}^{-\ul{\alpha}}\cdot f_{\mr{res}}\in\bigl\{h\in\mc{O}_{X_Z}\cc{\ul{s}}\mid v(h)\geq0\bigr\}=:\mc{O}_{X_Z}\cc{\ul{s}}_{\geq0}$.
 Now, put
 \begin{equation*}
  \Sigma':=\bigl\{\ul{n}\in\Sigma-\ul{\alpha}\mid
   p\ul{n}\not\in \Sigma-\ul{\alpha}
   \bigr\} \subset\mb{Z}^{r},
 \end{equation*}
 and for $\ul{n}\in\Sigma'$, we put
 \begin{equation*}
  c_{\ul{n}}:=
   b_{\ul{n}}+
   b^p_{\ul{n}/p}+\dots+
   b^{p^{M}}_{\ul{n}/p^M}
   \quad\in\mc{O}_{X_Z},
 \end{equation*}
 where $M$ is the largest number such that $\ul{n}/p^M\in\mb{Z}^{r}$.
 Now, we have the following 2 cases:
 \begin{enumerate}
  \item\label{nonULAcase}
       We have $c_{\ul{n}}\in\bigl(\mc{O}_{X_v}\otimes_{k_v}k_v^{\mr{perf}}\bigr)^{p^{\infty}}\,(=\bigcap_i(\dots)^{p^i})$
       for any $\ul{n}\in\Sigma'$.

  \item\label{KLtypecase}
       There exists $\ul{n}_0\in\Sigma'$ such that
       $c_{\ul{n}_0}\not\in\bigl(\mc{O}_{X_v}\otimes_{k_v}k_v^{\mr{perf}}\bigr)^{p^{m+1}}$
       for some integer $m$.
 \end{enumerate}
 Put
 \begin{equation*}
  g:=\sum_{\ul{n}\in\Sigma'}
   \sum_{i=1}^{d_{\ul{n}}}
   \bigl(
   \sum_{k=i}^{d_{\ul{n}}}
   \widetilde{b}_{\ul{n}/p^k}\ul{s}^{\ul{n}/p^k}
   \bigr)^{p^{i-1}},
 \end{equation*}
 where $d_{\ul{n}}:=\max\bigl\{i\mid \ul{n}/p^i\in\mb{Z}^r\bigr\}$ (this is the same as $M$ in the definition of $c_{\ul{n}}$),
 and $\widetilde{b}_{\ul{n}}\in\mc{O}_X$ such that $v\bigl(\widetilde{b}_{\ul{n}}-b_{\ul{n}}\bigr)>v(\pi)$.
 Note that $\pi g\in\mc{O}_X$.
 Then we may compute
 \begin{equation*}
   \bigl(\sum_{\ul{n}\in\Sigma-\ul{\alpha}}
   b_{\ul{n}}\ul{s}^{\ul{n}}\bigr)
   +(g^p-g)
   \in
   \sum_{\ul{n}\in\Sigma'}c_{\ul{n}}\ul{s}^{\ul{n}}
   +\mc{O}_{X_Z}\cc{\ul{s}}_{\geq0}.
 \end{equation*}

 Consider the case \ref{nonULAcase}.
 The section of $X\rightarrow S$ yields a section of $X_v\rightarrow\mr{Spec}(k_v)$.
 Since $X_v$ is smooth over $k_v$, we have the inclusion $\mc{O}_{X_v}\subset k_v\dd{x_1,\dots,x_d}$.
 We have
 \begin{equation*}
  k_v^{\mr{perf}}\subset
  \bigl(\mc{O}_{X_v}\otimes_{k_v}k_v^{\mr{perf}}\bigr)
   ^{p^{\infty}}\subset
   \bigl(k_v^{\mr{perf}}\dd{x_1,\dots,x_d}\bigr)^{p^{\infty}}
   =k_v^{\mr{perf}}.
 \end{equation*}
 Thus, $c_{\ul{n}}\in k^{\mr{perf}}_v\cap\mc{O}_{X_Z}$ in $(k_v\dd{x_1,\dots,x_d})^{\mr{perf}}$ for any $\ul{n}\in\Sigma'$.
 On the other hand, we have
 \begin{equation*}
  k_v\subset
   k^{\mr{perf}}_v\cap\mc{O}_{X_v}\subset
   k^{\mr{perf}}_v\cap k_v\dd{x_1,\dots,x_d}=k_v
 \end{equation*}
 and thus $k^{\mr{perf}}_v\cap\mc{O}_{X_v}=k_v$.
 This implies that $c_{\ul{n}}\in k_v\cap\mc{O}_{X_Z}=\mc{O}_Z$ since $X\rightarrow S$ is faithfully flat.
 Now, we can take $\widetilde{c}_{\ul{n}}\in\mc{O}_S$ such that
 $v\bigl(\widetilde{c}_{\ul{n}}-c_{\ul{n}}\bigr)>v(\pi)$.
 Then we have
 $f/\pi+(g^p-g)\in\sum_{\ul{n}\in\Sigma'}\widetilde{c}_{\ul{n}}\ul{s}^{\ul{n}}+\mc{O}_{X_Z}\cc{\ul{s}}_{\geq0}\subset\rat{S}+\mc{O}_{X_Z}\cc{\ul{s}}_{\geq0}$.
 Thus, $\NP\bigl(f/\pi+(g^p-g)\bigr)(v)=0$, and $g$ is a weak mollifier at $v$.
 If $v$ is divisorial, we have $\mc{O}_{X_Z}\cc{s_1}_{\geq0}=\mc{O}_{X_Z}\dd{s_1}$,
 and thus $f/\pi+(g^p-g)$ is already a \good \ally since $\mc{O}_{X_Z}\dd{s_1}\cap\mc{O}_{X_\eta}=\mc{O}_X$ by Corollary \ref{intergoff}.

 Finally, consider \ref{KLtypecase}. Let us show that this occurs only
 when $\NP^{(\infty)}(v)<0$ which is impossible from the assumption
 of the lemma. Take $\widetilde{c}_{\ul{n}}\in\mc{O}_X$
 (not in $\mc{O}_S$ this time) such that
 $v\bigl(\widetilde{c}_{\ul{n}}-c_{\ul{n}}\bigr)>v(\pi)$, then
 $f/\pi+(g^p-g)\in
 \sum_{\ul{n}\in\Sigma'}\widetilde{c}_{\ul{n}}\ul{s}^{\ul{n}}
 +\mc{O}_{X_v}\cc{\ul{s}}_{\geq0}$.
 Write $\widetilde{c}_{\ul{n}_0}=\sum a_{\ul{k}}\ul{x}^{\ul{k}}$ in $\mc{O}_S\dd{\ul{x}}$.
 We have the following commutative diagram:
 \begin{equation*}
  \xymatrix{
   \mc{O}_X\ar@{^{(}->}[r]\ar@{^{(}->}[d]&\mc{O}_S\dd{\ul{x}}\ar@{^{(}->}[d]\\
  \mc{O}_{X_v}\dd{\ul{s}}\ar@{^{(}->}[r]&k_v\dd{\ul{s},\ul{x}}.
   }
 \end{equation*}
 Since $c_{\ul{n}}\in\mc{O}_{X_Z}\subset\mc{O}_{X_v}\dd{\ul{s}}$,
 we may write $c_{\ul{n}_0}=\sum a'_{\ul{k}}\ul{x}^{\ul{k}}$ with $a'_{\ul{k}}\in k_v$ in $k_v\dd{\ul{s},\ul{x}}$.
 By the choice of $\widetilde{c}_{\ul{n}}$, we have $v(a_{\ul{k}}-a'_{\ul{k}})>v(\pi)$ for any $\ul{k}$, and in particular, $v(a_{\ul{k}})$ is either $0$ or $>v(\pi)$.
 Because of Lemma \ref{intFrob} and the assumption that
 $c_{\ul{n}_0}\not\in\mc{O}^{p^{m+1}}_{X_v}\otimes k_v$, there exists
 $p^{m+1}\nmid\ul{k}_0$ such that $v(a_{\ul{k}_0})=0$.
 Let $\ul{l}_0\in\Lambda$ such that $p^K\ul{l}_0=\ul{k}_0$ with some
 $K\in\mb{N}$. Then, for any integer $a\geq K$, we have
 \begin{align*}
  v&\bigl((\widetilde{c}_{\ul{n}_0}\ul{s}^{\ul{n}_0})^{(a)}_{\ul{l}_0}\bigr)=\\
  &v\Bigl(
  \ul{s}^{\ul{n}_0}a_{\ul{l}_0}+
  (\ul{s}^{\ul{n}_0}a_{p\ul{l}_0})^{1/p}
  +\dots+
  (\ul{s}^{\ul{n}_0}a_{\ul{k}_0})^{1/p^K}
  +
  (\ul{s}^{\ul{n}_0}a_{p^{K+1}\ul{l}_0})^{1/p^{K+1}}
  +\dots+
  (\ul{s}^{\ul{n}_0}a_{p^{a}\ul{l}_0})^{1/p^{a}}
  \Bigr)\\
  &\hspace{23em}
  \in
  \bigl\{p^{-i}v(\ul{s}^{\ul{n}_0})\mid i=0,\dots,K\bigr\}=:D.
 \end{align*}
 For $\ul{n}\in\Sigma'\setminus\{\ul{n}_0\}$, by a similar computation,
 $v\bigl((\widetilde{c}_{\ul{n}}\ul{s}^{\ul{n}})^{(a)}_{\ul{l}_0}\bigr)$ is also a real number of the form $p^{-i}v(\ul{s}^{\ul{n}})$.
 By the choice of $\Sigma'$, this in particular implies that this number cannot be in $D$.
 Thus, we have
 \begin{equation*}
  \NP^{(\infty)}(f/\pi)(v)
   =
   \NP^{(\infty)}\bigl(f/\pi+(g^p-g)\bigr)(v)
   \leq
   \lim_{a\rightarrow\infty}
   v\left(\bigl({\textstyle \sum_{\ul{n}\in\Sigma'}}\,\widetilde{c}_{\ul{n}}\ul{s}^{\ul{n}}\bigr)^{(a)}_{\ul{l}_0}\right)
   \leq
   p^{-K}v(\ul{s}^{\ul{n}_0})<0,
 \end{equation*}
 and the claim follows.
\end{proof}

\subsection{}
We have the following purity type result.

\begin{lem*}
 \label{nagapurilem}
 Let $v$ be a minimal valuation on $S$ centered on $z\in U\subset S$.
 Assume we have a system of coordinates $\{s_1,\dots,s_d\}$ on $U$ around $z$ such that $\pi\in s_1^{a_1}\dots s_k^{a_k}\mc{O}^{\times}_S$.
 Let $D_i:=V(s_i)$, and denote by $v_{D_i}$ the associated divisorial valuation.
 Then the following conditions are equivalent:
 \begin{enumerate}
  \item\label{puritylem1}
       $\NP^{(\infty)}(f/\pi)(v)=0$ and there exists a mollifier which provides a \good \ally of $f/\pi$ on an open neighborhood of $z$ in $U$;

  \item\label{puritylem2}
       for any divisorial valuation $v_D$ centered on $U$, we have $\NP^{(\infty)}(f/\pi)(v_D)=0$;
  \item\label{puritylem3}
       $\NP^{(\infty)}(f/\pi)(v_{D_i})=0$ for any $i$.
 \end{enumerate}
\end{lem*}
\begin{proof}
 Assume the condition \ref{puritylem1} holds.
 Take a mollifier $g$ on $U$ providing a \good \ally around $z$.
 Then taking a suitable $\gamma\in\rat{S}$, $f/\pi+(g^p-g)+\gamma$ is a regular function around $v$ since it cannot be of the form
 $f'/\sigma+\alpha$ where $f'$ is separable function, $\alpha$ is a regular function, and $v(\sigma)>0$ by the assumption that $\NP^{(\infty)}(f/\pi)(v)=0$.
 Thus, there exists a closed subscheme $Z\subset U$ not containing $z$ such that $f/\pi+(g^p-g)+\gamma$ is in $\mc{O}_X\otimes\mc{O}_{U\setminus Z}$.
 This implies that $\NP^{(\infty)}(f/\pi)(w)=0$ for any valuation $w$ centered on $U\setminus Z$.
 This in particular implies that $\NP^{(\infty)}(f/\pi)(v_{D_i})=0$, and condition \ref{puritylem3} holds.
 
 Now, let us check (\ref{puritylem3} $\Rightarrow$ \ref{puritylem1}).
 We show by the induction on $k$.
 Since $v_{D_k}$ is a divisorial valuation, there exists a function $g_k\in \mc{O}_{X}\otimes\pi^{-n}\mc{O}_U$ for some $n\geq0$
 such that $\NP\bigl(f/\pi+(g^p_k-g_k)\bigr)(v_{D_k})=0$ by applying Lemma \ref{monomialvalok} to $U\setminus V(s_1\dots s_{k-1})$.
 Note that by construction, $f/\pi+(g^p_k-g_k)\in\mc{O}_X\otimes\pi^{-pn}\mc{O}_U$.
 However, since $\NP(f/\pi+(g^p_k-g_k))(v_{D_k})=0$, there exists $\gamma_k\in\rat{S}$ such that
 $f/\pi+(g^p_k-g_k)+\gamma_k\in\mc{O}_X\otimes (s_1^{a_1}\dots s_{k-1}^{a_{k-1}})^{-pn}\mc{O}_U$.
 This construction also shows that the mollifier can be taken in $\pi^{-n}\mc{O}_{X_U}$ for some $n$.
 Now, use the induction hypothesis to conclude.

 Finally, let us check that (\ref{puritylem1} $\Rightarrow$ \ref{puritylem2}) holds.
 Let $Z$ be as in the proof of (\ref{puritylem1} $\Rightarrow$ \ref{puritylem3}).
 We only need to check the claim for valuations centered on $Z$ above.
 If $w$ is a valuation centered on $Z\setminus(D_1\cup\dots\cup D_k)$, then $f/\pi$ already provides a \good \ally, and we have nothing to show.
 Let $w$ be a valuation centered on $Z\cap (D_1\cup\dots\cup D_k)$.
 Then we can apply (\ref{puritylem3} $\Rightarrow$ \ref{puritylem1}) for this $w$ to conclude.
\end{proof}

\begin{cor}
 \label{opennbdofmolpt}
 Let $v\in\left<S\right>$ be of height $1$.
 Assume that $\NP^{(\infty)}(f/\pi)(v)=0$ and there exists a finite extension $L$ of $k(S)$
 and some extension $v'$ of $v$ to $L$ such that $f/\pi$ is locally admissible at $v'$.
 Then there exists a Zariski open neighborhood $W$ of $v$ in $\left<S\right>$ such that $\NP^{(\infty)}(f/\pi)(w)=0$ for any $w\in W$ and $\mr{ht}(w)=1$.
\end{cor}
\begin{proof}
 We may take a local alteration $S'\rightarrow S$ around $v$ and a coordinate system $\{s_1,\dots,s_d\}$ on $S'$ such that $V(\pi)$
 is a union of $V(s_i)$'s and there exists a mollifier of $v$ on $S'$ which provides a \good \ally of $f/\pi$ at $v$ by Lemma \ref{goodmodelze}.
 By shrinking $S'$ around the center of an extension of $v$ to $S'$, we may assume that $\NP^{(\infty)}(f/\pi)(w)=0$ for any valuation on $k(S')$ centered on $S'$.
 Now, let $W'$ be the set of valuations in $\left<S\right>$ which extends to $S'$.
 Lemma \ref{nagapurilem} and (\ref{invarpulnp}) imply that $\NP^{(\infty)}(f/\pi)(w)=0$ for any $w\in W'$ and $\mr{ht}(w)=1$.
 Thus, it remains to show that $W'$ is a neighborhood of $v$.

 Take a compactification $\overline{S}'\rightarrow S$ of $S'$ over $S$.
 By the Gruson-Raynaud flattening, we may find modifications $T\rightarrow S$ and
 $\overline{T}'\rightarrow\overline{S}'$ and a finite flat morphism $\overline{T}'\rightarrow T$.
 We may assume that $\overline{T}'$ is integral.
 Let $T'$ be the pullback of $S'$ to $\overline{T}'$.
 The situation can be depicted as follows:
 \begin{equation*}
  \xymatrix@C=50pt{
   T'\ar@{^{(}->}[r]\ar[d]\ar@{}[rd]|\square&\overline{T}'\ar[r]^-{\mr{fin.\ flat}}\ar[d]^(.4){\mr{modif}}&T\ar[d]^{\mr{modif}}\\
  S'\ar@{^{(}->}[r]^-{\mr{open}}&\overline{S}'\ar[r]^-{\mr{prop}}&S.
   }
 \end{equation*}
 Since $g\colon T'\rightarrow T$ is flat, the image $g(T')\subset T$ is open.
 Let $w$ be an element of $\left<S\right>$ whose center lies in $g(T')$.
 Then we claim that $w$ extends on $S'$, and in particular belongs to $W'$.
 To show this, it suffices to show that $w$ extends on $T'$.
 Let $z$ be the center of $w$ in $g(T')$.
 Take a point $z'\in T'$ over $z\in g(T')$.
 Let $\mr{Spec}(R_v)\rightarrow T$ be the morphism from the valuation ring associated with $v$, and consider $P:=\mr{Spec}(R_v)\times_T T'$.
 This scheme $P$ is integral.
 Indeed, by shrinking $T$ around $z$, we may assume that $T$ is affine.
 We have inclusions $\mc{O}_T\subset R_v\subset k(T)$.
 Since $\mc{O}_T\rightarrow\mc{O}_T'$ is flat, this induces the inclusions $\mc{O}_{T'}\subset\mc{O}_P\subset k(T)\otimes_{\mc{O}_T}\mc{O}_{T'}$.
 Since $k(T)$ is a localization of $\mc{O}_T$, $k(T)\otimes_{\mc{O}_T}\mc{O}_{T'}$ is a localization of $\mc{O}_{T'}$, and since $T'$ is integral,
 we have an inclusion $k(T)\otimes_{\mc{O}_T}\mc{O}_{T'}\subset k(T')$.
 Thus, we have $\mc{O}_P\subset k(T')$ and the integrality follows.
 The point $z'$ induces a point $z''$ of $P$.
 We may take a valuation of $k(P)=k(T')$ centered on $z''$, which is the desired valuation centered on $T'$.
 Thus, if we put $W$ to be the Zariski open neighborhood defined by $g(T')$ in $\left<S\right>$, we have $W'\supset W$.
\end{proof}

\subsection{}
We use the following lemma due to Kedlaya.
Even though it is completely contained in \cite[5.4.4]{Ked4}, it may not be easy to transform the statement into the form we need,
we decided to include the proof for the convenience of the reader.

\begin{lem*}[{\cite[5.4.4]{Ked4}}]
 \label{kedexpseqcons}
 Let $S\rightarrow S^0\times\mb{A}^1_t$ be an \'{e}tale morphism, $\pi\in\mc{O}_{S^0}$ and let $v\in\left<S\right>$.
 Put $v^0:=v|_{S^0}$, and assume that $v(t)$ is not contained in the divisible closure of $\Gamma_{v^0}$.
 Assume we are given a Zariski open neighborhood $W\subset\left<S\right>$ of $v$.
 Then by replacing $S^0$ by a local alteration around $v^0$, there exist
 \begin{itemize}
  \item $x_1,\dots,x_r\in\mc{O}_{S^0}$ such that the components of
	$V(\pi)$ are cut out by some of these functions;
  \item rational numbers $a_i,b_i\in\mb{Q}$ such that
	$v^0(x_i)/b_i<v(t)<v^0(x_i)/a_i$;
 \end{itemize}
 such that $S\bigl(w(x_i)/b_i<w(t)<w(x_i)/a_i\bigr):=
 \bigl\{w\in\left<S\right>\mid w(x_i)/b_i<w(t)<w(x_i)/a_i\bigr\}
 \subset W$.
\end{lem*}
\begin{proof}
 Put $Z:=V(\pi)$ and $r=\mr{rat.rk}(v^0)$.
 Since $v(t)>0$, the center of $v$ is contained in $V(t)$.
 Since $V(t)\rightarrow S^0$ is \'{e}tale by assumption, by replacing $S^0$ by $V(t)$, we may assume that $S\rightarrow S^0$ has a section and fix one.
 By \cite[2.3.5]{Ked4}, by replacing $S^0$ by its alteration around $v^0$, we may assume that there exists a coordinate system
 $\{x_1,\dots,x_d\}$ such that $\{v^0(x_i)\}_{r=1,\dots,r}$ are linearly independent over $\mb{Q}$, and $\{x_i\}_{i=1,\dots,r}$ cuts out the components of $Z$.
 Let $h\in\mc{O}_{S^0}\setminus\{0\}$.
 Then we may write that $h=(x_1^{e_1}\dots x_r^{e_r})\cdot a$ where $e_i\in\mb{Q}$ and $v^0(a)=0$.
 Take a modification so that $a\in\mc{O}_{S^0}^\times$ followed by an exposing alteration $g\colon S'^0\rightarrow S^0$ of $(S^0,Z)$
 (cf.\ \cite[2.3.4]{Ked4} for the terminology).
 Let $x'_i$ be the coordinates of $S'^0$.
 Then since components of $Z':=g^{-1}(Z)$ are cut out by $\{x'_i\}_{i=1,\dots,r}$, by shrinking $S'^0$ around the center of $v^0$ if necessarily,
 we may write $g^*(x_1^{e_1}\dots x_r^{e_r})=({x'_1}^{e'_1}\dots {x'_r}^{e'_r})\cdot b'$ where $b'\in\mc{O}_{S'^0}^\times$.
 Thus, over $S'^0$, we can write $g^*h=({x'_1}^{e'_1}\dots {x'_r}^{e'_r})\cdot u'$ with $u'\in\mc{O}_{S'^0}^{\times}$.
 Note that since ${x'_1}^{e'_1}\dots {x'_r}^{e'_r}\in\mc{O}_{S^0}$, $e'_i$ are non-negative integers.
 
 We can take $f^{(1)},\dots,f^{(2n)}\in\mc{O}_S$ such that
 $W\supset\bigl\{w\mid f^{(k)}/f^{(n+k)}(w)\geq0\quad(k=0,\dots,n)\bigr\}\ni v$.
 Let $f$ be one of $f^{(k)}$'s.
 Using the fixed section of $S\rightarrow S^0$, we may write $f=\sum f_i t^i$ in $\mc{O}_{S^0}\dd{t}$ with $f_i\in\mc{O}_{S^0}$.
 Since $s:=v(t)$ is not in the divisible closure, there exists $h$ such that $v^0(f_i)+is>v^0(f_h)+hs$ for any $i\neq h$.
 Note that this $h$ is chosen so that $v(f_it^i)>v(f_ht^h)$ for any $i\neq h$.
 Take an index $N$ such that $(N-h)s>v^0(f_h)$.
 Note that, with this choice of $N$, we have $v(f_it^i)>v(f_ht^h)$ for any $i\geq N$.
 By the argument of the previous paragraph, by taking further alteration, we may assume that
 $f_i=u_i\cdot(x_1^{e_{i1}}\dots x_r^{e_{ir}})$ where $e_{ij}\in\mb{Z}$, $u_i\in\mc{O}_{S^0}^{\times}$ for $i=0,\dots,N-1$.
 
 Now, for $f=f^{(k)}$, we put superscript ${}^{(k)}$ to $e_{ij}$ {\it etc}.
 To ease the notation, $e^{(k)}_{h^{(k)}i}$ is denoted by $e^{(k)}_{hi}$.
 We have $v\bigl((x_1^{e^{(k)}_{i1}}\dots x_r^{e^{(k)}_{ir}})\cdot t^{i-h^{(k)}}\bigr)>0$ for $i=0,\dots,N^{(k)}-1$ and $i\neq h^{(k)}$,
 $v\bigl((x_1^{-e_{h1}^{(k)}}\dots x_r^{-e_{hr}^{(k)}})\cdot t^{N^{(k)}-h^{(k)}}\bigr)>0$, and
 $v\bigl((x_1^{e^{(k)}_{h1}-e^{(n+k)}_{h1}}\dots x_r^{e^{(k)}_{hr}-e^{(n+k)}_{hr}})t^{h^{(k)}-h^{(n+k)}}\bigr)\geq0$.
 Note that if the last $\geq0$ is $=0$, then we have
 $h^{(k)}=h^{(n+k)}$ and $e^{(k)}_{hi}=e^{(n+k)}_{hi}$ for any $i$ by
 the choice of $x_i$ and $t$.
 We take $c_i$ such that $v(x_i)=c_is$. Then these conditions can be written as
 \begin{itemize}
  \item $c_1e^{(k)}_{i1}+\dots+ c_re^{(k)}_{ir}+(i-h^{(k)})>0$ for any $k$, $i=0,\dots,N^{(k)}-1$, and $i\neq h^{(k)}$;
  \item $-(c_1e^{(k)}_{h1}+\dots+ c_re^{(k)}_{hr})+(N^{(k)}-h^{(k)})>0$ for any $k$;
  \item $(c_1e^{(k)}_{h1}+\dots+ c_re^{(k)}_{hr})-(c_1e^{(n+k)}_{h1}+\dots+ c_re^{(n+k)}_{hr})+h^{(k)}-h^{(n+k)}\geq0$ for any $1\leq k\leq n$,
	and $=0$ if and only if $h^{(k)}=h^{(n+k)}$ and $e^{(k)}_{hi}=e^{(n+k)}_{hi}$ for any
	$i$.
 \end{itemize}
 Thus, we may take $a_i<c_i<b_i$ such that for any $a_i\leq d_i\leq
 b_i$, the inequalities holds even if we replace $c_i$ by
 $d_i$. This is what we need.
\end{proof}

\subsection{}
Let $t,s_2,\dots,s_e$ be coordinates on $S$ around $z\in S$.
Let $U:=(u,u_2\dots,u_e)\in\mb{Q}_{>0}^{e}$.
Then the {\em toric valuation} $v_U$ associated to $U$ is defined to be
\begin{equation*}
 v_U\bigl(\sum a_{\ul{k}}\ul{s}^{\ul{k}}\bigr)=
  \min\Bigl\{\sum k_iu_i\mid
  \mbox{$\ul{k}=(k_1,\dots,k_e)$ such that $a_{\ul{k}}\neq0$}
  \Bigr\}
\end{equation*}
for $\sum a_{\ul{k}}\ul{s}^{\ul{k}}\in k(z)\dd{t,s_2\dots,s_e}\supset\mc{O}_S$.
This valuation is divisorial.
Let $U'=(u',u_2\dots,u_e)$ such that $u'\geq u$.
Then for $g\in\mc{O}_S$, we have $v_{U'}(g)\geq v_U(g)$.
Assume we have an \'{e}tale morphism $S\rightarrow \mb{A}^1_t\times S^0$.
Assume $\pi\in\mc{O}_{S^0}$ (not just in $\mc{O}_S$!) and $f\in\mc{O}_X$.
Then we have
\begin{equation}
 \label{ineqnpinftori}
  \NP^{(\infty)}(f/\pi)(v_{U'})
  \geq
  \NP^{(\infty)}(f/\pi)(v_{U}).
\end{equation}

\subsection{}
Let us show that the weakly admissibility holds for any minimal valuation after a local alteration.
We use the induction on the transcendence defect.

\begin{lem*}
 \label{minvalok}
 If $v$ is a minimal valuation, there then exists a finite extension $L$ of $k(S)$
 and some extension $v'$ of $v$ to $L$ such that $f/\pi$ is locally admissible at $v'$.
\end{lem*}
\begin{proof}
 We may assume $k$ to be algebraically closed.
 When $\NP^{(\infty)}(f/\pi)(v)<0$, then the lemma has already been verified in Lemma \ref{npnegok}, so we may assume $\NP^{(\infty)}(f/\pi)(v)=0$.
 We use the induction on the transcendence defect ({\it i.e.}\ $\dim(S)-\mr{rat.rk}(v)-\mr{trdeg}_k(k_v)$) of $v$.
 Abhyankar's inequality implies that transcendence defect is non-negative (cf.\ \cite[2.1.4]{Ked4}).
 When the transcendence defect is $0$, the valuation is monomial in which case we have already treated in Lemma \ref{monomialvalok}.
 
 Assume that the claim holds for valuations whose transcendence defect is $<n$,
 and let $v\in\left<S\right>$ be a valuation of transcendence defect $n$.
 By replacing $S$ by its local alteration around $v$, we may find an \'{e}tale morphism $S\rightarrow S^0\times\mb{A}^1_t$ such that
 the image of $v$ denoted by $v^0$ in $\left<S^0\right>$ is of transcendence defect $n-1$, $\pi\in\mc{O}_{S^0}$, and $v(t)>0$ (cf.\ \cite[5.1.2]{Ked4}).
 Let $z$ (resp.\ $z^0$) be the center of $v$ (resp.\ $v^0$) in $S$ (resp. $S^0$).
 We have $k(z^0)=k(z)$ since $v$ is minimal and $k$ is assumed algebraically closed.
 We denote by $\rho\colon S\rightarrow S^0$ the projection.
 Let $R_{v^0}$ be the valuation ring corresponding to $v^0$.
 The point $v^0$ defines the morphism $\mr{Spec}(R_{v^0})\rightarrow S^0$ and let
 \begin{equation*}
  S_{v^0}:=\mr{Spec}(R_{v^0})\times_{S^0}S
   \rightarrow \mr{Spec}(R_{v^0}).
 \end{equation*}
 Let $a\in R_{v^0}$ be an element such that $v^0(a)>0$.
 We take the $a$-adic completion $S_{v^0}^\wedge$ of $S_{v^0}$ over the $a$-adic formal scheme $\mr{Spf}(R_{v^0}^\wedge)$.
 Put $\ell':=\mr{Frac}(R_{v^0}^\wedge)\,(=k(S^0)_{v^0}^\wedge)$,
 and let $\mb{S}:=(S^\wedge_{v^0})^{\mr{rig}}$ be the Raynaud generic fiber of $S_{v^0}^\wedge$.
 For a rigid space $X$, we denote by $\left<X\right>$ the underlying topological space in the sense of \cite[II, \S3.1]{FK}.
 This is the same as the underlying topological space considering $X$ as an adic space (cf.\ \cite[II, A.4.7]{FK}),
 and (the underlying topological space of) the Berkovich space associated to $X$ is simply the separation of $\left<X\right>$ (cf.\ \cite[II, C.4.34]{FK}).
 We have the following diagram of topological spaces:
 \begin{equation*}
  \xymatrix{
   \left<\mb{S}\right>\ar[d]\ar[r]^-{\approx}&
   \left<S_{v^0}\right>_{U_{\eta}}
   \setminus U_{\eta}
   \ar@{^{(}->}[r]\ar[d]&
   \left<S\right>_{U}
   \ar[d]&
   \left<S\right>\ar[d]\ar[l]\\
  \left<\mr{Sp}(\ell')\right>\ar[r]^-{\approx}&
   \left<\mr{Spec}(R_{v^0})\right>_{\eta}\setminus\{\eta\}
   \ar@{^{(}->}[r]&
   \left<S^0\right>_{U^0}
   &
   \left<S^0\right>
   \ar[l]
   }
 \end{equation*}
 where $U^0:=S^0\setminus\{z^0\}$, $U:=\rho^{-1}(U^0)$, $\eta$ is the
 generic point of $U^0$.
 Moreover, $\approx$ means homeomorphic,
 which follows by \cite[II, E.2.10]{FK}.
 Consider the composition
 \begin{equation*}
  \left<\mb{S}\right>\xrightarrow{\mr{sp}}
   S_{v^0}^\wedge\otimes_{R^\wedge_{v^0}}k(v_0)\cong
   \rho^{-1}(z_0).
 \end{equation*}
 Let $\mb{D}^{\circ}_{\ell'}$ be the pullback $\mr{sp}^{-1}(z)$ of
 $z\in\rho^{-1}(z_0)$, in other words, the tubular neighborhood of
 $z$. We have the following diagram:
 \begin{equation*}
  \xymatrix{
   \mr{Spec}(k(z))\ar@{^{(}->}[r]\ar[d]_{\sim}&
   S^\wedge_{v^0}\ar[d]^{\rho}\\
  \mr{Spec}(k(z^0))\ar@{^{(}->}[r]&
   \mr{Spf}(R^\wedge_{v^0}).
   }
 \end{equation*}
 The morphism $\rho$ is smooth. Thus,
 $\mb{D}^{\circ}_{\ell'}$ is an ``open disk'' (cf.\ \cite[Thm
 4.6]{BPR}\footnote{
 In fact, they assume that the base field, in our situation $\mr{Frac}(R^\wedge_{v^0})$, to be algebraically closed,
 but the result can be proven similarly.
 P. Berthelot calls this result the ``weak fibration theorem'' and a proof can be found in 1.3.2 of his famous preprint
 ``{\em Cohomologie rigide et cohomologie rigide \`{a} supports propres. Premi\'{w}re partie (version provisoire 1991)},
 Pr\'{e}publication IRMR {\bf 96-03} (1996)''.
 }).
 By construction, $v\in\left<S\right>$ defines a point
 $\alpha'\in\mb{D}^{\circ}_{\ell'}$. Arguing as \cite[5.2.2]{Ked4},
 $\alpha'$ is a terminal point, namely it is either of type 1 or 4.

 Now, let $\ell$ be the completion of an algebraic closure of $\ell'$.
 We take a lifting $\alpha$ of $\alpha'$ on $\mb{D}_{\ell}$.
 There exists an element $a\in\ell$ such that $\mb{D}^{\circ}_\ell\supset\mr{Sp}(\ell\!\left<t/a\right>)\ni\alpha$ (thus $v^0(a)>0$).
 We have
 \begin{equation*}
  \mc{O}_S\subset\widehat{\mc{O}}_{S,z}\cong
   \widehat{\mc{O}}_{S^0,z^0}\dd{t},\qquad
   \widehat{\mc{O}}_{S^0,z^0}\subset k(S^0)^\wedge_{v^0}=\ell'
   \subset\ell,
 \end{equation*}
 where the isomorphism follows by [EGA IV, 17.6.3] since the morphism
 $S\rightarrow S^0\times\mb{A}^1_t$ is \'{e}tale and $k(z^0)=k(z)$.
 On the other hand, we have
 $\mc{O}_S\subset\Gamma(\mb{D}^{\circ}_{\ell'},\mc{O})
 \subset\ell\!\left<t/a\right>$. Thus, we have
 \begin{equation*}
  \mc{O}_S\subset
   \widehat{\mc{O}}_{S^0,z^0}\dd{t}\cap
   \ell\!\left<t/a\right>
   \subset\bigl(\widehat{\mc{O}}_{S^0,z^0}[a]\bigr)
   \!\left<t/a\right>,
 \end{equation*}
 where the intersection is taken in $\ell\dd{t}$ (cf.\
 \ref{fixnotatbersp} for the notation of the last ring).
 Note that $R:=\widehat{\mc{O}}_{S^0,z^0}[a]$ is noetherian.
 Since $f$ is a function on $X$, we know that $f\in\mc{O}_S\dd{\ul{x}}$ is algebraic over $\mr{Frac}(\mc{O}_S[\ul{x}])$.
 Summing up, $f\in R\!\left<t/a\right>\dd{\ul{x}}\subset \ell\!\left<t/a\right>\dd{\ul{x}}$ is algebraic over $\mr{Frac}(\ell\!\left<t/a\right>[\ul{x}])$.
 Thus, we are in the situation of applying Proposition \ref{mainpropstabil}.

 The proposition implies that we can take $0<s_1<r_\alpha$ such that $0=\NP^{(\infty)}(f/\pi)(\alpha)=\NP^{(\infty)}(f/\pi)(v_{\alpha,s})$ for any $s\in[s_1,r_\alpha]$.
 There exists $b\in\mf{o}_\ell$ such that $v_{b,s_2}\geq v_{\alpha}$ (cf.\ \cite[2.2.6]{Ked4} for the notation) for some $s_1<s_1$.
 We may assume that $s_2$ is not in the divisible closure of $v^0(k(S^0))$.
 By approximating $b$, we may assume that $b$ is in the integral closure of $R_{v^0}$.
 By replacing $S^0$ by its alteration containing $b$, we may assume that $b\in\mc{O}_{S^0}$.
 Finally, replace $t$ by $t-b$, and we may assume that
 \begin{equation*}
  v_{\alpha,s_1}>v_{0,s_2}>v=v_{\alpha}.
 \end{equation*}
 Let $v_1$ be the valuation on $k(S)$ corresponding to $v_{0,s_2}$.
 Then by construction, $v_1$ is a minimal valuation whose transcendence defect is equal to $n-1$.

 Now, since $v_1$ admits an alteration which makes $f/\pi$ weakly admissible at some extension of $v_1$ by induction hypothesis and $\NP^{(\infty)}(f/\pi)(v_1)=0$,
 there exists a Zariski open neighborhood $W_1\subset\left<S\right>$ of $v_1$ such that $\NP^{(\infty)}(f/\pi)(w)=0$ for any $w\in W_1$ by Corollary \ref{opennbdofmolpt}.
 By Lemma \ref{kedexpseqcons} for $v_1$, by replacing $S^0$ by its local alteration around $v^0$, we may find a coordinate $\{z_1,\dots,z_r\}$
 on $S^0$ and rational numbers $a_i$, $b_i$ such that $v^0(z_i)/b_i<v_1(t)=s_2<v^0(z_i)/a_i$, $\NP^{(\infty)}(f/\pi)(w)=0$ for any
 $w\in S\bigl(w(z_i)/b_i<w(t)<w(z_i)/a_i\bigr)$, and the components of $V(\pi)\subset S^0$ passing through $z^0$ is cut out by $z_1,\dots,z_r$.
 This implies that for $U:=(u_0,u_1,\dots,u_r)\in\mb{Q}^{r+1}_{>0}$ such that $u_i/b_i<u_0<u_i/a_i$ for $i=1,\dots,r$, $\NP^{(\infty)}(f/\pi)(v_U)=0$.
 By (\ref{ineqnpinftori}), this implies that $\NP^{(\infty)}(f/\pi)(v_U)=0$
 for any $U\in D:=\bigl\{(u_0,\dots,u_r)\in\mb{Q}^{r+1}_{>0}\mid u_i/b_i<u_0\bigr\}$.
 Since $s_2\leq v(t)$, $D$ includes a neighborhood of $(v(t),v(z_1),\dots,v(z_r))=(v(t),v^0(z_1),\dots,v^0(z_r))$.

 Now, by a suitable toroidal blowup in $t,z_1,\dots,z_r$,
 we may take a local alteration $f_1\colon S_1\rightarrow S$ around $v$ such that $(S_1,Z_1:=V(\pi))$ is a smooth pair,
 and the divisorial valuation corresponding to each component of $Z_1$ is of the form $v_U$ for some $U$ such that $u_i/b_i<u_0$.
 Thus, by purity (Lemma \ref{nagapurilem}), $f/\pi$ possesses a mollifier around $v$ on $S_1$, as required.
\end{proof}

\begin{lem}
 \label{clokthevok}
 Let $h\geq1$ be an integer.
 Assume that the existence of a \good \ally after a local alteration is known for any field $k$ and any height $h$ valuation $v$ on $S$ such that
 $\mr{trdeg}_k(k_v)=0$.
 Then we have the existence of a \good \ally after a local alteration at any valuation of height $h$.
\end{lem}
\begin{proof}
 Let $v$ be a height $h$ valuation such that $\mr{trdeg}(k_v/k)=n$.
 By taking an alteration around $v$, we may assume that the center of $v$ on $S$ is of dimension $n$ and $S$ is smooth.
 Then we may find an open dense subscheme $U\subset S$ and a smooth morphism $U\rightarrow\mb{A}^n$
 such that the center of $v$ is sent to the generic point $\eta$ of $\mb{A}^n$.
 Since $v$ is centered over $\eta$, the restriction $v|_{k(\eta)}$ is trivial, and defines a valuation $v'$ in $\left<S_\eta\right>_{/k(\eta)}$.
 By construction, $v'$ is a valuation of height $h$ such that $\mr{trdeg}_{k(\eta)}(k_{v'})=0$.
 Thus, by assumption, there exists a finite extension $L$ of $k(S_\eta)$ over which there exists a mollifier providing a \good \ally at an extension $w'$ of $v'$.
 The valuation $w'$ defines a valuation $w\in\left<S\right>_{/k}$, and by definition, a mollifier of $w'$ can be taken as a mollifier of $w$.
\end{proof}

\begin{lem}
 \label{anyvalok}
 For any valuation, there exists a finite extension $L$ of $k(S)$ and some extension $v'$ of $v$ to $L$ such that $f/\pi$ is locally admissible at $v'$.
\end{lem}
\begin{proof}
 We show the claim by the induction on the height of the valuation.
 When the height is $1$, we have already treated in Lemmas
 \ref{minvalok}, \ref{goodmodelze}, and \ref{clokthevok}.
 We assume that the claim holds for $\mr{ht}(v)<h$. Let $\mr{ht}(v)=h$
 from now on. By Lemma \ref{clokthevok}, we may assume that
 $\mr{trdeg}_k(k_v)=0$. Moreover, we may assume $k$ to be algebraically
 closed. We may write $v=v'\circ\overline{v}$ where $\mr{ht}(v')=h-1$,
 $\mr{ht}(\overline{v})=1$. Using the induction hypothesis at $v'$,
 by replacing $S$ by its alteration and $v$, $v'$ by some of its
 extensions, we may assume that there exists $g\in\mc{O}_{X_\eta}$ such that $f/\pi+(g^p-g)$ is equal a function of either of the following forms:
 1.\ it belongs to $k(S)+\mc{O}_{X,v'}$;
 2.\ it belong to $k(S)+f'/\tau$ where $\sigma_{v'}(f')$ is separable and $v'(\tau)>0$.

 Let us treat Case 1.
 Shrinking $S$ around $v$ if necessary,
 we can write that $f/\pi+(g^p-g)\in f'/\tau+k(S)$ with some $f'\in\mc{O}_X$, $\tau\in\mc{O}_S$ such that $v'(\tau)=0$.
 If $v(\tau)=0$, then $f'/\tau$ is in $\mc{O}_{X}\otimes\mc{O}_{S,v}$, and this becomes a \good \ally if we modify $S$ suitably.
 Thus, we assume $v(\tau)>0$.
 Now, let $B:=|\mr{div}(\tau)|$. By assumption $v'\not\in B$ and $v\in B$.
 There exists a divisor $D\ni v'$ (thus $D\ni v$).
 Take an alteration theorem of de Jong for $(S,D\cup B)$, and we may assume that $D\cup B$ is a simple normal crossing divisor.
 By shrinking $S$ around $v$ and enlarging $D$,
 we may assume that the intersection $E$ of all the components of $D$ is irreducible whose generic point is the center of $v'$.
 We may assume that there exists local parameters $\{t_1,\dots,t_m\}$ such that $D_{\mr{red}}=V(t_1\cdots t_m)$ and $E=V(t_1,\dots,t_m)$.
 Complete the local parameters by $\{s_1,\dots,s_n,t_1,\dots,t_m\}$ so that this forms a system of local parameters around $v$,
 and defines an \'{e}tale morphism $S\rightarrow\mb{A}^n\times\mb{A}^m$ such that the fiber over $\mb{A}^n\times\{0\}$ is $E$.
 Note that the induced morphism $E\rightarrow\mb{A}^n$ is \'{e}tale as well.
 Replacing $S$ by $S\times_{\mb{A}^n}E$ and shrinking around some extension of $v$,
 we may assume that there exists an \'{e}tale morphism $S\rightarrow E\times\mb{A}^m$ such that the fiber over $E\times\{0\}$ is $E\subset S$.
 By construction, $k(E)=k_{v'}$ and $\overline{v}$ is a minimal valuation on $E$.

 Since $v'(\tau)=0$, we have $\tau|_E\neq 0$ and the function $f'/\tau|_E$ is defined on $E$.
 For a function $h\in\mc{O}_S$ such that $h/\tau|_E=0$, we have $v'(h/\tau)>0$, and thus $v(h/\tau)>0$.
 This implies that if $f'/\tau|_E=0$, then we already have $\NP(f'/\tau)(v)=0$.
 Thus, we may assume that $f'/\tau|_E\neq0$.
 Since $\overline{v}$ is minimal, we can take an alteration $E'\rightarrow E$ over which
 $\NP\bigl(f'/\tau|_{E'}+(g^p_{E'}-g_{E'})\bigr)(\overline{v})=\NP^{(\infty)}(f'/\tau|_E)(\overline{v})=:V$ for some $g_{E'}$ such that $E'$ is smooth.
 Note that if $V<0$, then $\sigma_{\overline{v}}\bigl(f'/\tau|_{E'}+(g^p_{E'}-g_{E'})\bigr)$ is separable by Lemma \ref{propnpval}.
 Let $S':=S\times_{(E\times\mb{A}^n)}(E'\times\mb{A}^n)$, $X':=X\times_SS'$.
 Take $\widetilde{g}_{E'}\in\mc{O}_{X',v'}$ whose restriction over $E'$ is $g_{E'}$.
 By construction, we have
 $\NP\bigl(f'/\tau+(\widetilde{g}_{E'}^p-\widetilde{g}_{E'})\bigr)(v)=\NP\bigl(f'/\tau|_{E'}+(g_{E'}^p-g_{E'})\bigr)(\overline{v})$
 and
 $\sigma_v\bigl(f'/\tau+(\widetilde{g}_{E'}^p-\widetilde{g}_{E'})\bigr)=\sigma_{\overline{v}}\bigl(f'/\tau|_{E'}+ (g_{E'}^p-g_{E'})\bigr)$.
 Summing up, if $V=0$, then $f'/\tau+(\widetilde{g}_{E'}^p-\widetilde{g}_{E'})$ is regular,
 and in the case $V<0$, $\sigma_{v}\bigl(f'/\tau+(\widetilde{g}_{E'}^p-\widetilde{g}_{E'})\bigr)$ is separable, and the claim follows.

 Let us handle Case 2.
 We may assume that $\sigma_v(f'/\tau)$ is not separable.
 Recall that $\mr{Fil}^\lambda_{v'}(\tau^{-1}\mc{O}_S)$ is the filtration on $\tau^{-1}\mc{O}_S\subset k(S)$ associated with $v'$ (cf.\ \ref{prepringa}).
 Then since $\mc{O}_S$ is noetherian, the submodule $\mr{Fil}^{\lambda}_{v'}(\tau^{-1}\mc{O}_S)$ is finitely generated for any $\lambda$.
 Thus, $\mr{gr}^{\lambda}_{v'}(\tau^{-1}\mc{O}_S)$ is a finitely generated $\mr{gr}_{v'}(\mc{O}_S)$-module as well.
 Now, for $\alpha\in k(S)$, we have an isomorphism $\kappa_\alpha\colon\mr{gr}^{v'(\alpha)}_{v'}(k(S))\xrightarrow{\sim}\mr{gr}^0_{v'}(k(S))$
 of $\mr{gr}^0_{v'}(k(S))$-modules by multiplying by $\alpha^{-1}$.
 We put $\overline{v}_\alpha:=\overline{v}\circ\kappa_\alpha$, and $\overline{\NP}_\alpha(f)(v)$ for $f$ such that $\NP(f)(v')=v'(\alpha)$
 to be the image in $\Gamma_{\overline{v}}$ of $\NP(\alpha^{-1}\cdot f)(v)$.

 Fix any $\alpha\in k(S)$ such that $v'(\alpha)=\NP(f'/\tau)(v')=:\lambda<0$.
 Put
 \begin{equation*}
  J:=
  \overline{v}_{\alpha}\bigl(\mr{gr}_{v'}^{\lambda}(\tau^{-1}\mc{O}_S)\bigr)
   =
   \overline{v}\bigl(\kappa_\alpha(\mr{gr}_{v'}^{\lambda}(\tau^{-1}\mc{O}_S))\bigr)
   \subset\Gamma_{\overline{v}}.
 \end{equation*}
 Since $\kappa_\alpha(\mr{gr}_{v'}^{\lambda}(\tau^{-1}\mc{O}_S))$ is a finitely generated $\mr{gr}^0_{v'}(\mc{O}_S)$-module by the observation above,
 by invoking Corollary \ref{discvalsemgr}, the subset $J\subset\mb{R}_\infty$ is discrete.
 On the other hand, write $f'/\tau=\sum a_{\ul{k}}\ul{x}^{\ul{k}}$ with $a_{\ul{k}}\in\mc{O}_S$.
 Put $\Delta:=\bigl\{\ul{k}\in\mb{N}^d\setminus\{\ul{0}\}\mid v'(a_{\ul{k}})=\lambda\,(=\NP(f'/\tau)(v'))\bigr\}$.
 Since $\sigma_{v'}(f'/\tau)$ is assumed to be separable, there exists $\ul{k}_0\in\Delta$ such that $p\nmid\ul{k}_0$.
 Put $B:=\overline{v}_{\alpha}(a_{\ul{k}_0})$.
 By the discreteness of $J$, the intersection $J\cap\left]-\infty,B\right]$ is a finite set.

 By arguing similar to Lemma \ref{redulemav}, since we are assuming that $\sigma_{v}(f'/\tau)$ is not separable,
 we can take $g\in\mc{O}_X\otimes(\tau^{-1}\mc{O}_S)^{1/q}$ ($q>1$) such that $\NP(g)(v')=\lambda/q$ and $-\sigma_v(f'/\tau)=\sigma_v(g^q)$.
 Since $\sigma_{v'}(f'/\tau)$ is separable, we have $\sigma_{v'}(f'/\tau+(g^q-g))\neq\sigma_{v'}(g^q)$,
 which implies that $\NP(f'/\tau)(v')=\NP(f'/\tau+(g^q-g))(v')$.
 Moreover, writing $f'/\tau+(g^q-g)=\sum a'_{\ul{k}}\ul{x}^{\ul{k}}$, we have $a'_{\ul{k}_0}=a_{\ul{k}_0}$ in $\mr{gr}^{\lambda}_{v'}(k(S))$
 because $\NP(g)(v')>q\NP(g)(v')=\NP(f'/\tau)(v')$ and $g$ does not affect the value in $\mr{gr}^{\lambda}_{v'}(k(S))$.
 Thus, we have
 \begin{equation*}
  \overline{\NP}_{\alpha}(f'/\tau)(v)
   <
  \overline{\NP}_\alpha(f'/\tau+(g^q-g))(v)
  =
  \overline{\NP}_\alpha(f'/\tau+g^q)(v)
   \in J\cap\left]-\infty,B\right].
 \end{equation*}
 If $\sigma_v(f'/\tau+(g^q-g))$ is still not separable, we repeat this process.
 Since $J\cap\left]-\infty,B\right]$ is a finite set, this procedure terminates in finitely many steps,
 and we may argue as Lemma \ref{npnegok} to conclude the proof.
\end{proof}

\begin{lem}
 \label{finallemgalclo}
 Let $v\in\left<S\right>$. Then there exists an alteration
 $S'\rightarrow S$ such that for {\em any} extension $v'$ of $v$ to $S'$
 centered on $z'$, there exists $g_{v'}\in\mc{O}_{X}\otimes k(S')$
 such that $f/\pi+(g_{v'}^p-g_{v'})\in h/\sigma+\mc{O}_{X,z'}$, where
 $\sigma\in\mc{O}_S$ and $h\in\mc{O}_X$ is either separable function or
 $0$ {\normalfont(}not just a constant!{\normalfont)}.
\end{lem}
\begin{proof}
 The proof is similar to \cite[2.4.2]{Ked4}.
 Let $s\colon S\rightarrow X$ be the fixed section.
 For a morphism $T\rightarrow S$, we also denote by $s$ the base change $T\rightarrow X_{T}$.
 Let $T'\rightarrow S$ be an alteration such that the pullback of $f/\pi$ admits a \good \ally at {\em some} extension $v'$ of $v$,
 whose existence is assured by Lemma \ref{anyvalok}.
 This means that there exists $g'\in\mc{O}_X\otimes k(T')$ such that $f/\pi+(g'^p-g')$ is a \good \ally of $f/\pi$ at $v'$.
 Then $f_{T'}:=\bigl(f/\pi+(g'^p-g')\bigr)|_{s(T')}$, the ``constant term'', is a rational function in $\rat{T'}$.
 Consider the alteration $T:=T'[t]/(t^p-t-f_{T'})$.
 Then by putting $g=g'-t$ and for any extension $w$ of $v'$, $f/\pi+(g^p-g)$ is good at $w$ and $\bigl(f/\pi+(g^p-g)\bigr)|_{s(T)}=0$.

 It suffices to check the lemma for $(f/\pi)^{p^n}$ for some $n$, since $f/\pi$ and $(f/\pi)^{p^n}$ shares the same alliance.
 There exists an integer $n$ such that $g^{p^n}\in \mc{O}_X\otimes k(S)^{\mr{sep}}$,
 where $k(S)^{\mr{sep}}$ is the separable closure of $k(S)$ in $k(T)$.
 Thus, by replacing $f/\pi$ by $(f/\pi)^{p^n}$, we may assume that the extension $k(T)/k(S)$ is separable.

 Take a Galois closure of $k(T)$, denoted by $L$.
 Let $u$ be any extension of $w$ to $L$.
 Then $f/\pi$ is locally admissible at $u$.
 Let $T_u$ be a model of $L$ over which $f/\pi+(g^p-g)$ is good with some $g$.
 Let $u'$ be any valuation on $L$ over $v$.
 Since the Galois group acts transitively on the set of extensions of valuation by [Bourbaki, Comm.\ Alg., Ch.\ VI \S8.6 Cor 1],
 there exists $\sigma\in\mr{Gal}(L/k(S))$ such that $\sigma(u)=u'$.
 Let $T_{u'}$ be the normalization of $T_u$ with respect to the morphism $\mr{Spec}(L)\xrightarrow{\sigma}\mr{Spec}(L)\rightarrow T_u$.
 Let $\widetilde{\sigma}\colon T_{u'}\rightarrow T_u$ be the induced morphism.
 Then $\widetilde{\sigma}^*(f/\pi+(g^p-g))=f/\pi+(\widetilde{\sigma}^*(g)^p-\widetilde{\sigma}^*(g))$ is good on $T_{u'}$.
 Thus, the local admissibility holds for the valuation $u'$.
 We put $\overline{T}_u$ to be a compactification of $T_u$ over $S$.

 Now, let $S'$ be a model of $L$ which is proper over $S$ and such that $S'$ dominates $\overline{T}_u$ for any extension $u$ of $v$.
 Then any extension of $v$ on $S'$ admits a \good \ally locally around the center.
 Moreover, we have $\bigl(f/\pi+(\sigma^*(g)^p-\sigma^*(g))\bigr)|_{s(S')}=0$.
 Since the restriction to $s(S')$ is $0$, the \good \ally {\em cannot} be of the form $h/\sigma+\mc{O}_{X,z'}$ where $h$ a {\em non-zero} constant function.
 Thus the lemma follows.
\end{proof}

\subsection{Conclusion of the proof of Theorem \ref{mainthmadm}}
\label{conclpfmain}
\mbox{}\\
Let us conclude the proof.
By \cite[Exp.\ II, 3.1.1]{G}, we may replace $S$ either by its alteration or by an \'{e}tale covering.
Let $\{U_i\}$ be an open covering of $X$. If the theorem is true for each $U_i\rightarrow S$, then it is true also for $X$.
Moreover, let $U\subset X$ be an open subscheme such that $U_s\subset X_s$ is dense for each $s\in S$.
For any base change by a morphism $S'\rightarrow S$, $U\times_S S'\subset X\times_S S'$ has the same property.
Thus, we may replace $X$ by $U$.

By Lemma \ref{geomirrlem}, we may assume that any point $x\in X$ has an open neighborhood which is fiberwise geometrically irreducible over $S$.
Since $X$ is noetherian, we may cover $X$ by finitely many open subschemes which are fiberwise geometrically irreducible.
Thus, we may assume that $S$, $X$ are affine and $X\rightarrow S$ is fiberwise geometrically irreducible.

When $h=0$, the claim is obvious, so we may assume that $h\neq0$.
Since $X$ is affine, we may write $h=f'/g'$ with $f',g'\in\mc{O}_X$.
Use Gruson-Raynaud flattening theorem to generate a modification $S'\rightarrow S$ so that the strict transform of $V(g')$ is flat over $S'$.
Thus, by removing the strict transform from $X\times_S S'$, replacing $S$ by $S'$, and taking an affine covering of $S$ again,
we may assume that $h=f/\pi$ where $f\in\mc{O}_X$ and $\pi\in\mc{O}_S$, or in other words, $h\in\mc{O}_{X_\eta}$.

Now, we can take finite open coverings $\{X_i\}$ and $\{S_i\}$ of $X$ and $S$ and an \'{e}tale morphism
$X_i\rightarrow\mb{A}^d_{S_i,(x_1,\dots,x_d)}$ for each $i$ such that $V(x_1,\dots,x_d)\rightarrow S_i$ is surjective.
Thus, we may take a coordinate $(x_1,\dots,x_d)$ of $X$ over $S$ such that the zero-locus $S':=V(x_1,\dots,x_d)$ surjects to $S$.
Since $x_i$ is a coordinate system, $S'$ is \'{e}tale over $S$, and we further replace $S$ by $S'$.
Then the zero-locus of $(x_1,\dots,x_d)$ is isomorphic to $S'\times_S S'$.
By [EGA IV, 17.9.3], the section $\Delta\colon S'\rightarrow S'\times_S S'$ is a connected component of $S'\times_S S'$.
Thus removing the components other than $\Delta(S')$, we may assume that the zero locus of $(x_1,\dots,x_d)$ is equal to $S'$.
This, finally, implies that it suffices to show the theorem under the hypothesis (*) of \ref{loclhypintr}.

An \ally is said to be a {\em strict \ally} if it is of the form of the conclusion of Lemma \ref{finallemgalclo}.
For $v\in\left<S\right>$, use Lemma \ref{finallemgalclo}, and take an alteration $S'_v\rightarrow S$ admitting strict \ally for any extension $w$ of $v$ to $S_v$.
By using Gruson-Raynaud flattening, we may find a following commutative diagram:
\begin{equation*}
 \xymatrix{
  S_v\ar[r]^-{\mr{modif}}\ar[d]_{\mr{fin.\ flat}}^(.4){g}&S'_v\ar[d]\\
 T\ar[r]^-{\mr{modif}}&S.
  }
\end{equation*}
Let $t\in T$ be the center of $v$.
For any point $t'\in g^{-1}(t)$, there exists a neighborhood $W_{t'}\subset S_v$ over which a strict \ally exists.
Put $Z_{t'}:=S_v\setminus W_{t'}$, $Z:=\bigcap_{t'\in g^{-1}(t)}Z_{t'}$, $U_v:= T\setminus g(Z)$.
Since $t'\not\in Z_{t'}$, we have $t\not\in g(Z)$, and $U_v$ is an neighborhood of $t$.
By construction, $f/\pi$ admits a strict \ally around any valuation on $S_v$ centered on $g^{-1}(U_v)$.
Since $\left<S\right>$ is quasi-compact, there exists a finite subset $v_1,\dots,v_n\in\left<S\right>$
such that $\{\left<U_{v_i}\right>\}_{i=1,\dots,n}$ covers $\left<S\right>$.
Let $T\subset S_{v_1}\times_S\dots\times_S S_{v_n}$ be a closed subscheme maximally dominant over $S$.
This is an alteration.
By assumption, $f/\pi$ is admissible over $T$.
\qed

\section{Recollections and complements of Part I}
\label{sect4}
From now on, we use the language of nearby cycles over general base freely.
This is recalled in [Part I, \S4], and the reader is invited to take a brief look at the section
(except for [Part I, 4.10], which we will not use) before reading this section.
Now, recall from Notations that we are fixing $\Lambda$, which acts as a coefficient of \'{e}tale sheaves.
Note that in [Part I], we assumed that $\Lambda$ is {\em local}, but we do not need this assumption here.
We made such an assumption so that the notion of rank works, which we do not consider here.

\subsection{}
Let us give additional remarks in this paragraph on Artin-Schreier sheaf.
For the conventions concerning Artin-Schreier $\Lambda$-module, see [Part I, 4.11].
In particular, we are implicitly fixing a non-trivial additive character $\psi$.
Let $X$ be a scheme, and let $f,g\in\mc{O}_X$.
The following properties are important:
\begin{enumerate}
 \item multiplicativity: $\mc{L}(f+g)\cong\mc{L}(f)\otimes\mc{L}(g)$;
 \item inversion: $\mc{L}(-f)\cong\mc{L}(f)^{\vee}$;
 \item relation with Frobenius: $\mc{L}(f^p)\cong\mc{L}(f)$.
\end{enumerate}
Since the verification is standard, we leave it to the reader.
Combining these properties, we have $\mc{L}(f^p-f)\cong\Lambda$.

Now, let $h\in\rat{X}$.
Then this can be regarded as a rational morphism $h\colon X\dashrightarrow\mb{A}^1_{\mb{F}_p}$.
The composition $X_{\mr{red}}\hookrightarrow X\dashrightarrow\mb{A}^1_{\mb{F}_p}$ is denoted by $h_{\mr{red}}$.
By [EGA I, 7.2.2], there exists the maximum open dense subscheme $j\colon U\hookrightarrow X_{\mr{red}}$ over which $h_{\mr{red}}$
is an actual morphism $h_{\mr{red},U}\colon U\rightarrow\mb{A}^1_{\mb{F}_p}$.
In this situation, we put
\begin{equation*}
 \mc{L}(h):=j_!\,h_{\mr{red},U}^*(\mc{L}_{\psi}).
\end{equation*}
If $f\colon X'\rightarrow X$ is a maximally dominant morphism, the pullback homomorphism $f^*\colon\rat{X}\rightarrow\rat{X'}$ is defined.
In this situation, we have a canonical homomorphism $f^*\mc{L}(h)\rightarrow\mc{L}(f^*h)$.
Even though this homomorphism is isomorphic on the pullback of the subscheme of definition of $h_{\mr{red}}$,
caution that it may {\em not} be isomorphic because $f^*h$ may be defined on a larger open subscheme.

\subsection{}
\label{Gactcomp}
For a morphism of finite type $X\rightarrow S$ between noetherian separated schemes,
we defined the group of flat \'{e}tale systems $\Comp_d(X/S)$ in [Part I, Definition 4.7].
This is a certain subgroup of $\prod_{\eta\in S^0}\mr{K}_0\mr{Cons}(X_{\overline{\eta}})$, where $S^0$ is the set of generic points of $S$,
and $\mr{K}_0\mr{Cons}(Y)$ is the Grothendieck group of $D_{\mr{ctf}}(Y)$.
If we are given a morphism $g\colon S'\rightarrow S$, it admits a pullback map $g^*\colon\Comp_d(X/S)\rightarrow\Comp_d(X_{S'}/S')$.
Recall that the pullback is given by nearby cycle functor.
We will use a variant of this construction later.
Let $G$ be a finite group.
Let $\mr{Cons}^G(X):=D_{\mr{ctf}}(X,\Lambda[G])$, the category of ctf-$\Lambda$-complexes with $G$-action.
Assume we are given a specialization map $t\rightsquigarrow s$ on $S$.
For $\mc{F}\in\mr{Cons}^G(X_{\overline{t}})$, $\Psi_{\overline{s}\leftarrow\overline{t}}(\mc{F})$ is naturally equipped with $G$-action.
Thus, replacing $\Comp$ by $\Comp\vphantom{E}^G$ in [Part I, Definition 4.7], we may define $\Comp\vphantom{E}_d^G(X/S)$.
Now, consider the left exact functor $\ul{\Gamma}^G:=\sHom_{\Lambda[G]}(\Lambda,-)$.

\begin{lem*}
 \begin{enumerate}
  \item The functor $\mr{R}\ul{\Gamma}^G\colon D^+(X,\Lambda[G])\rightarrow D^+(X,\Lambda)$ preserves ctf-complexes.

  \item The following diagram commutes
	\begin{equation*}
	 \xymatrix@C=40pt{
	  \Comp\vphantom{E}^G(X/S)\ar[r]^-{\mr{R}\ul{\Gamma}^G}\ar[d]_{g^*}&\Comp(X/S)\ar[d]^{g^*}\\
	 \Comp\vphantom{E}^G(X_{S'}/S')\ar[r]^-{\mr{R}\ul{\Gamma}^G}&\Comp(X_{S'}/S').
	  }
	\end{equation*}
 \end{enumerate}
\end{lem*}
\begin{proof}
 Let $x\in X$ be a point.
 For any $\mc{F}\in D^+(X,\Lambda[G])$, we have a canonical map $\mr{R}\ul{\Gamma}^G(\mc{F})_{\overline{x}}\rightarrow\mr{R}\Gamma^G(\mc{F}_{\overline{x}})$.
 This map is an isomorphism.
 Indeed, we may assume that $\mc{F}$ is a sheaf.
 We have
 \begin{equation*}
  \ul{\Gamma}^G(\mc{F})_{\overline{x}}:=
   \sHom_{\Lambda[G]}(\Lambda,\mc{F})_{\overline{x}}
   \cong
   \mr{Hom}_{\Lambda[G]}(\Lambda,\mc{F}_{\overline{x}})
   =:
   \Gamma^G(\mc{F}_{\overline{x}}),
 \end{equation*}
 where the isomorphism follows since $\Lambda$ is a finitely presented $\Lambda[G]$-module.
 Thus, it suffices to show that if $\mc{I}$ is an injective $\Lambda_X[G]$-module, then $\mc{I}_{\overline{x}}$ is an injective $\Lambda[G]$-module.
 We have $\mc{I}_{\overline{x}}\cong\indlim_U\Gamma(U,\mc{I})$, where the inductive limit runs over \'{e}tale neighborhood $U$ of $\overline{x}$.
 The $\Lambda[G]$-module $\Gamma(U,\mc{I})$ is injective since the functor $\Gamma(U,-)$ admits an exact left adjoint.
 Since $\Lambda[G]$ is noetherian, filtered inductive limits of injective $\Lambda[G]$-modules is injective by Baer's criterion, the claim follows.

 Let us show the first claim.
 Since the preservation of constructibility is obvious, we only need to show that it preserves complexes of finite Tor-dimension.
 Let $\mc{F}$ be a constructible complex of Tor-dimension $\leq r$, and $N$ be a $\Lambda$-module (not a sheaf).
 By the observation above, we have
 \begin{equation*}
  \bigl(\mr{R}\ul{\Gamma}^G(\mc{F})\otimes^{\mb{L}}_{\Lambda}N\bigr)_{\overline{x}}
   \cong
   \mr{R}\Gamma^G(\mc{F}_{\overline{x}})\otimes^{\mb{L}}_{\Lambda}N.
 \end{equation*}
 It suffices to show that the cohomology vanishes for degree $<-r$.
 Since $\mc{F}_{\overline{x}}$ is a $\Lambda[G]$-complex of Tor-dimension $\leq r$ whose cohomology is finitely generated,
 there exists a complex $P^{\bullet}$ of finitely generated projective $\Lambda[G]$-modules such that $P^i=0$ for $i<-r$.
 It remains to show that for a finitely generated projective $\Lambda[G]$-module $P$, $\Gamma^GP$ is a projective $\Lambda$-module,
 and $\mr{R}^i\Gamma^G(P)=0$ for $i>0$.
 Indeed, it suffices show this for $P=\Lambda[G]$.
 Since $\Gamma^G(\Lambda[G])=\Lambda$, this is clearly projective.
 Let $Q_{\bullet}\rightarrow\Lambda$ be a resolution by projective $\Lambda[G]$-modules.
 We have
 \begin{equation*}
  \mr{R}\Gamma^G(\Lambda[G])\cong
  \mr{Hom}_{\Lambda[G]}(Q_{\bullet},\Lambda[G])\cong
  \mr{Hom}_{\Lambda}(Q_{\bullet},\Lambda)\cong
  \mr{R}\mr{Hom}_{\Lambda}(\Lambda,\Lambda),
 \end{equation*}
 and the claim follows.

 Let us show the second claim.
 It suffices to show that $\Psi_{s\leftarrow t}$ and $\mr{R}\ul{\Gamma}^G$ commutes for any specialization map $t\rightsquigarrow s$ on $S$.
 We have the canonical map $\Psi_{s\leftarrow t}\circ\mr{R}\ul{\Gamma}^G\rightarrow\mr{R}\ul{\Gamma}^G\circ\Psi_{s\leftarrow t}$.
 To show that this is an isomorphism, it is pointwise, and in view of the commutation of taking stalk and $\mr{R}\ul{\Gamma}^G$ above,
 the verification is reduced to the commutation of $\mr{R}\ul{\Gamma}^G$ and $\mr{R}\Gamma$,
 whose verification is standard (by taking a free resolution of the left $\Lambda[G]$-module $\Lambda$).
\end{proof}

\subsection{}
Let us recall some terminologies from [Part I, 5.6] for the convenience of the reader.
As in [Part I], we put $\Box:=\mb{P}^1_k\setminus\{1\}$ with coordinate $t$.
For a morphism $V\rightarrow\Box$, we put $V_\infty:=V_{\mr{md}/\Box}\times_{\Box}\{\infty\}$,
where $V_{\mr{md}/\Box}$ denotes the maximal closed subscheme of $V$ which is maximally dominant over $\Box$.
Let $S$ be a noetherian separated scheme.
An $\Mdf$-sequence over $S$ (of length $n$) is a set $\{\mbf{V}_i\}_{i=0,\dots,n}$ where
$\mbf{V}_0:=(S\times\Box\rightarrow S\times\Box)$ and $\mbf{V}_i$ is a proper morphism $\mbf{V}_i\rightarrow\mbf{V}_{i-1,\infty}\times\Box$ for $i>0$.
The sequence is said to be {\em principal} if the morphism $V_i\rightarrow\mbf{V}_{i-1,\infty}\times\Box$ is surjective for each $i$.
A morphism of $\Mdf$-sequences (of the same length) $\phi\colon\{\mbf{V}_i\}\rightarrow\{\mbf{W}_i\}$
is a collection of morphisms $\phi_i\colon V_i\rightarrow W_i$ such that the following diagram commutes:
\begin{equation*}
 \xymatrix@C=40pt{
  V_i\ar[r]^-{\phi_i}\ar[d]_{\mbf{V}_i}&W_i\ar[d]^{\mbf{W}_i}\\
 \mbf{V}_{i-1,\infty}\times\Box\ar[r]^-{\phi_{i-1}}&
  \mbf{W}_{i-1,\infty}\times\Box.
  }
\end{equation*}

\subsection{}
Let $f\colon X\rightarrow S$ be a morphism of finite type, and $F$ be an element of $\Comp_d(X/S)$.
Let $\mbf{f}=(f_1,\dots,f_n)\in\mc{O}_X^{\times n}$, a finite ordered sequence of functions of $\mc{O}_X$.
Recall from [Part I, Definition 5.6] that an $\Mdf$-sequence $\{\mbf{V}_i\}$ is said to be {\em adapted to $(F;\mbf{f})$} if the following condition holds:
Put $F[\emptyset]_{\infty}:=F$, and assume $F[\{\mbf{V}_i\}_{i<k}]_{\infty}\in\Comp(X_{\mbf{V}_{n,\infty}}/\mbf{V}_{n,\infty})$ is defined.
Then, $F[\{\mbf{V}_i\}_{i<k}]_{\infty}\otimes\mc{L}(f_kt)$ is a flat \'{e}tale system on $\mbf{V}_{k}$.
We let $F[\{\mbf{V}_i\}_{i\leq k}]_{\infty}$ be the restriction to $\mbf{V}_{k,\infty}$ of this flat \'{e}tale system.

\begin{dfn*}
 Let $\{\mbf{V}_i\}_{i\leq n}$ is an $\Mdf$-sequence adapted to $(F;\mbf{f})$.
 \begin{enumerate}
  \item We say that the {\em boundary dimension} of $(F;\mbf{f})$ with respect to $\{\mbf{V}_i\}$ is $\leq r$ if
	$F[\{\mbf{V}_i\}_{i\leq n}]_{\infty}$ belongs to $\Comp_{r}(X_{\mbf{V}_{n,\infty}}/\mbf{V}_{n,\infty})$.
	The element $F$ is of {\em boundary dimension $\leq r$ with respect to $\mbf{f}$}
	if there exists a principal $\Mdf$-sequence $\{\mbf{V}_i\}_{i\leq n}$
	which is adapted to $(F;\mbf{f})$ and is of boundary dimension $\leq r$.

  \item Assume we are given $n$ functions $\mbf{h}:=\{h_i\}_{i=1,\dots,n}$ (with order) in $\mc{O}_X$ and $\mbf{h}':=\{h'_j\}_{j=1,\dots,m}$.
	We put $\mbf{h}\vee\mbf{h}':=\{h_1,\dots,h_n,h'_1,\dots,h'_m\}$.
	Note that $\mbf{h}\vee\mbf{h}'\neq\mbf{h}'\vee\mbf{h}$ since we are considering the order.
 \end{enumerate}
\end{dfn*}

\begin{rem*}
 If $F$ is of boundary dimension $\leq r$ with respect to $\mbf{f}$ then for any
 $\Mdf$-sequence $\{\mbf{V}_i\}_{i\leq n}$ which is adapted to $(F;\mbf{f})$,
 it is of boundary dimension $\leq r$ by [Part I, Lemma 5.6.2].
\end{rem*}

\section{Main result}
\label{sect5}
\begin{dfn}
 \label{defasmas}
 Let $X\rightarrow S$ be a smooth morphism and $\mc{F}$ be a constructible $\Lambda$-module on $X$.
 The sheaf $\mc{F}$ is said to be an {\em admissible AS sheaf} if the following holds Zariski locally around each point of $S$:
 Assume first that $S$ is irreducible.
 Then $\mc{F}$ is so if there exists a sheaf $\mc{G}$ on $X$, a closed subscheme $Y\subset X$,
 and an admissible function $h\in\rat{X}$ (cf.\ Definition \ref{dfnadmfun}) such that the following holds:
 1.\ for any $s\in S$, $Y_s\subset X_s$ is nowhere dense;
 2.\ $\mc{G}\otimes\mc{L}(h)^{-1}|_{U}$, for a sufficiently small open dense subscheme $U\subset X\setminus Y$,
 extends to a locally constant constructible sheaf (possibly zero) on $X\setminus Y$;
 3.\ $\mc{F}_{\overline{\eta}}\cong\mc{G}_{\overline{\eta}}$, where $\eta$ is the generic point of $S$.
 If $S$ is not irreducible, $\mc{F}$ is admissible AS sheaf if it is so over each irreducible component of $S$.

 For $F\in\Comp(X/S)$, we say that $F$ is {\em of admissible AS type}, if for each irreducible component $S_i$,
 there exists an admissible AS sheaf $\mc{F}_i$ on $X_{S_i}\rightarrow S_i$ such that $[\mc{F}_i]=F|_{S_i}$.
\end{dfn}

\subsection{}
\label{altadmassh}
The main result of the previous section (Theorem \ref{mainthmadm}) can
be interpreted as follows in terms of admissible AS sheaves:

\begin{cor*}
 Let $X\rightarrow S$ be a smooth morphism of finite type, and let $f\in\rat{X}$.
 Then there exists an alteration $S'\rightarrow S$ such that the pullback of $\mc{L}(f)$ to $X_{S'}$ is an admissible AS sheaf.
\end{cor*}

\begin{lem}
 \label{AsneacyAS}
 Let $X\rightarrow S$ be a smooth morphism and $F\in\Comp(X/S)$ which is of admissible AS type.
 For any morphism $g\colon S'\rightarrow S$, the pullback $g^*(F)\in\Comp(X_{S'}/S')$ remains to be of admissible AS type.
\end{lem}
\begin{proof}
 Since the claim is local, we may assume that $F=\left[\mc{H}\otimes\mc{L}(f/\sigma)\right]$
 where $\mc{H}$ is a locally constant constructible sheaf, $f\in\mc{O}_X$ is an $S$-separable function, and $\sigma\in\mc{O}_S$.
 We may also assume that $S'$ is irreducible with generic point $\eta'$, and let $g'\colon X_{S'}\rightarrow S'$ be the base change.
 By \cite[2.4.4]{KL}, $\mc{H}\otimes\mc{L}(f/\sigma)$ is universally locally acyclic.
 This implies that $g^*(F)=\left[g'^*(\mc{H}\otimes\mc{L}(f/\sigma))\right]$.
 This implies that if $g(S')\subset V(\sigma)$, then $g^*(F)=0$ since $g^*\mc{L}(f/\sigma)=0$, and in particular, it is of admissible AS type.
 Otherwise, $g'^*(f/\sigma)=g'^*(f)/g^*(\sigma)$ is well-defined in $\rat{S'}$, and we have $g^*(F)=\left[g'^*(\mc{H})\otimes\mc{L}(g^*(f/\sigma))\right]$,
 which is also of admissible AS type.
\end{proof}

\begin{lem}
 Consider the following diagram of noetherian schemes:
 \begin{equation*}
  \xymatrix{
   X\ar@<0.5ex>[r]^-{x}\ar@<-0.5ex>[r]_-{y}
   \ar@/_10pt/[rd]&
   \mb{A}^1_{z,S}\ar[d]\ar[r]\ar@{}[rd]|\square&
   \mb{A}^1_{z,\mb{A}^2}\ar[d]\\
  &S\ar[r]&\mb{A}^2_{(s,t)}.
   }
 \end{equation*}
 Here, $x$ is smooth and the function corresponding to $y$ is $S$-Frobenius separable.
 Abusing notations, we denote $x^*(z)$, $y^*(z)$ by $x$, $y$.
 Put $\mc{L}_n:=\mc{L}_!(x/s+y^n/t)$ for an integer $n>0$, where $\mc{L}_!$ means to take the zero extension outside of $s,t\neq0$.
 Then there exists an integer $N$ such that for any $n>N$ and $p\nmid n$,
 and any morphism $T\rightarrow S$ from a henselian trait $T$ whose closed point is sent to $V(t)\subset S$,
 we have $\dim\mr{Supp}(\Psi_{T}(\mc{L}_{n,T}))<\dim(X/S)$.
\end{lem}
\begin{proof}
 We may assume $X$ connected.
 First, we may assume that $y$ is smooth. Indeed, by assumption, there exists a separable function $y'$
 such that $y'^q=y$ for a number $q$ which is some power of $p$.
 We take the universally homeomorphism $\varphi\colon\mb{A}^2_{s,t'}\rightarrow\mb{A}^2_{s,t}$ sending $t$ to $t'^q$.
 Then we have $\varphi^*\mc{L}(y^n/t)=\mc{L}(y'^n/t')$, and it suffices to check the claim for $\mc{L}(y'/t')$.
 Shrink $X$ by fiberwise dense open subscheme, we may assume $y$ smooth.
 
 Given any finite dominant morphism $S'\rightarrow S$, we may replace $S$ by $S'$ because any specialization map in $S$ can be lifted to $S'$.
 Thus, by Lemma \ref{geomirrlem}, we may assume that any point $x\in X$ has an open neighborhood $U\subset X$ which is fiberwise geometrically irreducible over $S$.
 Now, the claim is Zariski local with respect to $S$ and $X$.
 Thus, we may assume that $X\rightarrow S$ is surjective and fiberwise geometrically irreducible,
 and that $S$ is affine and $\Omega^1_{X/S}$ is a free $\mc{O}_X$-module.
 Take a system of local coordinates $x,x_2,\dots,x_d$ of $X$ over $S$.
 Write $\mr{d}y=f_1\mr{d}x+f_2\mr{d}x_2+\dots+f_d\mr{d}x_d$.
 Let $S\supset D_i:=\bigl\{s\in S\mid f_i|_{X_s}=0 \mbox{ on $X_s$}\bigr\}$.
 Since $X\rightarrow S$ is equidimensional and the fibers are irreducible, $D_i$ is the set of $s\in S$ such that $\dim\bigl(f_i^{-1}(0)\cap X_s\bigr)=d$,
 thus $D_i$ is constructible by applying Chevalley's theorems [EGA IV, 13.1.3, 1.8.4] to $f_i^{-1}(0)\rightarrow S$.
 Since $y$ is smooth over $S$, we get $\bigcap D_i=\emptyset$.
 There exists $N$ such that for any $n\geq N$, $(y^{n} f_1)|_{X_s}$ is non-constant for any $s\in(S\setminus D_1)$
 by writing the constructible set $S\setminus D_1$ as a finite union of locally closed subschemes and by applying Lemma \ref{nowhconslem} to each of them.
 For any morphism $S'\rightarrow S$ from an integral scheme $S'$, constants $a,b\in\mc{O}_{S'}^{\times}$, and $n>N$ such that $p\nmid n$,
 the function $ax+by^n$ on $X\times_SS'$ is fiberwise generically smooth over $S'$ because
 \begin{equation*}
  \mr{d}(ax+by^n)=a\cdot \mr{d}x+b\cdot \mr{d}y^n
   =
   (a+nby^{n-1}f_1)\cdot \mr{d}x+\sum_{i\geq2}nbf_iy^{n-1}\cdot \mr{d}x_i
 \end{equation*}
 in $\Omega^1_{X_{s'}/s'}$ is not $0$ on $X_{s'}$ for any $s'\in S'$.
 
 To conclude the proof, it suffices to show that for any morphism $g\colon T:=\mr{Spec}(R)\rightarrow S$ from a henselian trait,
 the pullback $\mc{L}_{n,T}$ is universally locally acyclic for any $n>N$ such that $p\nmid n$.
 Let $u$ be a uniformizer of $R$.
 Let $(s_0,t_0)\in\mb{A}^2_{(s,t)}$ be the image of the closed point of $T$ via $T\rightarrow S\rightarrow\mb{A}^2_{(s,t)}$. If $s_0,t_0\neq0$,
 then $\mc{L}_n$ is locally constant constructible, so it is universally locally acyclic.
 If $s_0=0$ and $t_0\neq0$, then it is universally locally acyclic by \cite[2.4.4]{KL}, and the same for the case where $s_0\neq0$ and $t_0=0$.
 The remaining case is when $s_0=t_0=0$.
 Write $g^*(t)=t'u^\alpha$, $g^*(s)=s'u^{\beta}$ where $t',s'\in R^{\times}$ and $u$ is a uniformizer of $R$. Since $s_0=t_0=0$, we have $\alpha,\beta>0$.
 We get
 \begin{equation}
  \label{vancompsum}\tag{$\star$}
  \frac{x}{s}+\frac{y^n}{t}=
   \frac{t'u^\alpha x+s'u^\beta y^n}{t's'u^{\alpha+\beta}}
 \end{equation}
 Now, assume $\gamma:=\alpha-\beta>0$. Then (\ref{vancompsum}) is equal to $(t'u^\gamma x+s'y^n)/(t's'u^\alpha)$.
 Now, since $y^n$ is smooth if $p\nmid n$ and $y\neq0$ on the special fiber of $X_T$, we get that $t'u^\gamma x+s'y^n$ is smooth as well,
 thus by \cite[2.4.4]{KL}, the dimension of $\Psi_T(\mc{L}_{n,T})$ is $<\dim(X/S)$.
 Likewise, the claim also holds when $\alpha-\beta<0$.
 Finally, consider the case where $\alpha=\beta$.
 Then (\ref{vancompsum}) is equal to $(t'x+s'y^n)/(t's'u^\alpha)$.
 By the choice of $N$, $t'x+s'y^n$ is separable, so the claim follows.
\end{proof}

\begin{cor}
 \label{puinsreddim}
 Let $X\rightarrow S$ be a smooth morphism equidimensional of dimension $d>0$ between $k$-schemes of finite type,
 and $F\in\Comp(X/S)$ be of admissible AS type.
 Let $g\colon X\rightarrow\mb{A}^1_S$ be an {\em $S$-Frobenius separable} function.
 Then there exists an integer $n$ such that $(F,\{g^n\})$ is of boundary dimension $<d$.
\end{cor}
\begin{proof}
 We show that there exists $N>0$ such that for any $n>N$ and $p\nmid n$, $(F,\{g^n\})$ is of boundary dimension $<d$.
 Since $S$ is quasi-compact, by enlarging $N$, the claim is local, and we may assume that $F=[\mc{F}]$ where
 $\mc{F}\cong\left[\mc{H}\otimes\mc{L}(f/\sigma)\right]$ where $\mc{H}$ is a locally constant constructible sheaf,
 $f\in\mc{O}_X$ is an $S$-separable function, and $\sigma\in\mc{O}_S$.
 Let $t_{\Box}$ denote the coordinate of $\Box$.
 We define $S\times\Box\rightarrow\mb{A}^2_{(s,t)}$ to be the morphism sending $t$ to $t_{\Box}^{-1}\in\mc{O}_{\Box}$ and $s$ to $\sigma\in\mc{O}_S$.
 The functions $f$, $g$ induces the morphisms $x,y\colon X\times\Box\rightarrow\mb{A}^1_{S\times\Box}$.
 Invoking the lemma, we may find $N$ so that for any $n>N$, $p\nmid n$, and any morphism $T\rightarrow S\times\Box$
 where $T$ is a strict trait and the closed point is sent to a point in $S\times\{\infty\}$,
 the dimension of the support of $\Psi_T\bigl((\mc{H}\otimes\mc{L}(f/\sigma+g^nt_{\Box}))_T\bigr)$ is $<d$.
 Let $Z\rightarrow S\times\Box$ be a modification over which
 $\mc{F}\otimes\mc{L}(g^nt_{\Box})\cong\mc{H}\otimes\mc{L}(f/\sigma+g^nt_{\Box})$ is very good.
 For any point $z\in Z$ over $\infty\in\Box$, we may take a dominant morphism $T\rightarrow Z$
 such that $T$ is a strict trait and the closed point of $T$ is sent to $z$.
 Thus, by the goodness over $Z$, $(F,\{g^n\})$ is of boundary dimension $<d$.
\end{proof}

\begin{prop}
 \label{dimredas}
 Let $f\colon X\rightarrow S$ be a smooth morphism equidimensional of dimension $d>0$ of schemes of finite type over $k$,
 $\mf{h}$ be a non-constant family with respect to $f$ {\normalfont(}cf.\ {\normalfont\ref{nonconsfam}}{\normalfont)},
 and $\mc{F}$ be an object of $D_{\mr{ctf}}(X)$ whose cohomology sheaves are locally constant constructible.
 Then there exists $\mbf{h}\subset\mf{h}$ such that $([\mc{F}];\mbf{h})$ is of boundary dimension $<d$.
\end{prop}
\begin{proof}
 Put $F:=[\mc{F}]$.
 For a closed subscheme $Z\subset S$, we consider the following condition:
 \begin{quote}
  $(\mbox{Red})_Z$:
  there exists $\mbf{h}\subset\mf{h}$ such that $(F;\mbf{h})$ viewed as an object on
  $(X\setminus f^{-1}(Z))\rightarrow (S\setminus Z)$ is of boundary dimension $<d$.
 \end{quote}
 We show that if $(\mbox{Red})_Z$ holds then there exists a closed subset $Z'\subsetneq Z$ such that $(\mbox{Red})_{Z'}$ holds,
 and conclude by noetherian induction. Assume $(\mbox{Red})_Z$ is true.
 Let $\xi\in Z$ be a generic point, and we take $h\in\mf{h}$ such that $h|_{X_\xi}$ is non-constant.
 Then there exists a closed subscheme $Z'\subset Z$ such that $\xi\in Z\setminus Z'$ and an iterated Frobenius
 morphism $g\colon S'\rightarrow S$ so that the pullback $g^*(h|_{X_{Z\setminus Z'}})$ is Frobenius separable
 by Lemma \ref{lemonprofun}.\ref{lemonprofun-smgenpinsep}.
 Let us show that $(\mbox{Red})_{Z'}$ is satisfied.

 We first show that there exists a principal $\Mdf$-sequence $\{\mbf{W}_i\}_{i\leq n}$ adapted to $(F;\mbf{h})$
 so that $F[\{\mbf{W}_i\}]_{\infty}$ is of admissible AS type on $\mbf{W}_{n,\infty}$.
 We show this by induction on $n$.
 When $n=0$, the claim follows since $\mc{F}$ is assumed to be locally constant constructible.
 Assume that $F[\{\mbf{W}_i\}_{i<n}]_{\infty}$ is of admissible AS type.
 By Corollary \ref{altadmassh}, we may take an alteration $\mbf{W}_n\rightarrow\mbf{W}_{n-1,\infty}\times\Box$ so that
 $F[\{\mbf{W}_i\}_{i<n}]_{\infty}\otimes\mc{L}(h_n)$ is an flat \'{e}tale system of admissible AS type on $\mbf{W}_n$.
 By Lemma \ref{AsneacyAS}, $F[\{\mbf{W}_i\}_{i\leq n}]_{\infty}$ is of admissible AS type on $\mbf{W}_{n,\infty}$ as claimed.

 Now, we take a sequence $\mbf{h}\subset\mf{h}$ of length $n$ so that $(\mbox{Red})_Z$ is satisfied.
 We take a principal $\Mdf$-sequence adapted to $(\mc{F};\mbf{h})$ such that $G:=F[\{\mbf{W}_i\}_{i\leq n}]_{\infty}$ is of admissible AS type.
 Take an alteration $T\rightarrow\mbf{W}_{n,\infty}$ such that the morphism $\rho\colon T\rightarrow S$ factors through $S'$.
 Put $\widetilde{Z}:=\rho^{-1}(Z)$.
 Let $\{\widetilde{Z}_i\}_{i\in I}$ be the set of irreducible components of $\widetilde{Z}$, and put $G_i:=G|_{\widetilde{Z}_i}$.
 Since $h|_{X_{Z\setminus Z'}}$ is Frobenius separable after pulling back to $T$,
 Corollary \ref{puinsreddim} implies that there exists an integer $m$ such that
 $(G_i,\{h^m\})$ is of boundary dimension $<d$ on $\widetilde{Z}_i\setminus\rho^{-1}(Z')$ for any $i\in I$.
 Thus, $F$ is of boundary dimension $<d$ with respect to $\mbf{h}\vee\{h^m\}$ on $S\setminus Z'$, and $(\mbox{Red})_{Z'}$ is satisfied.
\end{proof}

\begin{lem}
 \label{tworedlem}
 Let $S$ be an irreducible $k$-scheme with generic point $\eta$.
 Let $f\colon Y\rightarrow X$ be a morphism of flat $S$-schemes of finite type equidimensional of dimension $d$,
 and let $F\in\Comp(X/S)$.
 Assume one of the following holds:
 \begin{enumerate}
  \item\label{tworedlem-1}
       $f$ is proper and $f_\eta$ is a Galois covering;

  \item $f$ is a modification.
 \end{enumerate}
 If $(f^*F,f^*\mbf{h})$ is of boundary dimension $<d$ for $\mbf{h}\subset\mc{O}_X$,
 then so is $(F,\mbf{h})$.
\end{lem}
\begin{proof}
 Let $W\rightarrow S$ be a morphism, and assume we are given $E\in\Comp(Y_W/W)$ such that $f_{W*}(E)$
 is in $\Comp\vphantom{E}^G(X_W/W)$ for some finite group $G$ (see \ref{Gactcomp} for the notation).
 Let $h\in\mc{O}_{X_W}$, and let $W'\rightarrow W\times\Box$ be a maximally dominant surjection
 over which $E\otimes\mc{L}((f^*h)t)$ is a flat \'{e}tale system.
 Then $f_{W*}(E\otimes\mc{L}(f^*ht))$ can be viewed naturally as an element of $\Comp\vphantom{E}^G(X_{W'}/W')$.
 Let $i\colon W'_\infty:=W'\times_{\Box}\{\infty\}\rightarrow W'$.
 By Lemma \ref{Gactcomp} and since $f$ is proper, we have
 \begin{equation}
  \label{tworedlem-eq1}
   \tag{$\star$}
  \mr{R}\ul{\Gamma}^G\mr{R}f_*i^*(E\otimes\mc{L}(f^*ht))
   =
   i^*\bigl(\mr{R}\ul{\Gamma}^G\mr{R}f_*(E)\otimes\mc{L}(ht)\bigr).
 \end{equation}
 Now, let $G$ be the Galois group of $Y_\eta\rightarrow X_\eta$.
 Then we have $\mr{R}\ul{\Gamma}^G(f_*f^*F)=F$.
 Let $\{\mbf{W}_i\}_{i\leq n}$ be a principal $\Mdf$-sequence adapted to both $(F,\mbf{h})$ and $(f^*F,f^*\mbf{h})$.
 Using (\ref{tworedlem-eq1}) iteratively, we have an equality
 $F[\{\mbf{W}_i\}]_{\infty}=\mr{R}\ul{\Gamma}^G\mr{R}f_*\bigl((f^*F)[\{\mbf{W}_i\}]_{\infty}\bigr)$.
 This implies that $\mr{Supp}\bigl(F[\{\mbf{W}_i\}]_{\infty}\bigr)\subset f\bigl(\mr{Supp}((f^*F)[\{\mbf{W}_i\}]_{\infty})\bigr)$,
 and the claim follows.

 Let us treat the second case.
 We may assume that $F=[\mc{F}]$ for some object $\mc{F}$ in $D_{\mr{ctf}}(X_{\overline{\eta}})$.
 By replacing $X$ by $X_{\mr{red}}$ and $Y$ by $Y\times_{X}X_{\mr{red}}$, we may assume that $X$ is reduced.
 Since $f$ is a modification, there exists a dense open subscheme $U_\eta\subset X_\eta$ such that
 $U_\eta\times_XY\rightarrow U_\eta$ is an isomorphism.
 Let $Z_{\eta}$ be the complement.
 Using Gruson-Raynaud flattening, we take a modification $S'\rightarrow S$ so that the closure of $Z_\eta$ in $X_{S'}$ is flat.
 Then we may remove the closure from $X_{S'}$, and assume that $Y_\eta\rightarrow X_\eta$ is an isomorphism.
 Now, we apply \ref{tworedlem-1} to conclude.
\end{proof}

\subsection{}
Let $X$ be a noetherian scheme.
For $\mc{F}\in D_{\mr{ctf}}(X)$, the support is defined to be $\mr{Supp}(\mc{F}):=\{x\in X\mid\mc{F}_{\overline{x}}\neq0\}^{-}$.
 and $F\in\mr{K}_0\mr{Cons}(X)$.
We define the support of $F$ to be $\mr{Supp}(F):=\bigcap_{F=[\mc{F}]}\mr{Supp}(\mc{F})$ where $\mc{F}\in D_{\mr{ctf}}(X)$.
In fact, there exists $\mc{F}\in D_{\mr{ctf}}(X)$ so that $\mr{Supp}(\mc{F})=\mr{Supp}(F)$.
Indeed, take any $\mc{G}\in D_{\mr{ctf}}(X)$ such that $[\mc{G}]=F$.
There exists a stratification $\{j_i\colon Z_i\hookrightarrow X\}_{i\in I}$ of $X$ such that $j_i^*\mc{G}$ has locally constant cohomology sheaves.
Let $I'$ be the set of $i\in I$ such that $j_i^*(F)\neq0$ in $\mr{K_0}\mr{Cons}(Z_i)$.
Then by putting $\mc{F}=\bigoplus_{i\in I'}j_{i!}j_i^*\mc{G}$, we have $\mr{Supp}(\mc{F})=\bigcup_{i\in I}\overline{Z}_i=\mr{Supp}(F)$.
Finally, let $X\rightarrow S$ be a morphism between separated $k$-schemes of finite type.
For $F\in\Comp(X/S)$ is in particular an element $\{F_{\eta}\}$ of $\prod_{\eta\in S^0}\mr{K}_0\mr{Cons}(X_{\overline{\eta}})$.
The support of $F$ is defined to be $\bigcup(\pi_{\eta}\mr{Supp}(F_\eta))^{-}$,
where $\pi_\eta\colon X_{\overline{\eta}}\rightarrow X_\eta$ and the closure is taken in $X$.

\begin{thm*}
 Let $X\rightarrow S$ be a morphism between separated $k$-schemes of finite type where $X$ is affine.
 Let $F\in\Comp(X/S)$ such that $\dim(\mr{Supp}(F)_s)\leq d$ for any $s\in S$,
 and let $\mf{h}\subset\mc{O}_X$ be a non-constant family of $\mr{Supp}(F)$ over $S$.
 Then there exists $\mbf{h}\subset\mf{h}$ such that $(F;\mbf{h})$ is of boundary dimension $<d$.
\end{thm*}
\begin{proof}
 We may assume $k$ is perfect.
 We may freely replace $X$ by its underlying reduced subscheme.
 Since $\Psi$ commutes with pushforward by closed immersion, we may replace $X$ by $\mr{Supp}(F)$, and assume that $\mr{Supp}(F)=X$.
 Let $g\colon S'\rightarrow S$ be a maximally dominant proper morphism.
 Then by the assumption on the dimension, $\mf{h}\times_S S'$ is a non-constant family on $\mr{Supp}(g^*F)$.
 Thus, we may replace $S$ by $S'$ and $\mf{h}$ by its pullback.
 This allows us to assume $S$ to be integral.
 Let $\eta$ be the generic point of $S$.
 By replacing $S$ by its finite dominant morphism, we may assume that $F=[\mc{F}]$ and $\mr{Supp}(\mc{F})=\mr{Supp}(F)$
 for some object $\mc{F}$ in $D_{\mr{ctf}}(X_{\eta})$ (not just $D_{\mr{ctf}}(X_{\overline{\eta}})$!).
 Further replacing $S$ and replacing $X$ by its underlying reduced subscheme, we may assume that the generic fiber $X_\eta$ is geometrically reduced.
 Since $\mc{F}$ is constructible, there exists a closed subscheme $Z_\eta\subset X_\eta$ such that
 any cohomology sheaf of $\mc{F}|_{X_\eta\setminus Z_\eta}$ is locally constant constructible $\Lambda$-module,
 $\dim(Z_\eta)<d$, and $X_{\eta}\setminus Z_{\eta}$ is equidimensional of dimension $d$.
 By Gruson-Raynaud flattening, there exists a modification $S'\rightarrow S$ such that the strict transforms of $X$ and
 $\overline{Z}_{\eta}$ in $X_{S'}$ are flat.
 Thus, we may assume that $X$ and $\overline{Z}_\eta$ are flat.
 Since $\Psi$ commutes with pullback by open immersion, the claim is generic around each fiber.
 This allows us to remove $\overline{Z}_\eta$ and assume that $\mc{F}$ is locally constant constructible on $X_\eta$ and
 $X\rightarrow S$ is flat of relative dimension $d$.
 Let $Y\rightarrow X$ be a finite morphism which trivializes $\mc{F}$ and is a Galois covering over $\eta$.
 We further use Gruson-Raynaud flattening for $Y\rightarrow S$, and we may assume that $Y$ is $S$-flat.
 By the reduced fiber theorem \cite{BLR}, we can take the following diagram:
 \begin{equation*}
  \xymatrix@C=30pt@R=10pt{
   Y\ar[d]&Y_{S'}\ar[l]\ar[dd]\ar@{}[ldd]|\square&
   Y'\ar[l]_-{\substack{\mr{fin}\\\mr{modif}}}^-{\alpha}\ar@/^10pt/[ldd]\\
  X\ar[d]&&\\
   S&S'\ar@{->>}[l]_-{\mr{fin\ \acute{e}t}}&}
 \end{equation*}
 such that $Y'\rightarrow S'$ is flat and has reduced fibers.
 It suffices to show the claim for $X_{S'}$ and the pullback of $\mc{F}$ and $\mf{h}$.
 Note that the pullback of $\mc{F}$ to $Y_{S'}$ is trivial.
 For any point $s'\in S'$, the schemes $Y_{s'}$, $X_{s'}$, $Y'_{s'}$ are equidimensional of dimension $d$.
 Let
 \begin{equation*}
  Z:=
   \bigl\{y\in Y_{S'}\mid Y'_y\cap\mr{Sing}(Y'/S')\neq\emptyset\bigr\}
   =\alpha\bigl(\mr{Sing}(Y'/S')\bigr).
 \end{equation*}
 The openness of smooth locus implies that $Z$ is a closed subset.
 For each point $s'\in S'$, we have $\mr{Sing}(Y'/S')\cap Y'_{s'}=\mr{Sing}(Y'_{s'}/s')$ (cf.\ [EGA IV, 17.5.1]).
 Thus, if $s'$ is a closed point, since $Y'_{s'}$ is reduced and $k$ is assumed perfect,
 $\mr{Sing}(Y'_{s'}/s')\subset Y'_{s'}$ is nowhere dense by [EGA IV, 17.15.2].
 Thus, the dimension of each fiber of $Z\rightarrow S'$ is $<d$ by [EGA IV, 13.1.5].
 This allows us to remove the image of $Z$ in $X_{S'}$, and assume that $Y'\rightarrow Y_{S'}$ is a finite modification and $Y'\rightarrow S'$ is smooth.
 Since $X$, $Y$ are flat over $S$, the pullback of $\mf{h}$ to $Y$ is a non-constant family over $S$ by
 Lemma \ref{nonconst-prop}.\ref{nonconst-prop-fin}.
 Thus, the pullback to $Y_{S'}$ is a non-constant family over $S'$, and by the flatness, its pullback $\mf{h}'$
 to $Y'$ is a non-constant family over $S'$, again by Lemma \ref{nonconst-prop}.\ref{nonconst-prop-fin}.
 Apply Proposition \ref{dimredas} and Lemma \ref{tworedlem} to conclude the proof. 
\end{proof}

\begin{cor}
 \label{mainresult}
 Assume that $X$ and $S$ are noetherian affine $\mb{F}_p$-schemes {\normalfont(}not necessarily of finite type{\normalfont)},
 and assume we are given a morphism $X\rightarrow S$ of finite type.
 For any $F\in\Comp_d(X/S)$ and $\mbf{h}\subset\mc{O}_{X\times\Delta^m}(X\times\Delta^m)$,
 we can find $\mbf{h}'\subset\mc{O}_{X}(X)$ such that for a principal $\Mdf$-sequence $\{\mbf{V}_i\}$ adapted to $(F;\mbf{h}\vee\mbf{h}')$,
 $F[\{\mbf{V}_i\}]_{\infty}$ is in $\Comp_{d-1}$.
\end{cor}
\begin{proof}
 By [EGA IV, 8.9.1], we may find a $\mb{F}_p$-scheme of finite type $S_0$, a morphism of finite type $X_0\rightarrow S_0$,
 functions $\mbf{h}_0\subset\mc{O}_{X_0\times\Delta^m}(X_0\times\Delta^m)$, and a dominant morphism $S\rightarrow S_0$
 such that $X\cong X_0\times_{S_0}S$ and $\mbf{h}$ is the pullback of $\mbf{h}_0$.
 Furthermore, by [SGA 4$\frac{1}{2}$, Rapport, 4.6 (i)] and [SGA 4, Exp.\ IX, 2.7.4],
 we may assume that $F$ is a pullback of an element of $\Comp_d(X_{S_0}/S_0)$.
 Thus, we may assume that $S=S_0$, and in particular, we may assume that $S$ is finite type over a field $k$ of characteristic $p$. 

 Let $S'\rightarrow S$ be a morphism, and $Z\subset X_{S'}$ be any closed subscheme.
 By Lemma \ref{nonconst-prop}.\ref{nonconst-prop-bc}, $\mc{O}_X$ is a non-constant family of $Z$.
 Take a principal $\Mdf$-sequence $\{\mbf{W}_i\}_{i\leq n}$ for $(F;\mbf{h})$.
 By Gruson-Raynaud flattening theorem, there exists a modification $g^*\colon W\rightarrow\mbf{W}_{n,\infty}$ such that
 $\mr{Supp}(g^*F[\{\mbf{W}_i\}]_{\infty})$ is fiberwise of dimension $\leq d$ over $W$.
 By the observation above, $\mc{O}_X(X)\subset\mc{O}_{X_W}(X_W)$ is a non-constant family of
 $\mr{Supp}\bigl(g^*F[\{\mbf{W}_i\}]_{\infty}\bigr)\subset X_W$.
 Invoking the theorem above, there exists $\mbf{h}'\subset\mc{O}_X(X)$ such that $(g^*F[\{\mbf{W}_i\}]_{\infty};\mbf{h}')$ is of boundary dimension $<d$.
 This is the same as saying that $(F;\mbf{h}\vee\mbf{h}')$ is of boundary dimension $<d$, and the corollary follows.
\end{proof}

\noindent
Tomoyuki Abe:\\
Kavli Institute for the Physics and Mathematics of the Universe (WPI)\\
The University of Tokyo\\
5-1-5 Kashiwanoha,  
Kashiwa, Chiba, 277-8583, Japan\\
e-mail: {\tt tomoyuki.abe@ipmu.jp}

\end{document}